\documentclass[a4paper,11pt]{amsart}

\usepackage{amssymb} 
\usepackage{xcolor} 
\usepackage{fancybox}
\usepackage{dashbox}
\newboolean{showcomments}
\setboolean{showcomments}{true}
\ifthenelse{\boolean{showcomments}}
{ \newcommand{\mynote}[3]{
  \fbox{\bfseries\sffamily\scriptsize#1}
  {\small$\blacktriangleright$\textsf{\emph{\color{#3}{#2}}}$\blacktriangleleft$}}}
{ \newcommand{\mynote}[3]{}}
\definecolor{asparagus}{rgb}{0.53, 0.66, 0.42}

\usepackage{mathrsfs}

  \usepackage{float}

  \usepackage{pgffor}
  \usepackage{tikz}
  \usetikzlibrary{calc,3d,arrows,positioning,shapes.misc}

  \usepackage[english]{babel}
  \usepackage{amsthm}
  \usepackage{amsmath}
  \usepackage{amssymb}
  \usepackage[latin1]{inputenc}
  \usepackage{setspace}
  \usepackage{enumerate}
  \usepackage{stmaryrd}

  \usepackage[colorlinks=true,linkcolor=black,citecolor=black,filecolor=black,urlcolor=black,menucolor=black]{hyperref}

  \usepackage{array}
  \usepackage{graphicx}
  \usepackage[babel]{csquotes}
  \usepackage{geometry}
  \usepackage{bbm}
  \usepackage{stmaryrd}
  \usepackage[all]{xy}
  \usepackage{mathrsfs}

  \theoremstyle{plain}
  \newtheorem{theorem}{Theorem}
  \newtheorem{proposition}{Proposition}

  \newtheorem{lemma}{Lemma}
  
  \theoremstyle{definition}
  \newtheorem{definition}{Definition}

  \newtheorem{rem}{Remark}

  \newcommand{\step}[1]{\par\medskip\par\noindent\textit{#1}} 
	\allowdisplaybreaks[3]

  \def \H{\mathcal{H}}
  \def \L{\mathcal{L}}
  \def \N{\mathbb{N}}
  \def \R{\mathbb{R}}
  
	\def \S{\mathbb{S}}
	\def \Tstrong{T_{\mathrm{strong}}}
		\def \Tcal{T_{\mathrm{cal}}}

  \newcommand {\supp} {\mathop \textup{supp}}

 \newcommand {\dist} {\mathop \textup{dist}}
 
  \newcommand {\p} {\partial}
  \newcommand {\eps} {\varepsilon}

  \title[De~Giorgi type varifold solutions for mean curvature flow]
	{A new varifold solution concept for mean curvature flow: Convergence of the Allen--Cahn equation and weak-strong uniqueness}
\author{Sebastian Hensel}
\address{Institute of Science and Technology Austria (IST Austria), Am~Campus~1, 
3400 Klosterneuburg, Austria}
\email{sebastian.hensel@ist.ac.at}
\curraddr{Hausdorff Center for Mathematics, Universit{\"a}t Bonn, Endenicher Allee 62, 53115 Bonn, Germany
(\texttt{sebastian.hensel@hcm.uni-bonn.de})}

\author{Tim Laux}
\address{Hausdorff Center for Mathematics, Universit{\"a}t Bonn, Endenicher Allee 62, 53115 Bonn, Germany}
\email{tim.laux@hcm.uni-bonn.de}
    \date{\today}

\begin{document}

    \begin{abstract}
	We propose a new weak solution concept for (two-phase) mean curvature flow which enjoys 
	both (unconditional) existence and (weak-strong) uniqueness properties.
    These solutions are evolving varifolds, just as in Brakke's formulation, 
		but are coupled to the phase volumes by a simple transport equation. 
    First, we show that, in the exact same setup as in Ilmanen's 
		proof [J.~Differential~Geom.~38, 417--461, (1993)], any limit point of solutions 
		to the Allen--Cahn equation is a varifold solution in our sense.
    Second, we prove that any calibrated flow in the sense of 
		Fischer et al.~[arXiv:2003.05478]---and hence any classical solution 
		to mean curvature flow---is unique in the class of our new varifold solutions. This is in sharp contrast to the case of Brakke flows, which a priori may disappear 
		at any given time and are therefore fatally non-unique. Finally, we propose
		an extension of the solution concept to the multi-phase case
		which is at least guaranteed to satisfy a weak-strong uniqueness principle.

	\medskip
    

  \medskip

  \end{abstract}
\maketitle


\section{Introduction}

	Weak solution concepts for mean curvature flow have been investigated since the seminal 
	work of Brakke~\cite{brakke}. Remarkably, Brakke constructed his weak solution more than 
	a decade before Gage and Hamilton published their treatise of 
	closed convex curves~\cite{gagehamilton}. His solution concept is based on the theory 
	of varifolds, a measure-theoretic generalization of embedded surfaces also used in this work. 
	It has been applied in a broad context as it enjoys excellent compactness properties, and---thanks to Huisken's monotonicity formula~\cite{huisken}---a partial regularity theory. 
	However, it is well-known that a priori Brakke's formulation lacks any uniqueness or continuity in time: at any given time the solution (or parts of it) might instantly disappear.
	More recently, Kim and Tonegawa~\cite{KimTonegawa} showed that a variant of Brakke's original construction not only yields a Brakke flow, but that in addition the enclosed volume changes continuously, so that at least the sudden disappearance of the solution can be ruled out for the particular Brakke flow constructed there.	
	
	In the present work, we introduce a new notion of weak solutions to mean curvature flow, which is also based on varifolds, enjoys basically the same compactness and existence properties, but in addition does not allow the un-physical  time-discontinuities and non-uniqueness of Brakke's formulation.
	In some sense, our weak formulation lives between the concepts 
	of Brakke's solution and the distributional (or BV) solution 
	introduced by Luckhaus and Sturzenhecker~\cite{LucStu}. The latter also enjoys a uniqueness theory as our new concept
	(see~\cite{Fischer-Hensel-Laux-Simon}), but all existence proofs 
	are either conditional (see e.g.~\cite{LucStu, Laux-Otto, Laux-Simon}), or rest on additional 
	geometric properties \cite{GuidoTim}.
	
	There is yet another weak solution concept, the viscosity (or levelset) solution~\cite{evansspruck-I, chengigagoto},
	which satisfies both existence and uniqueness (if one allows for ``fattening'' of levelsets) and is compatible 
	with all other solution concepts mentioned above in the sense that a.e.~levelset is a solution in all three 
	senses~\cite{evansspruck, mylecturenotes}. The main drawback of this concept is that it relies on the comparison 
	principle, which only applies in the simple case of two phases. All other notions mentioned above (including our new notion) can 
	be formulated in the more general case of multiple phases (or surface clusters) as they solely rely on the 
	gradient-flow structure of (multi-phase) mean curvature flow. 
	
	\medskip
	
	Of all the weak solution concepts for mean curvature flow, the one which we propose in the present work is 
	the concept which exploits the gradient-flow structure in the strongest form as it mirrors the concepts known 
	for general gradient flows as introduced by De Giorgi~\cite{degiorgi, ambrosio_minmov}, and Sandier and Serfaty~\cite{sandier_gamma_convergence}.
	We refer the reader interested in the general framework to the book by Ambrosio, Gigli, and Savar\'{e}~\cite{AGS}, and Mielke's lecture notes~\cite{mielke}.
	Ever since the seminal work of Jordan, Kinderlehrer, and Otto~\cite{JKO}, the study of (Wasserstein) gradient-flows in the context of partial differential equations has received continuous attention. 
	Formally speaking, a gradient flow is the steepest descent in an energy landscape, which means that for a given energy functional $E$ on a (usually infinitely dimensional curved) space $\mathcal{M}$ equipped with a metric tensor $g$, one considers the equation
	\begin{align}\label{eq:GF PDE}
		\frac{du}{dt}  = - \nabla_g E(u),
	\end{align}
	where $\nabla_g E(u) \in T_u\mathcal{M}$ is the tangent vector dual to the differential $dE(u)$, i.e., $ g_u(\nabla_g E(u),v) = dE(u).v$ for all tangent vectors $v\in T_u\mathcal{M}$.
	If $\mathcal{M}$ is a Riemannian manifold, it is easy to construct solutions to this system of ordinary differential equations, but this formal picture becomes more subtle in the infinite dimensional case as the nonlinearity $-\nabla_g E(u)$ may not interact well with weak convergence.
	However, it is easily seen that, when $\mathcal{M}$, $E$, and $u$  are sufficiently smooth so that one can apply the chain rule to $E(u(t))$, that the validity of \eqref{eq:GF PDE} on a given time interval $(0,T)$ is equivalent to the 
	optimal energy-dissipation inequality
	\begin{align}
		E(u(T)) +\frac12 \int_0^T g_{u(t)}\Big(\frac{du}{dt},\frac{du}{dt}\Big) dt + \frac12 \int_0^T {g_{u(t)}}\Big( \nabla_g E(u(t)), \nabla_g E(u(t))\Big) dt 
		\leq E(u(0)).
	\end{align}
	The beauty of this observation is that one only needs to prove an \emph{inequality} for the nonlinear terms in the equation. 
	
	\medskip
	
	In the case of mean curvature flow, still speaking formally, the energy~$E$ is the area functional; 
	the space $\mathcal{M}$ is the space of all (embedded) $n$-dimensional surfaces in some 
	fixed manifold, say, $\R^d$ with $d>n$; the tangent space $T_\Sigma\mathcal{M}$ at a given configuration 
	$\Sigma\in \mathcal{M}$ can be identified with normal velocity fields~$V$ on $\Sigma$; and the metric 
	tensor is given by the standard $L^2$-structure $g_\Sigma (V,V) = \int_\Sigma |V|^2 d\omega^n$. 
	To turn this formal picture into rigorous analysis is particularly subtle in the case of mean curvature 
	flow as this metric is completely degenerate, a phenomenon which was first recorded by Michor and 
	Mumford \cite{michor-mumford_curves}.
	Geometrically speaking, it is the invariance with respect to reparametrizations which causes this degeneracy and allows to move surfaces at arbitrarily small cost when introducing infinitesimal wrinkles.
	From a PDE viewpoint, one of the key differences between mean curvature flow and more regular Wasserstein gradient flows is that in the latter case, the evolution equation is based on the continuity equation, while the evolution equation for mean curvature flow is based on the transport equation. 
	This degeneracy is the reason why Almgren, Taylor, and Wang \cite{ATW}, and Luckhaus and Sturzenhecker \cite{LucStu} used a proxy for the induced distance of surfaces in their minimizing movements approximation.
	
	Together with Felix Otto \cite{lauxotto_de_giorgi} and in the general multi-phase case with Jona Lelmi \cite{laux-lelmi}, the second author introduced a concept similar to the one here, but still in the framework of sets of finite perimeter, which intrinsically lacks compactness: the perimeter functional 
	is only lower semi-continuous, and in general not continuous. 
	This is why these two convergence proofs are only conditional in the sense that one has to \emph{assume} that the time-integrated energies converge to those of the limit. Under certain assumptions on the geometry in the two-phase case, e.g.~mean convexity $H\geq 0$, one can in fact verify this assumption \cite{GuidoTim}. Geometrically, this assumption rules out the piling up of several layers of surfaces and therefore guarantees the limit varifold to have unit density. Such an assumption is not necessary in our case here.
	
	\medskip
	
	In the present work, we prove existence and (weak-strong) uniqueness of our solution concept. 
	To the best of our knowledge, this is the first concept	 which satisfies these two properties and does not rely on the comparison principle.
	More precisely, we show that solutions to the Allen--Cahn equation, up to passing to subsequences, converge to solutions of our new concept. This part is inspired by Ilmanen's fundamental work \cite{ilmanen}, which shows that any such limit is a Brakke solution.
	However, in our case, a new major difficulty arises since one needs to construct a normal velocity field in the sense of the new notion introduced in the present paper. Let us also mention that our natural square integrability of this normal velocity field, which appears in a transport-type equation for the indicator function of the enclosed volume automatically implies the (H\"older-) continuity of this volume. This is to be compared with the already
	mentioned work of Kim and Tonegawa~\cite{KimTonegawa}, which requires a substantial effort to
	provide the construction of a Brakke flow supported on a network of evolving interfaces
	associated with a continuously evolving family of phase volumes.
	
	The relation between the Allen--Cahn equation and mean curvature flow has been an ongoing investigation for several decades. We will only mention a few key results and refer to the introduction of \cite{Laux-Simon} for more details. 
	Bronsard and Kohn~\cite{Bronsard1991} pointed out the gradient-flow structure of the Allen--Cahn equation, and investigated the radially symmetric case. 
	Evans, Soner, and Souganidis~\cite{Evans1992} proved the convergence to the viscosity solution.
	Ilmanen's convergence proof to a Brakke flow~\cite{ilmanen}, which was fundamental to many following works on the Allen--Cahn equation (including the construction in the present work), has been generalized to several different settings, for example to the case of boundary contact \cite{mizunotonegawa, mizunotonegawa_erratum} or when the underlying space is a manifold with lower bounds on the Ricci curvature, see \cite{pisante1,pisante2}. 
	Furthermore, Tonegawa \cite{tonegawa_integrality} showed that the limiting varifold in Ilmanen's proof has integer multiplicity.
	Together with Fischer and Simon, one of the authors has recently derived the optimal convergence rate for the Allen--Cahn equation to two-phase mean curvature flow solely relying on the gradient-flow structure. 
	At the moment, the only convergence proof in the vector-valued case is the conditional convergence to a BV solution by Simon and one of the authors~\cite{Laux-Simon}.
	Recent advances in the static case include Del Pino, Kowalczyk, and Wei \cite{delpino} who derived counter-examples to a conjecture of De Giorgi by constructing entire stationary solutions to the Allen--Cahn equation satisfying a monotonicity condition. 
	Chodosh and Mantoulidis~\cite{chodosh} have settled a question on the unit density (and on the lower bound of the index) of limits of stationary solutions to the Allen--Cahn equation on $3$-dimensional manifolds with generic metrics. 
		
	\medskip
	
	In the second part of the paper, we show that our varifold solution does not allow for un-physical non-uniqueness before the onset of singularities. More precisely, we establish that as long as a classical solution to mean curvature flow (in the sense of a smooth family of parametrized surfaces moving along their mean curvature) exists, all varifold solutions in the sense of the present work coincide with this classical flow. Our proof of this second theorem generalizes our recent result with Fischer and Simon~\cite{Fischer-Hensel-Laux-Simon} to the case of such varifold solutions for (two-phase) mean curvature flow, and is
	based on a relative entropy technique. 
	
	The basic idea of this approach is to construct a quantity
	which on one side captures the distance between a weak and a classical solution in a sufficiently
	strong sense, and on the other side allows for an estimate on its time evolution. 
	The latter requirement
	makes this task non-trivial due to the typically limited regularity of weak solutions. However,
	many mathematical models in continuum mechanics are equipped with a natural (non-linear) functional
	of the solution: the associated energy. Moreover, control over the time evolution of the energy
	is often provided by means of an energy dissipation principle. In such a setting, the general
	recipe for the construction of a relative entropy is to start with the energy of the weak
	solution and to subtract from it a suitable non-linear functional of the strong solution
	which only depends in a linear way on the weak solution. This particular structure indeed gives hope to
	compute its time evolution as one (at least in principle) only relies on the 
	sharp energy dissipation inequality, the weak formulation of the problem, and finally the higher
	regularity of a classical solution. 
	
	The classical setup in which this approach has been applied concerns
	continuum mechanical models equipped with a strictly convex and dissipated energy functional.
	In fact, in this context the method of relative entropies dates back to the works of Dafermos~\cite{Dafermos1979}
	and DiPerna~\cite{DiPerna1979} on conservation laws in the late 70's, and has, for instance,
	also been used in the context of mathematical fluid mechanics (see, e.g., 
	\cite{Serrin2,FeireislJinNovotny,FeireislNovotny,Wiedemann2018}). In these models, one
	subtracts from the energy functional its first-order Taylor expansion with base point
	located at the classical solution. By strict convexity of the energy, one indeed obtains
	in this way a coercive error functional.

	The relative entropy approach was recently extended to
	cover basic models from curvature-driven interface evolution.
	For instance, Fischer and the first author established in~\cite{Fischer2020b}
	a weak-strong uniqueness principle for Navier--Stokes two-phase flow with surface tension
	based on a suitably constructed relative entropy functional. Their ideas for the treatment
	of the interfacial energy contribution can (at least in principle) be generalized to multi-phase
	evolution problems; see our recent work with Fischer and Simon~\cite{Fischer-Hensel-Laux-Simon}
	concerning multi-phase mean curvature flow in~$\R^2$
	or our recent result~\cite{Hensel2021} for the mean curvature flow of double bubbles in~$\R^3$.
	Let us mention that even before these recent developments, Jerrard and Smets~\cite{Jerrard2015}
	already employed the analogue of our relative entropy technique for a problem in codimension~$2$,
	namely the evolution of curves in~$\R^3$ by their binormal curvature.
	
	\medskip
	
	Finally, let us mention that there are other interface evolution problems for which varifold solutions have been considered. Abels proved the existence of varifold solutions to the two-phase Navier--Stokes equation with 
	surface tension \cite{abels}. More recently, Chambolle and the second author \cite{chambollelaux} showed that varifold solutions, which also satisfy the PDE in a distributional sense, exist for the Hele--Shaw or one-phase Mullins--Sekerka equation. Let us stress that in both of these works, the velocity appears in a continuity equation, not a transport equation.
	

	In spirit, our solution concept is related to Serfaty's solution~\cite{serfaty} (which is based on a definition of velocity due to Mugnai and R\"oger~\cite{mugnai-roeger}), Ilmanen's enhanced Brakke flow~\cite{ilmanen-elliptic-regularization}, and the weak solution defined by Bellettini and Mugnai~\cite{bellettini-mugnai}. 
	However, to the best of our knowledge, for these three solution concepts, a weak-strong uniqueness principle is not known.
	

\section{Main results}
	A key ingredient of the present work is our novel weak solution
	concept for two-phase mean curvature flow.
	Roughly speaking, it defines a solution only in terms of the optimal energy-dissipation relation
	\begin{align}\label{eq:de giorgi inequality formal}
		\frac{d}{dt}\H^{d-1}(\partial A(t)) \leq -\frac12 \int_{\partial A(t)} V^2 \,d\H^{d-1} -\frac12 \int_{\partial A(t)} |\mathbf{H}|^2 \,d\H^{d-1}.
	\end{align}
	The key is that both the normal velocity $V$ as well as the mean curvature vector $\mathbf{H}$ can be defined in a very low regularity setting. 
	This is well-known for the latter via an integration by parts along the closed surface $\partial A(t)$.	
	But for the former, to the best of our knowledge, we give the first definition in this general setting. Let us briefly explain the structure of this definition of $V$. It is well-known that for a smoothly evolving open set $A(t)$, the volume changes according to
	\begin{align*}
		\frac{d}{dt} \mathcal{L}^d(A(t)) = -\int_{\partial A(t)} V \,d\H^{d-1}.
	\end{align*}
	(Here and throughout we use the sign convention $V>0$ for shrinking $A(t)$.) 
	This identity can be localized with a smooth test function, and the resulting PDE  in terms of the characteristic function $\chi(x,t)=\chi_{A(t)}(x)$ reads
	\begin{align*}
		\partial_t \chi + V |\nabla \chi| =0.
	\end{align*}
	This simply means that $\chi$ is transported by the vector field $V n$, where $n=\frac{\nabla \chi}{|\nabla \chi|}$.
	Clearly, the term $|\nabla \chi|$ is not stable under weak convergence, which is why we need to relax this definition.
	When approximating with such smooth sets, we need to take the limit in the sense of varifolds, not currents, which yields
	\begin{align*}
			\partial_t \chi + V \omega =0,
	\end{align*}
	where $\omega$ is the mass measure of the limit varifold $\mu$, and $V$ now is measurable with respect to $\omega$ (instead of $|\nabla \chi|$).
	This is exactly the ``transport'' equation we will use in our definition.
	
	Finally, let us state (the formal version of) Brakke's inequality to compare it to our new concept: For any test function $\phi \in C_c^1(\R^d\times[0,\infty))$ with $\phi\geq0$ Brakke requires
	\begin{align*}
		\frac{d}{dt} \int_{\partial A(t)} \phi \, d\mathcal{H}^{d-1} \leq \int_{\partial A(t)} 
		\left(-\phi |\mathbf{H}|^2 + \mathbf{H}\cdot \nabla \phi + \partial_t \phi \right) \,d\mathcal{H}^{d-1}.
	\end{align*}
	Clearly, Brakke's formulation seems to contain much more information as it requires a whole family of inequalities. 
	In other words, Brakke asks for localized versions of the optimal energy dissipation relation, both in space and time. 
	However, Brakke's formulation in and of itself cannot exclude the sudden loss of (parts of) the solution. Our solution concept prevents this from happening as the inequality \eqref{eq:de giorgi inequality formal} itself provides a certain regularity in time: the phase volume only varies H\"older-continuously in time, see Lemma \ref{lemma:discussion De Giorgi solution}. 
	Even more, we will show in Theorem \ref{theo:weakStrongUniqueness} that, for smooth initial conditions, our weak solution is unique and hence agrees with the classical solution as long as the latter exists, i.e., until the first topological change.

\begin{definition}[De~Giorgi type varifold solutions for two-phase mean curvature flow]
\label{def:twoPhaseDeGiorgiVarifoldSolutions}
Let $\mu = \mathcal{L}^1\otimes (\mu_t)_{t\in (0,\infty)}$ be
a family of oriented varifolds $\mu_t \in \mathcal{M}(\R^d{\times}\S^{d-1})$, $t\in (0,\infty)$,
such that the map $(0,\infty) \ni t \mapsto \int_{\R^d{\times}\S^{d-1}} \eta(\cdot,\cdot,t) \,d\mu_t$
is measurable for all $\eta \in L^1((0,\infty);C_0(\R^d{\times}\S^{d-1}))$. Consider 
also a family $A=(A(t))_{t \in (0,\infty)}$ of open subsets of $\R^d$
with finite perimeter in~$\R^d$ such that the associated indicator function~$\chi(\cdot,t):=\chi_{A(t)}$,
$t \in (0,\infty)$, satisfies $\chi \in L^\infty((0,\infty);\mathrm{BV}(\R^d;\{0,1\}))$.
Let $\sigma>0$ be a surface tension constant.

Given an initial oriented varifold~$\mu_0 \in \mathcal{M}(\R^d{\times}\S^{d-1})$
and an initial phase indicator function $\chi_0 \in \mathrm{BV}(\R^d;\{0,1\})$,
we call the pair~$(\mu,\chi)$ a \emph{De~Giorgi type varifold solution for two-phase
mean curvature flow with initial data~$(\mu_0,\chi_0)$} if the following holds.
\begin{subequations}
\begin{itemize}
\item (Existence of a normal speed) Writing $\mu_t = \omega_t \otimes (\lambda_{x,t})_{x\in\R^d}$
for the disintegration of $\mu_t$, $t \in (0,\infty)$, we require the existence 
of some $V \in L^2((0,\infty);L^2(\R^d,d\omega_t))$ encoding a normal speed
in the sense of
\begin{align}
\label{eq:evolPhase}
\sigma\int_{\R^d} \chi(\cdot,T)\zeta(\cdot,T) \,dx - 
\sigma\int_{\R^d} \chi_0\zeta(\cdot,0) \,dx 
= \sigma\int_{0}^{T} \int_{\R^d} \chi \partial_t\zeta \,dx dt
- \int_{0}^{T} \int_{\R^d} V \zeta \,d\omega_t dt
\end{align}
for almost every $T \in (0,\infty)$ and all $\zeta \in C^\infty_{c}(\R^d {\times} [0,\infty))$.
\item (Existence of a generalized mean curvature vector) We require the existence of some
$\mathbf{H} \in L^2((0,\infty);L^2(\R^d,d\omega_t;\R^d))$ encoding a generalized
mean curvature vector by 
\begin{align}
\label{eq:weakCurvature}
\int_{0}^{\infty} \int_{\R^d} \mathbf{H} \cdot B \,d\omega_t dt
= - \int_{0}^{\infty} \int_{\R^d {\times} \S^{d-1}}
(I_d {-} p \otimes p) : \nabla B \,d\mu_t dt
\end{align}
for all $B \in C^\infty_{c}(\R^d {\times} [0,\infty);\R^d)$.
\item (De~Giorgi type inequality for mean curvature flow) A sharp energy dissipation
principle \`a la De~Giorgi holds true in form of
\begin{align}
\label{eq:DeGiorgiInequality}
\int_{\R^d} 1 \,d\omega_{T}
+ \frac{1}{2} \int_{0}^{T} \int_{\R^d} |V|^2 \,d\omega_t dt
+ \frac{1}{2} \int_{0}^{T} \int_{\R^d} |\mathbf{H}|^2 \,d\omega_t dt
\leq \int_{\R^d} 1 \,d\omega_{0}
\end{align}
for almost every $T \in (0,\infty)$.
\item (Compatibility) For almost every $t \in (0,\infty)$
and all $\xi \in C^\infty_{c}(\R^d)$ it holds
\begin{align}
\label{eq:compatibility}
\int_{\R^d} \xi \cdot \,\sigma d\nabla\chi(\cdot,t) 
= \int_{\R^d {\times} \S^{d-1}} \xi \cdot p \,d\mu_t.
\end{align}
\end{itemize}
\end{subequations}
\end{definition}

We observe that due to the compatibility condition~\eqref{eq:compatibility}
it holds for almost every~$t \in (0,\infty)$ that $\,\sigma |\nabla\chi(\cdot,t)| \leq \omega_t$ 
in the sense of measures. In particular, for almost every~$t \in (0,\infty)$
the Radon--Nikod\'ym derivative $\rho(\cdot,t) := \smash{\frac{\sigma d|\nabla\chi(\cdot,t)|}{d\omega_t}}$
exists and satisfies 
\begin{align}
\label{eq:inverseMultiplicity}
\rho(x,t) &\in [0,1] 
&&\text{ for } \omega_t \text{ almost every } x \in \R^d,
\\ \label{eq:RadonNikodymProperty}
\sigma\int_{\p^*A(t)} f \,d\H^{d-1} &= \int_{\R^d{\times}\S^{d-1}}
f\rho(\cdot,t) \,d\omega_t
&&\text{ for all } f \in C^\infty_{c}(\R^d).
\end{align}

In the following first main result, we show that any limit of the Allen--Cahn equation
\begin{align}\label{eq:allen cahn}
\partial_t u_\eps = \Delta u_\eps - \frac1{\eps^2} W'(u_\eps)
\quad\text{in } \R^d \times(0,\infty)
\end{align}
is a weak solution in the above sense. 
Here $W\colon \R\to [0,\infty)$ is a double-well potential, which for simplicity, we will assume to be of the standard form $W(u)=\frac14 u^2 (u-1)^2$.
Recall that the Allen--Cahn equation is the $L^2$-gradient flow of the Cahn--Hilliard energy 
\begin{align}
E_\eps(u) := \int_{\R^d} \left( \frac\eps2|\nabla u|^2 + \frac1\eps W(u)\right) dx
\end{align}
on the slow time scale. This gradient-flow structure can be recognized at the identity
\begin{align*}
\frac{d}{dt} E_\eps(u_\eps(\,\cdot\,,t)) = - \int_{\R^d} \eps(\partial_tu_\eps(\,\cdot\,,t))^2 \,dx\leq 0,
\end{align*}
which follows from a straight-forward computation. 
In view of our approach here, cf.~\eqref{eq:de giorgi inequality formal}, we should rather measure the  energy-dissipation in the balanced way
\begin{align}\label{eq:EDI naive}
	\frac{d}{dt} E_\eps(u_\eps(\,\cdot\,,t)) = -\frac12 \int_{\R^d}  \eps(\partial_tu_\eps(\,\cdot\,,t))^2 \,dx - \frac12\int_{\R^d}  \frac1\eps \left( \eps \Delta u_\eps(\,\cdot\,,t) - \frac1\eps W'(u_\eps(\,\cdot\,,t)) \right)^2 \,dx,
\end{align}
where we have used equation \eqref{eq:allen cahn} once more.

\begin{theorem}[Convergence of the Allen--Cahn equation to De Giorgi type varifold solution]
\label{theo:Existence}
Let $u_\eps$ denote the solution to the Allen--Cahn equation with well-prepared initial conditions  $u_{\eps,0}$ in the sense of Ilmanen~\emph{\eqref{eq:ilmanen initial 1}--\eqref{eq:ilmanen initial 3}}.
Then there exist a measurable function $\chi \colon \R^d \times (0,\infty ) \to \{0,1\} $ and a family of oriented varifolds $\mu = \mathcal{L}^1\otimes (\mu_t)_{t\in(0,\infty)} \in \mathcal{M}(  \R^d \times (0,\infty) \times \S^{d-1})$ with associated mass measure $\omega \in \mathcal{M}(\R^d\times (0,\infty))$ such that 
 \begin{align}
&\lim_{\eps \downarrow0} u_\eps = \chi&& \text{strongly in } L^1_{loc}(\R^d\times(0,\infty)),\\
&\lim_{\eps\downarrow0} \left(\frac\eps2|\nabla u_\eps|^2 +\frac1\eps W(u_\eps)\right) \mathcal{L}^d\otimes \mathcal{L}^1 = \omega && \text{weakly-$\ast$ as Radon measures.}
\end{align}
Moreover, the pair $(\chi,\mu)$ is a solution of mean curvature flow in the sense of Definition \ref{def:twoPhaseDeGiorgiVarifoldSolutions} with initial conditions $(\chi_0,\mu_0)$ given by
\begin{align}
& \lim_{\eps\downarrow0} u_{\eps,0}=	\chi_0& & \text{strongly in }L^1_{loc}(\R^d),\\
\label{eq:initialEnergyDensityConvergence}
	&\lim_{\eps \downarrow 0} \left(\frac\eps2|\nabla u_{\eps,0}|^2 +\frac1\eps W(u_{\eps,0})\right) \mathcal{L}^d =\omega_0 &&  \text{weakly-$\ast$ as finite Radon measures}.
\end{align}

In addition, if the initial energy measures $\left(\frac\eps2 |\nabla u_{\eps,0}|^2 +\frac1\eps W(u_{\eps,0})\right) \mathcal{L}^d$ have uniformly bounded second moments (e.g.~if $u_{\eps,0}$ has compact support), the energy-dissipation inequality \eqref{eq:DeGiorgiInequality} holds also with $(0,T)$ replaced by $(s,t)$ for \emph{all} $t\in (0,\infty)$ and almost every $s<t$. In particular, the total mass is non-increasing, i.e.,  $\omega_t(\R^d) \leq \omega_s(\R^d)$ for all $0<s<t$.

Finally, for a.e.~$t\in(0,\infty)$, the corresponding unoriented varifold $\frac1\sigma \hat\mu_t$ (which identifies antipodal points $\pm p$ on $\S^{d-1}$) is integer rectifiable.
\end{theorem}

%
We turn to the question of weak-strong uniqueness for
our concept of De~Giorgi type varifold solutions for two-phase
mean curvature flow. Our strategy is based on lifting the
recently developed relative entropy approach for curvature driven interface evolution
problems to the present varifold setting by constructing two suitable
distance measures between a De~Giorgi type varifold solution
and a classical solution for two-phase mean curvature flow.
To fix notation, let $\Tstrong \in (0,\infty)$ be a finite time horizon,
and let $\mathscr{A}=(\mathscr{A}(t))_{t \in [0,\Tstrong]}$
be a smoothly evolving family of bounded domains in~$\R^d$
whose associated family of interfaces $\mathcal{I}=(\mathcal{I}(t))_{t \in [0,\Tstrong]}$
is assumed to evolve smoothly by their mean curvature. Moreover, 
let $(\mu,\chi)$ be a De~Giorgi type varifold solution
in the sense of Definition~\ref{def:twoPhaseDeGiorgiVarifoldSolutions}.

The distance measures
between~$\mathscr{A}$ and~$(\mu,\chi)$ then more precisely consist of a relative entropy 
\begin{align}
\label{eq:relativeEntropyIntro}
E[\mu,\chi|\mathscr{A}](t) := \int_{\R^d {\times} \S^{d-1}} 
1 - p \cdot \xi(x,t) \,d\mu_t(x,p) \geq 0, 
\quad t \in [0,\Tstrong],
\end{align}
as well as a bulk error
\begin{align}
\label{eq:bulkErrorIntro}
E_{\mathrm{bulk}}[\chi|\mathscr{A}](t) := \sigma\int_{\R^d} 
|\chi_{A(t)}(x) {-} \chi_{\mathscr{A}(t)}(x)| |\vartheta(x,t)| \,dx \geq 0, 
\quad t \in [0,\Tstrong].
\end{align}
The vector field~$\xi$ appearing in the definition of the relative
entropy~\eqref{eq:relativeEntropyIntro} represents a suitable
extension of the (inward pointing) normal vector field associated
with the classical solution~$\mathscr{A}$, whereas the weight~$|\vartheta|$
should be thought of as a smooth version of $(x,t) \mapsto \min\{1,\dist(x,\mathcal{I}(t))\}$.
For a rigorous construction of the required data~$(\xi,\vartheta)$,
we refer the reader to the beginning of Section~\ref{sec:weakStrongUniqueness}.
Let us only mention at this point that the length of the vector field~$\xi$ will be subject
to a condition of the form $\min\{1,\dist^2(x,\mathcal{I}(t))\} \leq 1 - |\xi(x,t)|$
for all $(x,t)\in \R^d{\times}[0,\Tstrong]$. In particular, the definition~\eqref{eq:relativeEntropyIntro}
thus provides an error control in terms of the tilt-excess as well as the distance to the
classical solution in form of
\begin{align*}
\int_{\R^d {\times} \S^{d-1}} \frac{1}{2}|p {-} \xi(\cdot,t)|^2 \,d\mu_t 
+ \int_{\R^d} \min\{1,\mathrm{dist}^2(\cdot,\mathcal{I}(t))\} \,d\omega_t 
&\leq 2E[\mu,\chi|\mathscr{A}](t).
\end{align*}
Finally, we remark that~\eqref{eq:bulkErrorIntro} represents nothing more than
a version of the well-known Luckhaus--Sturzenhecker~\cite{LucStu} 
or Almgren--Taylor--Wang~\cite{ATW} type distance.

With these ingredients in place, we 
may now turn to the formulation of the second main result of the present work.

\begin{theorem}[Weak-strong uniqueness and quantitative stability for two-phase De~Giorgi type varifold solutions]
\label{theo:weakStrongUniqueness}
Let $\Tstrong\in (0,\infty)$ be a finite time horizon, 
and let $\mathscr{A}=(\mathscr{A}_t)_{t\in [0,\Tstrong]}$ be a smoothly evolving 
family of bounded domains in~$\R^d$
whose associated family of interfaces $(\mathcal{I}(t))_{t \in [0,\Tstrong]}$
evolves smoothly by their mean curvature. Furthermore,
let $(\mu,\chi)$ be a De~Giorgi type varifold solution
for two-phase mean curvature flow with initial data~$(\mu_0,\chi_0)$ in the sense of
Definition~\ref{def:twoPhaseDeGiorgiVarifoldSolutions}.

Defining the relative entropy~$E[\mu,\chi|\mathscr{A}]$ 
and the bulk error~$E_{\mathrm{bulk}}[\chi|\mathscr{A}]$ 
by~\eqref{eq:relativeEntropyIntro} and~\eqref{eq:bulkErrorIntro}, respectively, 
we then have stability estimates for these quantities in form of
\begin{align}
\label{eq:stabilityRelEntropy}
E[\mu,\chi|\mathscr{A}](T) 
&\leq E[\mu_0,\chi_0|\mathscr{A}(0)]
+ C \int_{0}^{T} E[\mu,\chi|\mathscr{A}](t) \,dt
\\ \label{eq:stabilityBulkError}
E_{\mathrm{bulk}}[\chi|\mathscr{A}](T)
&\leq E_{\mathrm{bulk}}[\chi_0|\mathscr{A}(0)]
+ E[\mu_0,\chi_0|\mathscr{A}(0)]
+ \int_{0}^{T} E_{\mathrm{bulk}}[\chi|\mathscr{A}](t)
+ E[\mu,\chi|\mathscr{A}](t) \,dt
\end{align}
for some constant $C=C(\mathscr{A})>0$
and almost every $T \in [0,\Tstrong]$.

In particular, if $\chi_0=\chi_{\mathscr{A}(0)}$
and $\mu_0 = \sigma|\nabla\chi_0|\otimes(\smash{\delta_{\frac{\nabla\chi_0}{|\nabla\chi_0|}(x)}})_{x\in\R^d}$
hold true, then for almost every $t \in [0,\Tstrong]$
\begin{align}
\label{eq:weakStrong1}
\chi(\cdot,t) &= \chi_{\mathscr{A}(t)}
\text{ almost everywhere in } \R^d,
\\ \label{eq:weakStrong2}
\mu_t &= \sigma|\nabla\chi(\cdot,t)|\otimes(\delta_{\frac{\nabla\chi(\cdot,t)}{|\nabla\chi(\cdot,t)|}(x)})_{x\in\R^d}.
\end{align}
In words, the De~Giorgi type varifold solution~$(\mu,\chi)$ reduces to the smooth evolution~$\mathscr{A}$.
\end{theorem}

	In addition to our two main theorems (Theorems \ref{theo:Existence} and \ref{theo:weakStrongUniqueness} above), one can also derive further properties of our new solution concept from Definition \ref{def:twoPhaseDeGiorgiVarifoldSolutions}. 
	We collect these properties in the following lemma. Its proof is relegated to a separate Section~\ref{sec:proofLemma}.
	
	\begin{lemma}\label{lemma:discussion De Giorgi solution}
		\begin{enumerate}[(i)]
			\item 	(Classical solutions are weak solutions)
			\label{item:classical is weak}
			If $(\partial \mathscr{A}(t))_{t\in[0,T)} $ 
			is a classical solution to mean curvature flow for some $T\in (0,\infty]$, 
			then the pair consisting of $\chi(x,t):= \chi_{\mathscr{A}(t)}(x)$ 
			and $\mu_t:= |\nabla \chi(\cdot,t)| \otimes 
			\delta_{\frac{\nabla \chi(\cdot,t)}{|\nabla \chi(\cdot,t)|}} $ 
			is a De~Giorgi type varifold solution in the sense of 
			Definition~\ref{def:twoPhaseDeGiorgiVarifoldSolutions} (with obvious modifications 
			to Definition~\ref{def:twoPhaseDeGiorgiVarifoldSolutions} to restrict the required conditions
			to the possibly finite open time interval~$(0,T)$).
						
			\item (Regular weak solutions are classical solutions) 
			\label{item:weak is classical}
			If $(\chi,\mu)$ is a De~Giorgi type varifold solution 
			such that there exists $T\in (0,\infty]$ so that $\chi(x,t)=\chi_{\mathscr{A}(t)}(x)$
			and $\mu_t = |\nabla \chi(\cdot,t)| \otimes 
			\delta_{\frac{\nabla \chi(\cdot,t)}{|\nabla \chi(\cdot,t)|}}$ 
			for all $t\in (0,T)$ and some smoothly evolving family $(\mathscr{A}(t))_{t\in [0,T)}$, 
			then $(\partial \mathscr{A}(t))_{t\in [0,T)}$
			is a classical solution to mean curvature flow.
			
			\item (Regularity properties of weak solutions)
			\label{item:structure}
			Any De~Giorgi type varifold solution $(\chi,\mu)$ according to Definition~\ref{def:twoPhaseDeGiorgiVarifoldSolutions} has a certain regularity. First, for a.e.~$t\in (0,\infty)$, the corresponding unoriented varifold $\hat \mu_t$ is $(d-1)$-rectifiable. Second, the volumes change (H\"older-)continuously in time: 
			\begin{align}\label{eq:lemma cont vol}
				\sigma \mathcal{L}^d(A(t) \Delta A(s))
				\leq \sqrt{2}\omega_0(\R^d) \sqrt{t-s} \quad \text{for all } 0\leq s<t<\infty.
			\end{align} 
			
			\item (Sequential compactness of solution space)
			\label{item:compact sol space}
			Let $(\chi_k,\mu_k)$ be a sequence of De~Giorgi type varifold solutions such that the energies are non-increasing in time, i.e., for all $k$, the map $t\mapsto (\omega_k)_t(\R^d)$ is non-increasing. 
			Moreover, we assume that the initial energy $(\omega_k)_0(\R^d)$ 
			is uniformly bounded in $k\in \N$, and that the sequence of $(\omega_k)_0$
			is tight. Then there exists a subsequence $(k_\ell)_\ell$ and a De~Giorgi type varifold solution $(\chi,\mu)$ such that $\chi_{k_\ell}\to \chi $ in $L^1_{loc}$ and 
			$ \mu_{k_\ell}\stackrel{\ast}{\rightharpoonup}\mu$ as Radon measures.
		\end{enumerate}
	\end{lemma}
	
	Let us finally comment on higher multiplicity in our solution concept.
	\begin{rem}
		While the De Giorgi type varifold solution allows for arbitrary multiplicity of the varifolds $\mu_t$,  some peculiarities may occur in general for De Giorgi varifold solutions when $\mu_t$ has higher multiplicity. Let us note that, as long as a classical solution exists, we can control the multiplicity of our solutions by our relative entropy functional \eqref{eq:relativeEntropyIntro}, see \eqref{eq:coercivity5} below for the details.
		However, there exist general De Giorgi varifold solutions, which have higher multiplicity, and in this case it is easy to see that the inequality in the energy-dissipation relation \eqref{eq:DeGiorgiInequality} becomes strict. 
		For example, a self-similarly shrinking sphere with higher multiplicity, e.g.~$\chi(\cdot,t)=0$ and $\omega_t= 2 \H^{d-1}\llcorner \partial B_{r(t)}$, or $\chi(\cdot,t)=\chi_{B_{r(t)}}$ and $\omega_t= 3 \H^{d-1}\llcorner \partial B_{r(t)}$, where $r(t)=\sqrt{1-2(d-1)t^2}$
		can be seen to satisfy Definition \ref{def:twoPhaseDeGiorgiVarifoldSolutions}. But here De Giorgi's inequality is strict ``$<$'', because $V=0$ in the first case (since the phase volume does not change) and $V=-\frac13 \frac{dr}{dt}$ in the second case, instead of the expected normal velocity $V=-\frac{dr}{dt}$, which would naturally appear in the case of multiplicity $1$.
		Theorem \ref{theo:weakStrongUniqueness} shows that for smooth initial conditions with unit density and for short time, the density of $\omega_t$ is $1$, which rules out the two examples above. However, it is not clear whether higher multiplicity interfaces could in principle appear after singularities.
		Since our construction applies in the same framework as Ilmanen's construction of a Brakke flow \cite{ilmanen}, one could also combine both our De Giorgi inequality and Brakke's inequality. Solutions satisfying both inequalities (in particular any limit of solutions to the Allen--Cahn equation) would then of course enjoy all properties of our and of Brakke's solution.
	\end{rem}


\subsection*{Structure of the paper}
The rest of the paper is structured as follows. In Section \ref{sec:existence}, we prove Theorem \ref{theo:Existence}. Section \ref{sec:weakStrongUniqueness} is devoted to the proof of Theorem \ref{theo:weakStrongUniqueness}. In Section \ref{sec:proofLemma}, we provide a short proof of Lemma \ref{lemma:discussion De Giorgi solution}. Finally, in Section \ref{sec:multi-phase}, we propose a generalization of our solution concept to the multi-phase case. Here we also state and prove a weak-strong uniqueness statement analogous to Theorem \ref{theo:weakStrongUniqueness}. 
%

\section{Existence of De~Giorgi type varifold solutions by Allen--Cahn approximation}
\label{sec:existence}

Let $u_\eps$ be the solution to \eqref{eq:allen cahn} with given initial conditions $u_\eps(\,\cdot\,,0)=u_{\eps,0}$. Throughout we will assume the uniform energy bound for the initial conditions
\begin{align}\label{eq:ac initial energy}
	E_0 := \sup_{\eps>0} E_\eps(u_{\eps,0}) <\infty.
\end{align}
The basis of our proof of Theorem \ref{theo:Existence} is the energy-dissipation inequality for $u_\eps$, which states
\begin{align}\label{eq:EDI}
	E_\eps(u_\eps(\,\cdot\,,T)) + \frac12  \int_{\R^d	\times(0,T)}\eps (\partial_t u_\eps)^2 \,dx dt + \frac12 \int_{\R^d	\times(0,T)}\frac1\eps \left( \eps \Delta u_\eps - \frac1\eps W'(u_\eps) \right)^2 \,dx dt = E_\eps(u_{\eps,0})
\end{align}
for all $T\in(0,\infty)$. This identity (which we will only use as an inequality ``$\leq$'') is an immediate consequence of integrating \eqref{eq:EDI naive} and plugging in \eqref{eq:allen cahn} to replace half of the dissipation by the right-hand side of the PDE.
Furthermore, we will crucially use the well-known quantity
\begin{align}
	\psi_\eps := \phi\circ u_\eps, \quad \text{where } \phi(u) := \int_0^{u} \sqrt{2W(s)} \, ds.
\end{align}
Recall the Modica--Mortola/Bogomol'nyi-trick \cite{modicamortola, bogomolnyi}, i.e., the elegant combination of the chain rule and Young's inequality
$
	|\nabla \psi_\eps| = \sqrt{2W(u_\eps)} |\nabla u_\eps| \leq \frac\eps2 |\nabla u_\eps|^2 + \frac1\eps W(u_\eps),
$ which implies
\begin{align}\label{eq:nabla psi intro}
	\sup_{t\in(0,\infty)} \int_{\R^d} 	|\nabla \psi_\eps| \,dx \leq E_\eps(u_{\eps,0}).
\end{align}
This motivates to associate to the solution $u_\eps$ the oriented space-time varifold
\begin{align}\label{eq:def mu_eps}
	\mu_\eps := \left( | \nabla \psi_\eps| \L^d\llcorner \R^d\right) \otimes 
	\left(\L^1 \llcorner (0,\infty)\right) \otimes 
	\left(\delta_\frac{\nabla \psi_\eps}{|\nabla \psi_\eps|}\llcorner\S^{d-1}\right).
\end{align}
In simpler words, $\mu_\eps$ is the Radon measure on $\R^d\times (0,\infty)\times \S^{d-1}$ such that for any test function $\varphi \in C_c(\R^d\times (0,\infty) \times \S^{d-1})$
\begin{align}
	\int_{\R^d\times (0,\infty)\times \S^{d-1}} \varphi(x,t,p) \,d\mu_\eps(x,t,p) 
	= \int_{\R^d\times(0,\infty)} \varphi\left(x,t,\frac{\nabla\psi_\eps}
	{|\nabla\psi_\eps|}\right) |\nabla \psi_\eps(x,t)|  \,dx dt.
\end{align}
We denote the associated mass measure as the $(x,t)$-marginal of $\mu_\eps$
\begin{align}\label{eq:def lambda_eps}
	\omega_\eps :=  \left( | \nabla \psi_\eps| \L^d\llcorner \R^d\right) \otimes \left(\L^1\llcorner (0,\infty)\right),
\end{align}
which simply means that for any test function $\zeta \in C_c(\R^d\times(0,\infty))$
\begin{align}
	\int_{\R^d\times(0,\infty)} \zeta(x,t) \,d\omega_\eps (x,t)
	= \int_{\R^d\times(0,\infty)} \zeta (x,t) |\nabla \psi_\eps(x,t)| \,dx dt.
\end{align}
Next, we define the approximate normal velocity $V_\eps$ and the approximate 
mean curvature vector~$\mathbf{H}_\eps$ to be the $\omega_\eps$-measurable functions
\begin{align}
	\label{eq:def Veps} V_\eps &:= -\frac{\eps \partial_t u_\eps}{\sqrt{2W(u_\eps)}},
	\\ \label{eq:def Heps} \mathbf{H}_\eps &:=  -\left( \eps \Delta u_\eps - \frac1\eps W'(u_\eps)\right) \frac{\nabla \psi_\eps}{|\nabla \psi_\eps|}.
\end{align}
On a technical note, in all of the above formulas, we define $\frac{\nabla \psi_\eps}{|\nabla \psi_\eps|}:=e_1$ on the set $\{|\nabla \psi_\eps|=0\}$ and $ \frac{\eps \partial_t u_\eps}{\sqrt{2W(u_\eps)}} := 0$ on the set $\{\sqrt{2W(u_\eps)} =0\} \subset \{|\nabla \psi_\eps|=0\}$, which anyways are outside the support of the measure $\omega_\eps$. 

The definitions of $\mu_\eps$ and $\mathbf{H}_\eps$ are very much in the spirit of Ilmanen's fundamental contribution \cite{ilmanen}, where the energy density $\frac{\eps}{2}|\nabla u_\eps|^2 + \frac1\eps W(u_\eps)$ is used instead of $|\nabla \psi_\eps|$ in the definition of the varifold. We will see that our definition \eqref{eq:def mu_eps} is very handy, but this is only a minor change. The normal velocity $V_\eps$, however, does not appear at all in the work \cite{ilmanen}---as the velocity does not appear in Brakke's formulation. The precise choice \eqref{eq:def Veps} will be handy as well, which will become evident in Lemma \ref{lemma:Veps} and Proposition \ref{prop:Vsquared}.


The following simple compactness statement can be found in the literature.
\begin{proposition}[Compactness]\label{prop:compactness}
	Let the initial conditions $u_\eps(\,\cdot\,,0)$ satisfy the energy bound~\eqref{eq:ac initial energy}.
	Then 
	\begin{align}
		\label{eq:nabla psi}\sup_{t\in(0,\infty)} \int_{\R^d}|\nabla \psi_\eps| \,dx \leq& E_0,
		\\ 	\label{eq:dt psi}\int_0^T\int_{\R^d}|\partial_t \psi_\eps| \,dx dt \leq &(1+2T)E_0 \quad \text{for all }T<\infty.
	\end{align}
	Therefore, after passage to a subsequence $\eps\downarrow0$, there exists a family of 
	sets of finite perimeter $(A(t))_{t\in (0,\infty)}$ such that 
	\begin{align}
		\label{eq:compactness psi} \psi_\eps \to \psi & \quad \text{in }L^1_{loc}(\R^d\times (0,\infty)),
	\end{align}
	where $\psi(x,t) := \sigma \chi_{A(t)}(x)$ with $\sigma := \phi(1) =\int_0^1\sqrt{2W(s)}\,ds$,
	\begin{align}
		\operatorname{ess\,sup}_{t\in(0,\infty)} 
		\limits \sigma \mathcal{H}^{d-1}(\partial^\ast A(t)) \leq E_0
	\end{align}
	and for any $0\leq s<t $ it holds
	\begin{align}\label{eq:compactness cont vol}
		\sigma \mathcal{L}^d(A(t) \Delta A(s))
		\leq E_0 \sqrt{t-s}.
	\end{align}

	Furthermore, there exist Radon measures $\mu$ and $\omega$ on $\R^d\times(0,\infty)\times \S^{d-1}$ and $\R^d\times(0,\infty)$, respectively, such that 
	\begin{align}
		\label{eq:compactness mu} \mu_\eps \stackrel{\ast}{\rightharpoonup} \mu 
		&\quad 
		\text{as Radon measures in } \R^d\times (0,\infty)\times \S^{d-1},
		\\\label{eq:compactness lambda} \omega_\eps \stackrel{\ast}{\rightharpoonup} \omega& 
		\quad 
		\text{as Radon measures in } \R^d\times (0,\infty).
 	\end{align}
\end{proposition}

\begin{rem}
	Ilmanen~\cite[Sections 5.3--5.5]{ilmanen} showed even more, namely that 
	\begin{align}\label{eq:ilmanen slicing}
		(\omega_\eps)_t \stackrel{\ast}{\rightharpoonup} \omega_t
		\quad 
		\text{as finite Radon measures in } \R^d
	\end{align}
	for a.e.~$t\in(0,\infty)$. However, this is not needed in in the main part of our proof as the disintegration measure $\omega_t$ only enters into our solution concept through the \emph{total} mass/energy $\omega_t(\R^d)$.
	This term is treated in Lemma~\ref{lemma:energy lsc} below.
	We will however make use of \eqref{eq:ilmanen slicing} to show the slight improvement of our statement at the end of Theorem \ref{theo:Existence}.
\end{rem}

\begin{proof}[Sketch of the proof of Propositon \ref{prop:compactness}]
	Although the statements can be found in the literature, let us briefly sketch the simple proofs here.
	
	The estimates \eqref{eq:nabla psi} and \eqref{eq:dt psi} are immediate consequences of the Modica--Mortola/Bogomol'nyi-trick, see the discussion in the paragraph before \eqref{eq:nabla psi intro}. The same trick applies to the time derivative instead of the spatial gradient. Then the space-time integral of $|\nabla \psi_\eps|$ is controlled using the energy-dissipation estimate in form of
	\begin{align*}
		\sup_{T\in (0,\infty)} \left\{ E_\eps(u_\eps(\,\cdot\,,T)) +  \int_{\R^d	\times(0,T)}\eps (\partial_t u_\eps)^2 \,dx dt \right\} \leq E_\eps(u_{\eps,0}).
	\end{align*}
	
	The H\"older-continuity follows (with a suboptimal constant) from the more general statement \cite[Lemma 2.9]{Laux-Simon}. To get the optimal constant in our statement, we estimate by Cauchy--Schwarz
		\begin{align*}
		\int_{\R^d} |\psi_\eps(x,t)-\psi_\eps(x,s)|\,dx
			&\leq  \int_s^t \int_{\R^d} |\partial_t \psi_\eps| \,dx \, dt'
			\\&= \int_{\R^d\times(s,t)} \sqrt{2W(u_\eps)} | \partial_t u_\eps|\,dx\,dt'
			\\&\leq \left( \int_{\R^d\times(0,\infty)} \eps (\partial_t u_\eps)^2 \, dx\,dt'\right)^\frac12 
			\left( \int_{\R^d\times(s,t)} 2\frac1\eps W(u_\eps) \, dx\,dt'\right)^\frac12.
		\end{align*}
	The left-hand side of our inequality converges to the desired 
	$\sigma \mathcal{L}^d(A(t)\Delta A(s))$, at least for a.e.~$0<s<t$.
	The first right-hand side integral is bounded by $E_0$ and the second one by $2(t-s) E_0$. This gives the estimate with an additional prefactor $\sqrt{2}$. Treating the second integral with Proposition \ref{prop:ilmanen} below instead, one gets rid of this prefactor. After possibly redefining the limit $\chi$ on a negligible set of (positive) times, we obtain \eqref{eq:compactness cont vol} for \emph{all} pairs $0\leq s<t$.
\end{proof}

	The next proposition is at the heart of Ilmanen's fundamental contribution \cite{ilmanen}.
	It states that for well-chosen initial conditions, the discrepancy $\frac\eps2|\nabla u_\eps(x,t)|^2  - \frac1\eps W(u_\eps(x,t))$ between the two terms of the energy will remain non-positive for all time. Furthermore, in the limit $\eps\downarrow0$, this discrepancy vanishes in the sense that 
	\begin{align}
		\left(\frac\eps2|\nabla u_\eps|^2  - \frac1\eps W(u_\eps)\right) \mathcal{L}^d\otimes \mathcal{L}^1  \stackrel{\ast}{\rightharpoonup} 0 \quad \text{as measures}.
	\end{align} 
	In other words, asymptotically as $\eps \downarrow0$, there is equipartition of energy. 
	We also include a simple post-processed version of this statement for later use.
\begin{proposition}[Ilmanen \cite{ilmanen}]\label{prop:ilmanen}
	If the initial conditions $u_{\eps,0}$ are well-prepared in the sense of Ilmanen, i.e., 
	\begin{align}
	  \label{eq:ilmanen initial 1}\frac\eps2|\nabla u_{\eps,0}(x)|^2  
		\leq \frac1\eps W(u_{\eps,0}(x)) &\quad \text{ for all $x\in \R^d$,} 
		\\	\label{eq:ilmanen initial 2} \sup_{\eps>0}E_\eps(u_{\eps,0}) <\infty,
	\end{align}
	and denoting $\omega_0:=\textup{weak-}^\ast \lim \left(\frac\eps2|\nabla u_{\eps,0}|^2 + \frac1\eps W(u_{\eps,0})\right) \mathcal{L}^{d}$, it holds
	\begin{align}
		\label{eq:ilmanen initial 3} \lim_{\eps\downarrow0} E_\eps(u_{\eps,0}) = \omega_0(\R^d).
	\end{align}
	Then the analogous statements to \eqref{eq:ilmanen initial 1} and \eqref{eq:ilmanen initial 2}  hold for all $t>0$:
	\begin{align}
	\label{eq:ilmanen 2} \frac\eps2|\nabla u_\eps(x,t)|^2  \leq \frac1\eps W(u_\eps(x,t)) &\quad \text{ for all $x\in \R^d$ and all $t\geq 0$, and} 
	\\ \label{eq:ilmanen 3}\sup_{\eps>0}E_\eps(u_\eps(\,\cdot\,,t)) <\infty& \quad \text{for all $t\geq 0$}.
	\end{align}
	Moreover,	for all test functions $\zeta \in C_c(\R^d\times (0,\infty))$, it holds
	\begin{align}\label{eq:discrepancy 1}
		\lim_{\eps\downarrow 0} \int_{\R^d\times(0,\infty)} \zeta \left(\frac\eps2|\nabla u_\eps(x,t)|^2  - \frac1\eps W(u_\eps(x,t))\right) dx dt =0
	\end{align}
	and
	\begin{align}
			\notag&\lim_{\eps\downarrow 0} \int_{\R^d\times(0,\infty)} \frac12 \zeta \left(\sqrt{\eps} |\nabla u_\eps| 
			- \frac1{\sqrt{\eps}}\sqrt{2W(u_\eps)} \right)^2 dx dt 
			\\\label{eq:discrepancy 2}&=\lim_{\eps\downarrow 0} \int_{\R^d\times(0,\infty)} \zeta \left( \frac\eps2 |\nabla u_\eps|^2 
			+ \frac1\eps W( u_\eps) - |\nabla \psi_\eps|\right)dx dt =0.
	\end{align}
\end{proposition}

The assertion~\eqref{eq:ilmanen 2} follows from the maximum principle applied either to the quotient 
(as is done in~\cite{ilmanen}) or to the difference of the two terms in the energy. 
It is the statement~\eqref{eq:discrepancy 1} which is highly non-trivial. In fact, this is one of the major achievements of the fundamental contribution~\cite{ilmanen}. Ilmanen's proof is based on a monotonicity formula for the Allen--Cahn equation in analogy to Huisken's celebrated monotonicity formula~\cite{huisken} for the mean curvature flow. We refer to~\cite{ilmanen} for more detail.

\begin{proof}	
	Since \eqref{eq:ilmanen 2}, \eqref{eq:ilmanen 3}, and \eqref{eq:discrepancy 1} are derived in Ilmanen's work \cite{ilmanen}, we will only show how to post-process the statement \eqref{eq:discrepancy 1} to the second version \eqref{eq:discrepancy 2}. 
	We first make the trivial observation
	\begin{align*}
			\limsup_{\eps\downarrow 0}& \left| \int_{\R^d\times(0,\infty)} \zeta \left(\sqrt{\eps} |\nabla u_\eps| - \sqrt{2W(u_\eps)} \right)^2dx dt \right| 
			\\
			&\leq 	\limsup_{\eps\downarrow 0}\int_{\R^d\times(0,\infty)}  |\zeta| \left(\sqrt{\eps} |\nabla u_\eps| - \sqrt{2W(u_\eps)} \right)^2dx dt. 
	\end{align*}
	Once we show that the right-hand side is non-positive, we are done. So we may assume $\zeta \geq 0$ and only need to prove the upper bound
	\begin{align}\label{eq:proof of ilmamen lemma upper bound}
		\limsup_{\eps\downarrow 0}\int_{\R^d\times(0,\infty)}  \zeta \left(\sqrt{\eps} |\nabla u_\eps| - \sqrt{2W(u_\eps)} \right)^2dx dt \leq 0.
	\end{align} 
	Using \eqref{eq:ilmanen 2} in the form of $  \sqrt{2W(u_\eps)} \geq \eps |\nabla u_\eps|$ we can estimate pointwise
	\begin{align*}
		 \frac12\left(\sqrt{\eps} |\nabla u_\eps| - \sqrt{2W(u_\eps)} \right)^2
		 =& \frac{\eps}2 |\nabla u_\eps|^2 + \frac1\eps W(u_\eps) - \sqrt{2W(u_\eps)} |\nabla u_\eps|
		 \\ \leq & \frac{\eps}2 |\nabla u_\eps|^2 + \frac1\eps W(u_\eps) - \eps |\nabla u_\eps|^2
		 = \frac1\eps W(u_\eps)- \frac\eps2 |\nabla u_\eps|^2.
	\end{align*}
	By \eqref{eq:discrepancy 1}, when tested with $\zeta$, the right-hand side vanishes in the limit $\eps \downarrow 0$. 
	Finally, by the chain rule $|\nabla \psi_\eps| = \sqrt{2W(u_\eps)} |\nabla u_\eps|$, we see that the first identity in \eqref{eq:discrepancy 2} holds even for fixed $\eps$. This concludes the proof of the proposition.
\end{proof}

The following lemma provides a sharp lower bound for the energy in terms of the mass of the time slices of $\omega_\eps =\lim_\eps \omega_\eps$ constructed in Proposition \ref{prop:compactness}.

\begin{lemma}\label{lemma:energy lsc}
	The measures $\omega$ and $\mu$ from Proposition \ref{prop:compactness} can be written as $\omega=\left(\mathcal{L}^1\llcorner (0,\infty)\right) \otimes (\omega_t)_t $ and $\mu=\left(\mathcal{L}^1\llcorner (0,\infty)\right) \otimes (\omega_t)_t \otimes (\lambda_{x,t})_{x,t} $, where $(\omega_t)_t$ is a weakly-$\ast$ $\mathcal{L}^1$-measurable family of
	finite Radon measures
	and $(\lambda_{x,t})_{x,t} $ is a weakly-$\ast$ $\mathcal{L}^1 \otimes (\omega_t)_t$-measurable family of
	Radon probability measures. Furthermore, for a.e.\ $t\in(0,\infty)$ it holds
	\begin{align}\label{eq:comp mass at t}
		\liminf_{\eps\downarrow0}E_\eps(u_\eps(\,\cdot\,,t)) \geq \omega_t(\R^d).
	\end{align}
\end{lemma}

\begin{proof}
	For each $\eps>0$, the real function $t\mapsto e_\eps(t):= E_\eps(u_\eps(\,\cdot\,,t))$ is monotonically non-increasing, cf.~\eqref{eq:EDI naive}. These functions are also uniformly bounded $0\leq e_\eps(t) \leq E_0<\infty$ for all $t\in(0,\infty)$ and all $\eps>0$. Hence, by Helly's theorem we find a subsequence $\eps\downarrow0$ and a (monotonically non-increasing) real function $t\mapsto e(t)$ such that
	\begin{align}\label{eq:proof comp mass at t1}
	e_\eps(t) \to e(t) \quad \text{for a.e.\ } t \in (0,\infty).
	\end{align}
	Now we want to use a cylindrical test function. Let $\eta \in C_{c}((0,\infty))$ with $\eta \geq 0$ 
	and $\zeta \in C_{c}(\R^d)$ with $\zeta \in [0,1]$. 
	Using the definition of $\omega_\eps$, $\zeta \leq 1$, and the non-negativity \eqref{eq:ilmanen 2}, we obtain
	\begin{align}
		\int_{\R^d\times (0,\infty)} \zeta(x)\eta(t) \,d\omega_\eps(x,t)
		 \leq \int_0^\infty \eta(t) e_\eps(t) \,dt.
	\end{align}
	To pass to the limit $\eps\downarrow 0$, we use the compactness~\eqref{eq:compactness lambda} 
	for the left-hand side. For the right-hand side, we use~\eqref{eq:proof comp mass at t1} and the dominated convergence theorem, which yields
	\begin{align}
	\label{eq:localized omega leq e}
	\int_{\R^d \times (0,\infty)}\zeta(x)\eta(t) \,d\omega(x,t)
		\leq \int_0^\infty \eta(t) e(t)\,dt.
	\end{align}
	By monotone convergence this can be upgraded to 
	\begin{align}\label{eq:omega leq e}
		\int_{\R^d \times (0,\infty)} \eta(t) \,d\omega(x,t)
		\leq \int_0^\infty \eta(t) e(t)\,dt.
	\end{align}
	Since the right-hand side is finite for non-negative $\eta \in C_c((0,\infty))$, this means
	that the projection of~$\omega$ onto the time variable is a Radon measure.
	In particular, we can disintegrate $\omega$, i.e., there exists a Radon measure $\sigma$ on $(0,\infty)$ and a weakly-$\ast$ $\sigma$-measurable family of Radon probability measures $(\tau_t)_t$ on $\R^d$ such that $\omega = \sigma \otimes (\tau_t)_t$. Then \eqref{eq:omega leq e} means that $\sigma \leq e \mathcal{L}^1\llcorner(0,\infty)$, so that by the Radon--Nikod\'ym theorem, there exists an $\mathcal{L}^1$-measurable function $\rho$ on $(0,\infty)$ with $\rho(t) \leq e(t)$ for a.e.~$t\in (0,\infty)$ such that $\sigma = \rho \mathcal{L}^1 \llcorner (0,\infty)$ and hence we obtain the claimed representation for $\omega$ with $ \omega_t = \rho(t) \tau_t$. 
	Note that also $\omega_t$ is weakly-$\ast$ measurable and, in addition, satisfies the claimed inequality \eqref{eq:comp mass at t}. Just another disintegration on the $p$-variable for $\mu$ proves the claimed representation of $\mu$.
\end{proof}

The following innocent looking but crucial lemma justifies the curious choice \eqref{eq:def Veps} for our approximate velocity $V_\eps$. It states that $V_\eps$ is uniformly square integrable with respect to our energy measure $\omega_\eps$, and that the pair $(\psi_\eps, V_\eps)$ solves an approximate transport equation 
\begin{align}\label{eq:transport omega}
	\partial_t \psi_\eps + V_\eps \omega_\eps\approx0.
\end{align}
 We call this a transport equation, since by the definition \eqref{eq:def lambda_eps} of $\omega_\eps$, this operator simply reads
\begin{align}
	\partial_t \psi_\eps + V_\eps \omega_\eps = \partial_t \psi_\eps 
	+ \Big(V_\eps \frac{\nabla \psi_\eps}{|\nabla \psi_\eps|} 
	  \cdot \nabla \Big)\psi_\eps,
\end{align}
so that \eqref{eq:transport omega} means that $\psi_\eps$ is transported by the velocity vector field $V_\eps\frac{\nabla \psi_\eps}{|\nabla \psi_\eps|}$.

\begin{lemma}\label{lemma:Veps}
	If the initial conditions $u_{\eps,0}$ are well-prepared in the sense of 
	Ilmanen~\eqref{eq:ilmanen initial 1} and~\eqref{eq:ilmanen initial 2}, 
	then \begin{align}\label{eq:Veps L2}
		 \int_{\R^d\times(0,\infty)} V_\eps^2 d\omega_\eps \leq E_\eps(u_{\eps,0}).
	\end{align}
	Furthermore, for any test function $\zeta \in C_c(\R^d\times(0,\infty))$
	\begin{align}\label{eq:Veps is almost velocity}
		\lim_{\eps\downarrow0} \left(	\int_{\R^d\times(0,\infty)} \zeta V_\eps \,d\omega_\eps +\int_{\R^d\times(0,\infty)} \zeta \p_t \psi_\eps \,dxdt \right) =0.
	\end{align}
\end{lemma}
\begin{proof}
	We first derive the $L^2$-estimate \eqref{eq:Veps L2}. By definitions \eqref{eq:def mu_eps} and \eqref{eq:def Veps}, and the chain rule we have for every $T \in (0,\infty)$
	\begin{align*}
		\int_{\R^d {\times} (0,T)} V_\eps^2 \,d\omega_\eps 
		&= \int_{\{W(u_\eps)\neq 0\} \cap (\R^d {\times} (0,T))} 
		\eps^2 \frac{(\partial_t u_\eps)^2}{2W(u_\eps)} \sqrt{2W(u_\eps)} |\nabla u_\eps| \,dx dt
		\\&
		= \int_{\{W(u_\eps)\neq 0\} \cap (\R^d {\times} (0,T))} 
		\eps (\partial_t u_\eps )^2 \frac{\sqrt{\eps} |\nabla u_\eps|}{\frac1{\sqrt{\eps}}\sqrt{2W(u_\eps)}} \,dx dt.	
	\end{align*}
	Miraculously enough, by \eqref{eq:ilmanen 2} in Ilmanen's Proposition (Proposition \ref{prop:ilmanen}), 
	the factor $\smash{\frac{\sqrt{\eps} |\nabla u_\eps|}{\frac1{\sqrt{\eps}}\sqrt{2W(u_\eps)}}}$ 
	is less or equal to $1$, which shows for every $T \in (0,\infty)$
	\begin{align}
	\label{eq:uniformDissipationEstimateAux}
			\int_{\R^d {\times} (0,T)}
			V_\eps^2 \,d\omega_\eps  \leq \int_{\R^d {\times} (0,T)}
			\eps (\partial_t u_\eps )^2 \,dxdt.
	\end{align}
	Now we conclude the argument for \eqref{eq:Veps L2} by the energy-dissipation inequality \eqref{eq:EDI}
	and taking the limit $T \nearrow \infty$.
	
	Next, we want to derive \eqref{eq:Veps is almost velocity}. We fix $\zeta \in C_c(\R^d\times(0,\infty))$ and again plug in the definitions \eqref{eq:def mu_eps} and \eqref{eq:def Veps}, and then apply the chain rule in form of $|\nabla \psi_\eps| = \sqrt{2W(u_\eps)} |\nabla u_\eps|$, and add zero to obtain
	\begin{align*}
			\int_{\R^d\times(0,\infty)} \zeta V_\eps \,d\omega_\eps 
			=&- \int_{\{W(u_\eps)\neq 0\}} \zeta \eps \frac{\partial_t u_\eps}{\sqrt{2W(u_\eps)}} |\nabla \psi_\eps| \,dx dt
			\\=&- \int_{\{W(u_\eps)\neq 0\}} \zeta \eps \partial_t u_\eps |\nabla u_\eps| \,dx dt
			\\= &-\int_{\{W(u_\eps)\neq 0\}} \zeta \sqrt{\eps} \partial_t u_\eps \left(\sqrt{\eps} |\nabla u_\eps| - \frac1{\sqrt{\eps}}\sqrt{2W(u_\eps)} \right) \,dx dt
			\\&-\int_{\{W(u_\eps)\neq 0\}} \zeta \partial_t u_\eps \sqrt{2W(u_\eps)} \,dx dt.
	\end{align*}
	Now we may drop the restriction ${\{W(u_\eps)\neq 0\}}$ in the last integral and recognize the appearing product as a derivative $\partial_t u_\eps \sqrt{2W(u_\eps)} = \partial_t \psi_\eps$. Rearranging terms and applying Cauchy--Schwarz yields
	\begin{align*}
		&\left|\int_{\R^d\times(0,\infty)} \zeta V_\eps \,d\omega_\eps +\int_{\R^d\times(0,\infty)} \zeta \partial_t\psi_\eps \,dx dt \right|
		\\&\leq \left(\int_{\R^d\times(0,\infty)}  \eps (\partial_t u_\eps)^2 \,dx dt \right)^\frac12\left(\int_{\R^d\times(0,\infty)} \zeta^2\left(\sqrt{\eps} |\nabla u_\eps| - \frac1{\sqrt{\eps}}\sqrt{2W(u_\eps)} \right)^2 \,dx dt \right)^\frac12.
	\end{align*}
	By the energy-dissipation inequality \eqref{eq:EDI}, the first factor is uniformly bounded in $\eps$, while the second factor  vanishes in the limit $\eps\downarrow 0$ thanks to the equipartition of energy \eqref{eq:discrepancy 2}. This concludes the proof.
\end{proof}

Let us first state the remaining two ingredients for Theorem \ref{theo:Existence}. 
The first proposition will finish our analysis of the velocity. In it we construct a velocity $V$ for the limit, which solves the ``transport equation''
\begin{align}\label{eq:transport limit}
	\partial_t \psi + V \omega =0.
\end{align}
Note that this is not quite a transport equation, as we only know that $\omega\geq |\nabla \psi|$ in the sense of measures. 
Geometrically speaking, the transport equation \eqref{eq:transport limit} means that $V$ is the normal velocity of the evolving family of finite perimeter sets $\Omega(t)$---with the novelty in this work that it might also have a support away from the reduced boundary $\partial^\ast \Omega(t)$.
Furthermore, the proposition provides a sharp inequality between the dissipation term $\int \eps (\partial_t u_\eps)^2dxdt $ appearing in the energy-dissipation inequality of the Allen--Cahn equation and the corresponding term $\int V^2 d\omega$ in the energy-dissipation inequality for the mean curvature flow \eqref{eq:DeGiorgiInequality},  where $V$ is precisely the velocity field in the transport equation \eqref{eq:transport limit}

\begin{proposition}\label{prop:Vsquared}
	For $\eps>0$, let $u_\eps$ be the solution to the Allen--Cahn equation \eqref{eq:allen cahn} with well-prepared initial conditions in the sense of \eqref{eq:ilmanen 2} and \eqref{eq:ilmanen 3}, let $V_\eps$ be defined by \eqref{eq:def Veps}, 
	and let~$\mu$ and~$\omega$ be given by Proposition \ref{prop:compactness}.
	Then there exists an $\omega$-measurable function $V \colon \R^d \times (0,\infty) \to \R$,  which is the normal velocity of $\chi$ in the precise sense of \eqref{eq:evolPhase}. 
	Furthermore, for a.e.\ $T \in (0,\infty)$
	we have the following lower semi-continuity-type inequality
	\begin{align}\label{eq:lsc V}
		\liminf_{\eps\downarrow0} \frac12 
		\int_{\R^d {\times} (0,T)}
		V_\eps^2 \,d\omega_\eps
		\geq \frac12  \int_{\R^d {\times} (0,T)}  V^2 \,d\omega.
	\end{align}
	Moreover, for a.e.\ $T \in (0,\infty)$
	the velocity $V$ provides a sharp lower bound for the dissipation functional
	\begin{align}\label{eq:lower bound dissipation dtu}
		\liminf_{\eps\downarrow0} \frac12 \int_{\R^d {\times} (0,T)}
		\eps (\partial_t u_\eps )^2 \,dx dt 
		\geq \frac12  \int_{\R^d {\times} (0,T)}  V^2 \,d\omega.
	\end{align}
\end{proposition}

	The following second proposition gives us a sharp inequality between the gradient-terms in the energy-dissipation inequalities. More precisely, but still only formally, it gives a sharp lower bound for the term $\int \frac1\eps \left(\delta E_\eps\right)^2 dx dt $ in terms of the corresponding term in the sharp-interface limit $\int \left( \delta \omega\right)^2 d\omega$.
	Although the proof is contained in Ilmanen's work \cite{ilmanen}, we reproduce the statement and will, for the reader's convenience, give a short proof at the end of this section.

	\begin{proposition}[Ilmanen \cite{ilmanen}]\label{prop:Hsquared}
		For $\eps>0$, let $u_\eps$ be the solution to the Allen--Cahn equation~\eqref{eq:allen cahn} with well-prepared initial conditions in the sense of Ilmanen \eqref{eq:ilmanen 2} and \eqref{eq:ilmanen 3}, let $\mathbf{H}_\eps$ be defined by~\eqref{eq:def Heps}, and let $\mu$ and $\omega$ be given by Proposition \ref{prop:compactness}.
		Then there exists a $\mu$-measurable vector field $\mathbf{H} \colon \R^d \times (0,\infty) \to \R^d$,  which is the mean curvature vector of the oriented space-time varifold $\mu$ in the precise sense of \eqref{eq:weakCurvature}. 
		Furthermore, for a.e.\ $T \in (0,\infty)$
		we have the following sharp lower bound
		\begin{align}\label{eq:lsc H}
			\liminf_{\eps\downarrow0} \frac12 \int_{\R^d {\times} (0,T)}
			\frac1\eps |\mathbf{H}_\eps|^2 \,dx dt
		\geq \frac12  \int_{\R^d {\times} (0,T)} |\mathbf{H}|^2 \,d\omega.
		\end{align}
	\end{proposition}

\begin{proof}[Proof of Proposition \ref{prop:Vsquared}]
	We start by proving the existence of $V$. 
	It is enough to show the absolute continuity 
	\begin{align}\label{eq:dtchi abs cont}
		\sigma \partial_t\chi =	\partial_t \psi  \ll \omega.
	\end{align}
	Indeed, then we can define $V$ as the Radon--Nikod\'ym  derivative $\frac{ d(  \partial_t \psi)}{d\omega}$. Then automatically, $V$ satisfies \eqref{eq:evolPhase} for any $T<\infty$ and any test function $\zeta \in C_0^1(\R^d\times(0,T))$. 
	To now post-process this to the actual statement, we need to approximate in \eqref{eq:evolPhase} a given test function $\zeta \in C_c^1(\R^d\times[0,\infty))$ with a sequence $\zeta_n \in C_0^1(\R^d\times(0,T))$. This is standard: take for concreteness $\zeta_n:= \eta_n \zeta$ with $\zeta$ the piecewise linear cutoff function $\eta(t)=t/n$ for $t\in[0,1/n]$, $\eta(t)=(T-t)/n$ for $t\in[T-1/n,T]$ and $\eta(t)=1$ otherwise. Plugging $\zeta_n$ into \eqref{eq:evolPhase} and applying the product rule yields
	\begin{align}
		\sigma n\int_{T-\frac1n}^{T}\int_{\R^d} \chi\zeta \,dx dt- 
		\sigma n\int_0^{\frac1n} \int_{\R^d} \chi\zeta \,dxdt 
		= \sigma\int_{0}^{T} \int_{\R^d} \chi \eta_n \partial_t\zeta \,dxdt
		- \int_{0}^{T} \int_{\R^d} V \eta_n \zeta \,d\omega_tdt.
	\end{align}
	By \eqref{eq:compactness cont vol}, the two average integrals on the left-hand side converge to the desired limits. The terms on the right-hand side converge by absolute continuity.
	
	In order to show \eqref{eq:dtchi abs cont}, assume that $U\subset \R^d\times(0,\infty)$ is open such that 
	\begin{align}\label{eq:lambda small}
		\omega(U) = \mu(U\times \S^{d-1}) < \eps.
	\end{align}
	Now we claim that there exists a continuous function $\delta \colon [0,\infty) \to [0,\infty)$ with $\delta(0)=0$ such that 
	\begin{align}\label{eq:dtchi small}
		| \partial_t\psi| (U)\leq \delta(\eps).
	\end{align} 
	Indeed, for any $\zeta \in C_c^1(U)$, using \eqref{eq:compactness psi}, integration by parts, \eqref{eq:Veps is almost velocity}, and Cauchy--Schwarz, we may write 
	\begin{align*}
		\int_{\R^d\times(0,\infty)}  \psi \partial_t \zeta  \,dx dt
		=& \lim_{\eps \downarrow0} \int_{\R^d\times(0,\infty)} \psi_\eps \partial_t \zeta  \,dx dt
		\\=& \lim_{\eps \downarrow0} \int_{\R^d\times(0,\infty)} \zeta V_\eps \,d\omega_\eps
		\\ \leq & \left( \liminf_{\eps\downarrow0} \int_{\R^d\times(0,\infty)} V_\eps^2 \,d\omega_\eps\right)^\frac12
		 \left( \lim_{\eps\downarrow0} \int_{\R^d\times(0,\infty)}  \zeta^2 \,d \omega_\eps\right)^\frac12.
	\end{align*}
	Applying the compactness \eqref{eq:Veps L2} and \eqref{eq:compactness lambda} (and then using the symmetry $\zeta \mapsto -\zeta$ to get the modulus on the left-hand side)
	\begin{align}\label{eq:dtchi zetasquared lambda}
		\left|\int_{\R^d\times(0,\infty)}  \psi  \partial_t\zeta \,dx dt\right|
			\leq \left(\liminf_{\eps\downarrow0} \int_{\R^d\times(0,\infty)} V_\eps^2 \,d\omega_\eps\right)^\frac12 
			\left(\int_{\R^d\times(0,\infty)} \zeta^2 \,d \omega\right)^\frac12.
	\end{align}
	Taking the supremum over all such $\zeta$, this yields in particular \eqref{eq:dtchi small} 
	with module $\delta(s) = C \sqrt{s}$ for some constant $C<\infty$ due to \eqref{eq:Veps L2}. 
	Hence, we have proven the absolute continuity \eqref{eq:dtchi abs cont} and can now define $V$ as the density. Finally, turning back to \eqref{eq:dtchi zetasquared lambda} once more, the Riesz representation theorem yields the $L^2$-bound \eqref{eq:lsc V}.
	
	The sharp lower bound on the dissipation~\eqref{eq:lower bound dissipation dtu} now follows
	from the already established estimates~\eqref{eq:lsc V} and~\eqref{eq:uniformDissipationEstimateAux}.
	%
\end{proof}

\begin{proof}[Proof of Proposition \ref{prop:Hsquared}]
	Although this statement is contained in \cite{ilmanen}, we give a short argument here for the reader's convenience.
	To this end, we fix $T \in (0,\infty]$ and
	$\xi\in C_c(\R^d\times(0,T);\R^d)$.
	
	First, we apply the trivial inequality $\frac12 |a|^2 \geq a\cdot b - \frac12 |b|^2 $ with $a=\eps \Delta u_\eps - \frac1{\eps}W'(u_\eps)$ and $b=\eps \xi \cdot \nabla u_\eps$ to obtain
	\begin{align*}
		\frac12 \int_{\R^d\times(0,\infty)} \frac1\eps|\mathbf{H}_\eps|^2  \,dx dt
		=& \frac12 \int_{\R^d\times(0,\infty)} \left(\eps \Delta u_\eps - \frac1{\eps}W'(u_\eps)\right)^2 \frac1\eps \,dx dt
		\\ \geq& \int_{\R^d\times(0,\infty)}  \left(\eps \Delta u_\eps - \frac1{\eps}W'(u_\eps) \right) \left( \eps \xi\cdot \nabla u_\eps \right) \frac 1\eps \,dx dt
		\\&- \frac12 \int_{\R^d\times(0,\infty)} (\eps \xi\cdot \nabla u_\eps)^2 \frac1\eps \,dx dt.
	\end{align*}
	We first note that the second right-hand side term can be treated as follows
	\begin{align*}
		- \frac12 \int_{\R^d\times(0,\infty)} (\eps \xi\cdot \nabla u_\eps)^2 \frac1\eps \,dx dt \geq -\frac12 \int_{\R^d\times(0,\infty)}  |\xi|^2 \eps |\nabla u_\eps|^2 \,dx dt,
	\end{align*}
	which converges to $-\frac12 \int |\xi|^2 d\omega$ as $\eps \downarrow 0$ by the equipartition \eqref{eq:discrepancy 1} and the convergence \eqref{eq:compactness lambda} of~$\omega_\eps$.
	For the first right-hand side term, a straight-forward integration by parts yields
	\begin{align*}
		\int_{\R^d\times(0,\infty)}  \left(\eps \Delta u_\eps - \frac1{\eps}W'(u_\eps) \right) \left( \eps \xi\cdot \nabla u_\eps \right) \frac 1\eps \,dx dt
		= \int_{\R^d\times(0,\infty)} \mathbf{T}_\eps \colon \nabla \xi  \,dx dt,
	\end{align*}
	where $\mathbf{T}_\eps$ denotes the energy-stress tensor given by
	\begin{align*}
		\mathbf{T}_\eps = \left( \frac\eps2 |\nabla u_\eps|^2 + \frac1\eps W(u_\eps) \right) I_d - \eps \nabla u_\eps \otimes \nabla u_\eps.
	\end{align*}
	Now clearly
	\begin{align*}
		\int_{\R^d\times(0,\infty)} \mathbf{T}_\eps \colon \nabla \xi  \,dx dt
		= &\int_{\R^d\times(0,\infty)} (\nabla \cdot \xi) \left( \frac\eps2 |\nabla u_\eps|^2 + \frac1\eps W(u_\eps) \right) \,dx dt 
		\\&- \int_{\R^d\times(0,\infty)}  \frac{\nabla u_\eps}{|\nabla u_\eps|} \cdot \nabla \xi \frac{\nabla u_\eps}{|\nabla u_\eps|} \eps |\nabla u_\eps|^2 \,dx dt.
	\end{align*}
	Here, the first term converges to the desired limit $\int (\nabla \cdot \xi) d\omega$ thanks to the compactness \eqref{eq:compactness lambda} and the vanishing of the discrepancy \eqref{eq:discrepancy 2}. Since $\frac{\nabla u_\eps}{|\nabla u_\eps|} = \frac{\nabla \psi_\eps}{|\nabla \psi_\eps|}$ on the set $\{\nabla \psi_\eps \neq 0\} \supset \{ \nabla u_\eps \neq 0\}$, the second term can be written as
	\begin{align*}
		 \int_{\R^d\times(0,\infty)}  \frac{\nabla u_\eps}{|\nabla u_\eps|} \cdot \nabla \xi \frac{\nabla u_\eps}{|\nabla u_\eps|} \eps |\nabla u_\eps|^2 \,dx dt
		= & \int_{\R^d\times (0,\infty)\times\S^{d-1}} p \cdot \nabla \xi p \,d\mu_\eps(x,t,p) 
		\\&+ \int_{\R^d\times(0,\infty)}  \frac{\nabla u_\eps}{|\nabla u_\eps|} \cdot \nabla \xi \frac{\nabla u_\eps}{|\nabla u_\eps|}  \left( \eps |\nabla u_\eps|^2 - |\nabla \psi_\eps| \right) \,dx dt.
	\end{align*}
	The first right-hand side term converges to $\int p\cdot \nabla \xi p \,d\mu $ as $\eps \downarrow 0$ due to the compactness \eqref{eq:compactness mu}, while the second right-hand side term is estimated using Cauchy--Schwarz 
	\begin{align*}
		&\int_{\R^d\times(0,\infty)} |\nabla \xi|\left| \eps |\nabla u_\eps|^2 - |\nabla \psi_\eps| \right| \,dx dt 
		\\&=	\int_{\R^d\times(0,\infty)} |\nabla \xi|\sqrt{\eps} |\nabla u_\eps|\left| \sqrt{\eps} |\nabla u_\eps| - \sqrt{2W(u_\eps)} \right| \,dx dt 
		\\&\leq \left(\int_{\R^d\times(0,\infty)} |\nabla \xi| \eps |\nabla u_\eps|^2  \,dx dt \right)^\frac12
		\left(\int_{\R^d\times(0,\infty)}|\nabla \xi| \left( \sqrt{\eps} |\nabla u_\eps| - \sqrt{2W(u_\eps)} \right)^2 \,dx dt \right)^\frac12,
	\end{align*}
	which vanishes in the limit $\eps \downarrow 0$ thanks to the equipartition of energy \eqref{eq:discrepancy 2}, the uniform bound on the energy \eqref{eq:EDI} and of course the fact that $\sup |\nabla \xi|$ has compact support.
	
	Therefore, we conclude that
	\begin{align*}
		\liminf_{\eps\downarrow 0} \frac12 \int_{\R^d\times(0,\infty)} \frac1\eps|\mathbf{H}_\eps|^2 \,dx dt 
		\geq \int_{\R^d\times(0,\infty)\times\S^{d-1}} \left( I_d - p\otimes p\right) \colon \nabla \xi \,d\mu - \frac12 \int_{\R^d\times(0,\infty)} |\xi|^2 \,d\omega
	\end{align*}
	for all $T \in (0,\infty]$ and all $\xi\in C_c(\R^d\times(0,T);\R^d)$.
	
	Let first $T=\infty$. Since $\xi$ was arbitrary and the left-hand side is finite, 
	by the Riesz representation theorem, there exists a $\omega$-measurable vector field $\mathbf{H} \colon \R^d\times(0,\infty)\to \R^d$, which is the mean curvature vector of $\mu$ in the precise sense of \eqref{eq:weakCurvature}. 
	Second, for a.e.\ $T \in (0,\infty)$ since again~$\xi$ was arbitrary
	\begin{align*}
			\liminf_{\eps\downarrow 0} \frac12 \int_{\R^d {\times} (0,T)}
			\frac1\eps|\mathbf{H}_\eps|^2  \,dx dt
			\geq \frac12 \int_{\R^d {\times} (0,T)} |\mathbf{H}|^2 \,d\omega. 
	\end{align*}
	This concludes the proof.
\end{proof}

We may finally combine the results of this section to
provide a proof of the first main result of the present work.

\begin{proof}[Proof of Theorem~\ref{theo:Existence}]
		\textit{Step 1: Convergence to solution.}
			The main part of the theorem now follows directly from the previous propositions. Indeed, the compactness follows from Proposition \ref{prop:compactness}. 
			To check that the limit is a De Giorgi type varifold solution, we need to check the four items in Definition \ref{def:twoPhaseDeGiorgiVarifoldSolutions}. First, by Proposition \ref{prop:Vsquared}, our function $V$ satisfies the defining equation \eqref{eq:evolPhase} of the normal velocity.
			Proposition \ref{prop:Hsquared} guarantees that $\mathbf{H}$ is the generalized mean curvature vector field in the sense of \eqref{eq:weakCurvature}.
			
			To verify the De Giorgi type inequality \eqref{eq:DeGiorgiInequality}, we fix $T\in(0,\infty)$ and wish to pass to the limit $\eps\downarrow0$ in the time-integrated version of \eqref{eq:EDI}:
			\begin{align}\label{eq:proof thm1 de giorgi for allen cahn}
					 E_\eps(u_\eps(\,\cdot\,,T)) + \frac12 \int_{\R^d\times(0,T)}  \eps(\partial_tu_\eps)^2 \,dx\,dt + \frac12\int_{\R^d\times(0,T)}  \frac1\eps \left| \mathbf{H}_\eps\right|^2 \,dx\,dt 
					 = E_\eps(u_{\eps,0}).
			\end{align}
			The right-hand side term converges to the desired limit $\omega_0(\R^d)$ 
			by assumption~\eqref{eq:ilmanen initial 3}. 
			Now we turn to the left-hand side terms, for which we only need the lower-semi continuity, which was established in Lemma \ref{lemma:energy lsc}, Proposition \ref{prop:Vsquared}, and Proposition \ref{prop:Hsquared}, respectively.
			
			The compatibility \eqref{eq:compatibility} follows from passing to the limit $\eps\downarrow0$ in the  (time-integrated) linear relation 
			\begin{align*}
				\int_{\R^d\times(0,\infty)} \xi(x,t) \cdot \nabla\psi_\eps(x,t) \,dx\,dt
				= \int_{\R^d {\times} \S^{d-1}\times(0,\infty)} \xi(x,t) \cdot p \,d \mu_\eps(x,t,p)
			\end{align*}	
			for all $\xi \in C_c^1(\R^d\times[0,\infty))$ and then localizing in time.
			
			\textit{Step 2: Additional properties.} 
			As stated in the theorem, let us now assume that the second moments of the initial energy measure are bounded, i.e.,
			\begin{align}\label{eq:auxiliary proof second moment}
				\sup_{\eps>0} \int_{\R^d} |x|^2 \left( \frac\eps2 |\nabla u_{\eps,0}(x)|^2 + \frac1\eps W(u_{\eps,0}(x))\right) \,dx <\infty.
			\end{align}
			We want to verify the De Giorgi type inequality \eqref{eq:DeGiorgiInequality} for intervals $(s,t)$ instead of $(0,T)$ for all $t\in (0,\infty)$ and a.e.~$s<t$. 
			We first observe that the $\eps$-version \eqref{eq:proof thm1 de giorgi for allen cahn} holds with $(0,T)$ replaced by $(s,t)$ for all $0\leq s<t<\infty$.
			Then we may pass to the limit on the left-hand side just as before (after realizing that the analogous lower semi-continuity statements also hold on the time interval $(s,t)$ for a.e.~$s<t$). 
			Now we only need to argue that the right-hand side actually converges: $\lim_{\eps\downarrow0}E_\eps(u_\eps(\,\cdot\,,s)) = \omega_s(\R^d)$.
			This follows as soon as we know that the energy measures have uniformly bounded second moments, i.e.,
			\begin{align}\label{eq:proof thm 1 2nd moments}
				\sup_{\eps>0} \int_{\R^d} |x|^2 \left( \frac\eps2 |\nabla u_\eps(x,t)|^2 + \frac1\eps W(u_\eps(x,t)) \right) \,dx <\infty
			\end{align}
			for a.e.~$t\in (0,\infty)$. (However, not uniformly in $t$.)
			Indeed, \eqref{eq:proof thm 1 2nd moments} implies that each measure $(\omega_\eps)_t$ has uniformly bounded second moments and hence is tight as $\eps \downarrow0$, so that we may test the convergence \eqref{eq:ilmanen slicing} with the constant test function $1$.
			
			To show \eqref{eq:proof thm 1 2nd moments}, we compute, dropping the arguments of $u_\eps$,
			\begin{align*}
				\frac{d}{dt} \int_{\R^d} (1+ |x|^2)  \left( \frac\eps2 |\nabla u_\eps|^2 + \frac1\eps W(u_\eps) \right) dx 
				&= \int_{\R^d}(1+|x|^2) \left( \nabla \cdot (\eps \partial_t u_\eps \nabla u_\eps)- \eps (\partial_t u_\eps)^2 \right) dx
				\\&= \int_{\R^d} (2x \cdot \nabla u_\eps) \eps \partial_t u_\eps - (1+|x|^2) \eps (\partial_t u_\eps)^2 \,dx
				\\&\leq -\eps \int_{\R^d} (\partial_tu_\eps)^2 \,dx + \int_{\R^d} |x|^2 \eps |\nabla u_\eps|^2\,dx
				\\&\leq  2 \int_{\R^d} (1+|x|^2) \left( \frac\eps2 |\nabla u_\eps|^2 + \frac1\eps W(u_\eps)\right)dx
			\end{align*}
			where we have used Young's inequality and have simply used $ -(\partial_t u_\eps)^2\leq 0 $, $|x|^2 \leq (1+|x|^2)$, and $0\leq W(u_\eps) $ in the last step.
			Now \eqref{eq:proof thm 1 2nd moments} follows from Gronwall's inequality, \eqref{eq:auxiliary proof second moment}, and \eqref{eq:ilmanen initial 2}.
			
			The integrality of limits of the Allen--Cahn equation $\frac1\sigma \hat \mu$ was shown by Tonegawa \cite[Theorem 2.2]{tonegawa_integrality}. We note that, in dimensions $d=2,3$, one can alternatively apply the work of R\"oger and Sch\"atzle. Indeed, \eqref{eq:DeGiorgiInequality} implies that $\int |\mathbf{H}(\cdot,t)|^2 \,d\omega_t <\infty$ for a.e.~$t$. Fixing such a $t$, we may apply \cite[Theorem 5.1]{roegerschatzle} along a subsequence $\eps_n\downarrow0$, and along a sequence of times $t_n \to t$. 
\end{proof}

\section{Weak-strong uniqueness for De~Giorgi type varifold solutions}
\label{sec:weakStrongUniqueness}
Let $\Tstrong \in (0,\infty)$ be a finite time horizon,
and let $(\mathscr{A}(t))_{t \in [0,\Tstrong]}$
be a smoothly evolving family of bounded domains in~$\R^d$
whose associated family of interfaces $(\mathcal{I}(t))_{t \in [0,\Tstrong]}$
is assumed to evolve smoothly by their mean curvature. Given this data,
we will first perform some auxiliary constructions before we provide
a rigorous definition for the data entering the distance
measures~\eqref{eq:relativeEntropyIntro} and~\eqref{eq:bulkErrorIntro}. 

Denote for each~$t\in[0,\Tstrong]$ by $n_{\mathcal{I}}(t)$ the unit normal vector along $\mathcal{I}(t)$
pointing inside~$\mathscr{A}(t)$. By the tubular neighborhood theorem
and the assumed smoothness of the data, there exists a small scale $r_c \in (0,1)$
such that the map
\begin{align}
\label{eq:diffeoTubNbhd}
\Psi \colon \bigg(\bigcup_{t \in [0,\Tstrong]} \mathcal{I}(t) \times \{t\} \bigg)
\times (-2r_c,2r_c) &\to \R^d \times [0,\Tstrong],
\\ \nonumber
\big((x,t),s\big) &\mapsto \big(x{+}sn_{\mathcal{I}(t)}(x),t \big)
\end{align}
is a smooth diffeomorphism onto its image $\mathrm{im}(\Psi) = 
\{(x,t) \colon \dist(x,\mathcal{I}(t)) < 2r_c\}$.
The projection of the inverse~$\Psi^{-1}$ onto the last component
is further denoted by $s_{\mathcal{I}}\colon\mathrm{im}(\Psi)\to(-2r_c,2r_c)$.
Note that $s_{\mathcal{I}}$ is a smooth space-time function encoding
a signed distance to the evolving family of interfaces $(\mathcal{I}(t))_{t \in [0,\Tstrong]}$
oriented in form of $\nabla s_{\mathcal{I}}(x,t)=n_{\mathcal{I}(t)}(x)$
for $(x,t)\in\mathrm{im}(\Psi)$. We also define a smooth space-time map
$P_{\mathcal{I}}\colon\mathrm{im}(\Psi)\to\smash{\bigcup_{t \in [0,\Tstrong]} \mathcal{I}(t)}$
by requiring that $P_{\mathcal{I}}(x,t)$ denotes the unique nearest point on the interface~$\mathcal{I}(t)$
for each $x\in\R^d$ satisfying $\dist(x,\mathcal{I}(t))<2r_c$.

We next fix a smooth cutoff $\kappa\colon\R\to [0,1]$ such that
$\kappa\equiv 1$ on $[-\smash{\frac{1}{2}},\smash{\frac{1}{2}}]$
and $\kappa\equiv 0$ on $\R\setminus (-1,1)$, and then define
a quadratic cutoff 
\begin{align}
\zeta(r) := (1 - r^2)\kappa(r^2), \quad r\in\R.
\end{align}
These choices allow to define two smooth space-time vector fields
\begin{align}
\label{def:xi}
\xi \colon \R^d{\times}[0,\Tstrong] &\to \{|x| {\leq} 1\},
&& (x,t) \mapsto \zeta\Big(\frac{s_{\mathcal{I}}(x,t)}{r_c}\Big)n_{\mathcal{I}(t)}(x),
\\ \label{def:B}
B \colon \R^d{\times}[0,\Tstrong] &\to \R^d
&& (x,t) \mapsto \zeta\Big(\frac{s_{\mathcal{I}}(x,t)}{r_c}\Big)
(-\Delta s_{\mathcal{I}})(P_{\mathcal{I}}(x,t))
n_{\mathcal{I}(t)}(x).
\end{align}
The interpretation to keep in mind is that~$\xi$ represents an extension
of the unit normal vector field along the smoothly evolving interfaces
$(\mathcal{I}(t))_{t \in [0,\Tstrong]}$, whereas~$B$ represents an
extension of the associated normal velocity. In particular, expressing the evolution
law of the interfaces $(\mathcal{I}(t))_{t \in [0,\Tstrong]}$
in form of $\partial_t s_{\mathcal{I}} + (B\cdot\nabla)s_{\mathcal{I}} = 0$
on $\supp\xi = \{(x,t)\colon\dist(x,\mathcal{I}(t)) \leq r_c\}$
it follows from the definition of the pair~$(\xi,B)$,
the properties of the cutoff function~$\zeta$,
and straightforward arguments that
\begin{align}
\label{eq:transportXi}
\big(
\partial_t\xi + (B\cdot\nabla)\xi + (\nabla B)^\mathsf{T}\xi
\big)(x,t)
&= O\big(\dist(x,\mathcal{I}(t))\big)
\\ \label{eq:transportLengthXi}
\big(
\xi\cdot\big(\partial_t\xi + (B\cdot\nabla)\xi\big)
\big)(x,t)
&= O\big(\mathrm{dist}^2(x,\mathcal{I}(t))\big)
\\ \label{eq:motionByMeanCurvature}
\big(
B\cdot\xi + \nabla\cdot\xi
\big)(x,t)
&= O\big(\dist(x,\mathcal{I}(t))\big)
\end{align}
for all $(x,t) \in \R^d{\times}[0,\Tstrong]$ such that $\dist(x,\mathcal{I}(t)) \leq r_c$.
Moreover, it obviously holds
\begin{align}
\min\{1,r_c^{-2}\mathrm{dist}^2(x,\mathcal{I}(t))\} \leq 1 - |\xi(x,t)| 
\end{align}
for all $(x,t) \in \R^d{\times}[0,\Tstrong]$.

Next, we fix a smooth truncation
$\bar\vartheta\colon \R\to [-1,1]$ satisfying $\bar\vartheta \equiv 1$
on $(-\infty,-1]$, $\bar\vartheta \equiv -1$ on $[1,\infty)$, $\bar\vartheta'<0$
in $(-1,1)$, $\bar\vartheta(0)=0$ and $|\bar\vartheta(r)|\geq |r|$ for all $r\in [-1,1]$.
With such a map, we then define a smooth weight
\begin{align}
\vartheta\colon\R^d{\times}[0,\Tstrong] \to [-1,1],
\quad (x,t) \mapsto \bar\vartheta\Big(\frac{s_{\mathcal{I}}(x,t)}{r_c}\Big).
\end{align}
It is straightforward to check that
\begin{align}
\label{eq:evolWeight}
\big(\partial_t\vartheta + (B\cdot\nabla)\vartheta\big)(x,t) &= 0,
\\
\label{eq:coercivity0}
\min\{1,r_c^{-1}\dist(x,\mathcal{I}(t))\} &\leq |\vartheta(x,t)|
\leq C\min\{1,r_c^{-1}\dist(x,\mathcal{I}(t))\}
\end{align}
for some $C>0$ and all $(x,t) \in \R^d{\times}[0,\Tstrong]$.

With all of this notation and properties in place, we are ready to
recall the definition of the two suitable measures quantifying 
the distance between a De~Giorgi type varifold solution
and a classical solution for two-phase mean curvature flow.
These consist of a relative entropy 
\begin{align}
\label{eq:relativeEntropy}
E[\mu,\chi|\mathscr{A}](t) := \int_{\R^d {\times} \S^{d-1}} 
1 - p \cdot \xi(x,t) \,d\mu_t(x,p) \geq 0, 
\quad t \in [0,\Tstrong],
\end{align}
as well as a bulk error
\begin{align}
\label{eq:bulkError}
E_{\mathrm{bulk}}[\chi|\mathscr{A}](t) := \sigma\int_{\R^d} 
|\chi_{A(t)}(x) {-} \chi_{\mathscr{A}(t)}(x)| |\vartheta(x,t)| \,dx \geq 0, 
\quad t \in [0,\Tstrong].
\end{align}
By the definition and properties of $\xi$ and $\vartheta$, note that it holds
\begin{align}
\label{eq:coercivity1}
\int_{\R^d {\times} \S^{d-1}} |p {-} \xi(\cdot,t)|^2 \,d\mu_t 
&\leq 2E[\mu,\chi|\mathscr{A}](t),
\\ \label{eq:coercivity2}
\int_{\R^d} \min\{1,r_c^{-2}\mathrm{dist}^2(\cdot,\mathcal{I}(t))\} \,d\omega_t 
&\leq E[\mu,\chi|\mathscr{A}](t),
\end{align}
as well as
\begin{align}
\label{eq:coercivity3}
\sigma\int_{\R^d} (\chi_{A(t)} {-} \chi_{\mathscr{A}(t)})\vartheta(\cdot,t) \,dx
&= E_{\mathrm{bulk}}[\chi|\mathscr{A}](t)
\\ \label{eq:coercivity4}
\sigma\int_{\R^d} |\chi_{A(t)} {-} \chi_{\mathscr{A}(t)}|
\min\{1,r_c^{-1}\mathrm{dist}(\cdot,\mathcal{I}(t))\} \,dx
&\leq E_{\mathrm{bulk}}[\chi|\mathscr{A}](t).
\end{align}

The relative entropy in fact admits several equivalent expressions. For
the computation of its time evolution, it is appropriate to rewrite
it based on the compatibility condition~\eqref{eq:compatibility}
and an integration by parts in form of
(with the measure-theoretic inward pointing unit normal along $\p^*A(t)$ defined by
$n(\cdot,t) := \smash{\frac{\nabla\chi(\cdot,t)}{|\nabla\chi(\cdot,t)|}}$, $t \in (0,\infty)$)
\begin{align}
\nonumber
E[\mu,\chi|\mathscr{A}](t)
&= \int_{\R^d} 1 \,d\omega_t - \sigma\int_{\p^*A(t)} 
n(\cdot,t)\cdot\xi(\cdot,t) \,d\H^{d-1}
\\& \label{eq:relEntropy2}
= \int_{\R^d} 1 \,d\omega_t + \sigma\int_{\R^d} \chi(\cdot,t)
(\nabla\cdot\xi)(\cdot,t) \,dx.
\end{align}
To identify further coercivity properties, it is useful to add and subtract
the Radon--Nikod\'ym derivative~$\rho$ and again
make use of the compatibility condition~\eqref{eq:compatibility} to obtain
\begin{align}
\label{eq:relEntropy3}
E[\mu,\chi|\mathscr{A}](t)
&= \int_{\R^d} 1 - \rho(\cdot,t) \,d\omega_t 
+ \sigma\int_{\p^*A(t)} 1 - n(\cdot,t)\cdot\xi(\cdot,t) \,d\H^{d-1}.
\end{align}
In particular, we may deduce from the properties of $\rho$ and $\xi$ that
\begin{align}
\label{eq:coercivity5}
\int_{\R^d} 1 - \rho(\cdot,t) \,d\omega_t 
&\leq E[\mu,\chi|\mathscr{A}](t),
\\ \label{eq:coercivity6}
\sigma\int_{\p^*A(t)} |n(\cdot,t){-}\xi(\cdot,t)|^2 \,d\H^{d-1}
&\leq 2E[\mu,\chi|\mathscr{A}](t),
\\ \label{eq:coercivity7}
\sigma\int_{\p^*A(t)} \min\{1,r_c^{-2}\mathrm{dist}^2(\cdot,\mathcal{I}(t))\} \,d\H^{d-1}
&\leq E[\mu,\chi|\mathscr{A}](t).
\end{align}
Now, we have everything in place to proceed with the

\begin{proof}[Proof of Theorem~\ref{theo:weakStrongUniqueness}]
We split the proof into three steps.

\textit{Step 1: Proof of the stability estimate~\emph{\eqref{eq:stabilityRelEntropy}}.}
Plugging in the representation~\eqref{eq:relEntropy2} at time $t=T$
and $t=0$, making use of the De~Giorgi type inequality~\eqref{eq:DeGiorgiInequality}
and the evolution equation~\eqref{eq:evolPhase}, and integrating by parts shows
\begin{align}
\label{eq:stabilityEstimateAux1}
E[\mu,\chi|\mathscr{A}](T)  &\leq
E[\mu_0,\chi_0|\mathscr{A}(0)] 
- \frac{1}{2} \int_{0}^{T} \int_{\R^d} |V|^2 \,d\omega_t dt
- \frac{1}{2} \int_{0}^{T} \int_{\R^d} |\mathbf{H}|^2 \,d\omega_t dt
\\&~~~ \nonumber
- \int_{0}^{T} \int_{\R^d} V(\nabla\cdot\xi) \,d\omega_t dt
- \sigma\int_{0}^{T} \int_{\p^*A(t)} n\cdot\partial_t\xi \,d\H^{d-1} dt.
\end{align}
Having in mind the approximate transport equations~\eqref{eq:transportXi}
and~\eqref{eq:transportLengthXi} as well as the coercivity properties~\eqref{eq:coercivity6}
and~\eqref{eq:coercivity7}, we add zero twice to obtain the estimate
\begin{align}
\nonumber
&- \sigma\int_{0}^{T} \int_{\p^*A(t)} n\cdot\partial_t\xi \,d\H^{d-1} dt
\\ \nonumber
&= - \sigma\int_{0}^{T} \int_{\p^*A(t)} n\cdot\big(\partial_t\xi {+} 
(B\cdot\nabla)\xi {+} (\nabla B)^\mathsf{T}\xi\big) \,d\H^{d-1} dt
\\&~~~ \nonumber
+ \sigma\int_{0}^{T} \int_{\p^*A(t)} n\cdot\big((B\cdot\nabla)\xi 
{+} (\nabla B)^\mathsf{T}\xi\big) \,d\H^{d-1} dt
\\& \nonumber
= - \sigma\int_{0}^{T} \int_{\p^*A(t)} (n{-}\xi)\cdot\big(\partial_t\xi {+} 
(B\cdot\nabla)\xi {+} (\nabla B)^\mathsf{T}\xi\big) \,d\H^{d-1} dt
\\&~~~ \nonumber
- \sigma\int_{0}^{T} \int_{\p^*A(t)} \xi\cdot\big(
\partial_t\xi {+} (B\cdot\nabla)\xi\big) \,d\H^{d-1} dt
+ \sigma\int_{0}^{T} \int_{\p^*A(t)} \xi\otimes(n{-}\xi) : \nabla B \,d\H^{d-1} dt
\\&~~~ \nonumber
+ \sigma\int_{0}^{T} \int_{\p^*A(t)} n\cdot(B\cdot\nabla)\xi \,d\H^{d-1} dt
\\& \label{eq:stabilityEstimateAux2}
\leq \sigma\int_{0}^{T} \int_{\p^*A(t)} \xi\otimes(n{-}\xi) : \nabla B \,d\H^{d-1} dt
+ \sigma\int_{0}^{T} \int_{\p^*A(t)} n\cdot(B\cdot\nabla)\xi \,d\H^{d-1} dt
\\&~~~ \nonumber
+ C \int_{0}^{T} E[\mu,\chi|\mathscr{A}](t) \,dt.
\end{align}
Based on the compatibility condition~\eqref{eq:compatibility},
the identity~\eqref{eq:RadonNikodymProperty}, the coercivity property~\eqref{eq:coercivity5},
and adding zero we also have
\begin{align}
\nonumber
&\sigma\int_{0}^{T} \int_{\p^*A(t)} \xi\otimes(n{-}\xi) : \nabla B \,d\H^{d-1} dt
\\& \nonumber
= \int_{0}^{T} \int_{\R^d{\times}\S^{d-1}} \xi \otimes (p {-} \xi) : \nabla B \,d\mu_t dt
- \int_{0}^{T} \int_{\R^d{\times}\S^{d-1}} (\rho {-} 1)(\xi \otimes \xi : \nabla B) \,d\omega_t dt
\\& \label{eq:stabilityEstimateAux3}
\leq \int_{0}^{T} \int_{\R^d{\times}\S^{d-1}} \xi \otimes (p {-} \xi) : \nabla B \,d\mu_t dt
+ C \int_{0}^{T} E[\mu,\chi|\mathscr{A}](t) \,dt.
\end{align}
Next, we obtain by appealing to the product rule in form of $n\cdot(B\cdot\nabla)\xi
=n\cdot(\nabla\cdot(\xi\otimes B)) - (n\cdot\xi)(\nabla\cdot B)$, adding zero three times,
applying again the identity~\eqref{eq:RadonNikodymProperty}
as well as the coercivity property~\eqref{eq:coercivity5}, and finally using
the representation~\eqref{eq:relEntropy3}
\begin{align}
\nonumber
&\sigma\int_{0}^{T} \int_{\p^*A(t)} n\cdot(B\cdot\nabla)\xi \,d\H^{d-1} dt
\\ \nonumber
&= \sigma\int_{0}^{T} \int_{\p^*A(t)} n\cdot\big(\nabla\cdot(\xi\otimes B)\big) \,d\H^{d-1} dt
- \sigma\int_{0}^{T} \int_{\p^*A(t)} (n\cdot\xi - 1)(\nabla\cdot B) \,d\H^{d-1} dt
\\&~~~ \nonumber
- \int_{0}^{T} \int_{\R^d{\times}\S^{d-1}} (\rho {-} 1)(\nabla\cdot B) \,d\omega_t dt
- \int_{0}^{T} \int_{\R^d{\times}\S^{d-1}} (I_d {-} p\otimes p) : \nabla B \,d\mu_t dt
\\&~~~ \nonumber
- \int_{0}^{T} \int_{\R^d{\times}\S^{d-1}} p\otimes p : \nabla B \,d\mu_t dt
\\& \label{eq:stabilityEstimateAux4}
\leq \sigma\int_{0}^{T} \int_{\p^*A(t)} n\cdot\big(\nabla\cdot(\xi\otimes B)\big) \,d\H^{d-1} dt
- \int_{0}^{T} \int_{\R^d{\times}\S^{d-1}} (I_d {-} p\otimes p) : \nabla B \,d\mu_t dt
\\&~~~ \nonumber
- \int_{0}^{T} \int_{\R^d{\times}\S^{d-1}} p\otimes p : \nabla B \,d\mu_t dt
+ C \int_{0}^{T} E[\mu,\chi|\mathscr{A}](t) \,dt.
\end{align}
We further compute due to an integration by parts, the symmetry relation
$\nabla\cdot(\nabla\cdot(\xi\otimes B)) = \nabla\cdot(\nabla\cdot(B\otimes\xi))$,
reverting the integration by parts, the product rule in form of
$n\cdot(\nabla\cdot(B\otimes\xi)) = (n\cdot B)(\nabla\cdot\xi)
+ n\otimes\xi:\nabla B$, and the compatibility condition~\eqref{eq:compatibility}
\begin{align}
\nonumber
&\sigma\int_{0}^{T} \int_{\p^*A(t)} n\cdot\big(\nabla\cdot(\xi\otimes B)\big) \,d\H^{d-1} dt
\\ \nonumber
&= \sigma\int_{0}^{T} \int_{\p^*A(t)} n\cdot\big(\nabla\cdot(B\otimes \xi)\big) \,d\H^{d-1} dt
\\& \nonumber
= \sigma\int_{0}^{T} \int_{\p^*A(t)} (n\cdot B)(\nabla\cdot\xi) \,d\H^{d-1} dt
+ \sigma\int_{0}^{T} \int_{\p^*A(t)} n\otimes\xi:\nabla B \,d\H^{d-1} dt
\\& \label{eq:stabilityEstimateAux5}
= \sigma\int_{0}^{T} \int_{\p^*A(t)} (n\cdot B)(\nabla\cdot\xi) \,d\H^{d-1} dt
+ \int_{0}^{T} \int_{\R^d{\times}\S^{d-1}} p\otimes\xi:\nabla B \,d\mu_t dt.
\end{align}
The combination of the estimates~\eqref{eq:stabilityEstimateAux2}--\eqref{eq:stabilityEstimateAux5}
thus implies
\begin{align*}
&- \sigma\int_{0}^{T} \int_{\p^*A(t)} n\cdot\partial_t\xi \,d\H^{d-1} dt
\\
&\leq \sigma\int_{0}^{T} \int_{\p^*A(t)} (n\cdot B)(\nabla\cdot\xi) \,d\H^{d-1} dt
- \int_{0}^{T} \int_{\R^d{\times}\S^{d-1}} (I_d {-} p\otimes p) : \nabla B \,d\mu_t dt
\\&~~~
- \int_{0}^{T} \int_{\R^d{\times}\S^{d-1}} (p {-} \xi) \otimes (p {-} \xi) : \nabla B \,d\mu_t dt
+ C \int_{0}^{T} E[\mu,\chi|\mathscr{A}](t) \,dt.
\end{align*}
Based on the coercivity property~\eqref{eq:coercivity1},
the previous display in turn upgrades~\eqref{eq:stabilityEstimateAux1} to
\begin{align}
\label{eq:stabilityEstimateAux6}
E[\mu,\chi|\mathscr{A}](T)  &\leq
E[\mu_0,\chi_0|\mathscr{A}(0)] 
- \frac{1}{2} \int_{0}^{T} \int_{\R^d} |V|^2 \,d\omega_t dt
- \frac{1}{2} \int_{0}^{T} \int_{\R^d} |\mathbf{H}|^2 \,d\omega_t dt
\\&~~~ \nonumber
- \int_{0}^{T} \int_{\R^d} V(\nabla\cdot\xi) \,d\omega_t dt
+ \sigma\int_{0}^{T} \int_{\p^*A(t)} (n\cdot B)(\nabla\cdot\xi) \,d\H^{d-1} dt
\\&~~~ \nonumber
- \int_{0}^{T} \int_{\R^d{\times}\S^{d-1}} (I_d {-} p\otimes p) : \nabla B \,d\mu_t dt
+ C \int_{0}^{T} E[\mu,\chi|\mathscr{A}](t) \,dt.
\end{align}

What remains to be done in the derivation of the stability estimate~\eqref{eq:stabilityRelEntropy}
is to make use of the control provided by the dissipation terms from the De~Giorgi inequality~\eqref{eq:DeGiorgiInequality}.
We start by completing squares twice, applying the estimate~\eqref{eq:motionByMeanCurvature}
as well as the identity~\eqref{eq:RadonNikodymProperty}, and
exploiting the coercivity properties~\eqref{eq:coercivity2} and~\eqref{eq:coercivity5}
\begin{align}
\nonumber
&- \frac{1}{2} \int_{0}^{T} \int_{\R^d} |V|^2 \,d\omega_t dt
- \int_{0}^{T} \int_{\R^d} V(\nabla\cdot\xi) \,d\omega_t dt
+ \sigma\int_{0}^{T} \int_{\p^*A(t)} (n\cdot B)(\nabla\cdot\xi) \,d\H^{d-1} dt
\\& \nonumber
= - \frac{1}{2} \int_{0}^{T} \int_{\R^d} |V {+} (\nabla\cdot\xi)|^2 \,d\omega_t dt
+ \frac{1}{2}\int_{0}^{T} \int_{\R^d} (\nabla\cdot\xi)^2 \,d\omega_t dt
\\&~~~ \nonumber
+ \sigma\int_{0}^{T} \int_{\p^*A(t)} (n\cdot B)(\nabla\cdot\xi) \,d\H^{d-1} dt
\\& \nonumber
= - \frac{1}{2} \int_{0}^{T} \int_{\R^d} |V {+} (\nabla\cdot\xi)|^2 \,d\omega_t dt
+ \frac{1}{2}\int_{0}^{T} \int_{\R^d} |\nabla\cdot\xi {+} (B\cdot\xi)|^2 \,d\omega_t dt
\\&~~~ \nonumber
- \int_{0}^{T} \int_{\R^d} (1{-}\rho)(B\cdot\xi)(\nabla\cdot\xi) \,d\omega_t dt
\\&~~~ \nonumber
+ \sigma\int_{0}^{T} \int_{\p^*A(t)} \big((n{-}\xi)\cdot B\big)(\nabla\cdot\xi) \,d\H^{d-1} dt
- \frac{1}{2}\int_{0}^{T} \int_{\R^d} (B\cdot\xi)^2 \,d\omega_t dt
\\& \label{eq:stabilityEstimateAux7}
\leq - \frac{1}{2} \int_{0}^{T} \int_{\R^d} |V {+} (\nabla\cdot\xi)|^2 \,d\omega_t dt
- \frac{1}{2}\int_{0}^{T} \int_{\R^d} (B\cdot\xi)^2 \,d\omega_t dt
\\&~~~ \nonumber
+ \sigma\int_{0}^{T} \int_{\p^*A(t)} \big((n{-}\xi)\cdot B\big)(\nabla\cdot\xi) \,d\H^{d-1} dt
+ C \int_{0}^{T} E[\mu,\chi|\mathscr{A}](t) \,dt.
\end{align}
Furthermore, decomposing and completing squares also shows
(recalling from definition~\eqref{def:xi} that $\xi=|\xi|\nabla s_{\mathcal{I}}$
and $\supp\xi(\cdot,t)=\{\dist(\cdot,\mathcal{I}(t)) \leq r_c\}$)
\begin{align}
\nonumber
&- \frac{1}{2} \int_{0}^{T} \int_{\R^d} |\mathbf{H}|^2 \,d\omega_t dt
- \frac{1}{2}\int_{0}^{T} \int_{\R^d} (B\cdot\xi)^2 \,d\omega_t dt
\\& \label{eq:stabilityEstimateAux8}
= - \frac{1}{2} \int_{0}^{T} \int_{\{\dist(\cdot,\mathcal{I}(t)) > r_c\}} |\mathbf{H}|^2 \,d\omega_t dt
- \frac{1}{2} \int_{0}^{T} \int_{\{\dist(\cdot,\mathcal{I}(t)) \leq r_c\}} 
|(I_d{-}\nabla s_{\mathcal{I}}\otimes\nabla s_{\mathcal{I}})\mathbf{H}|^2 \,d\omega_t dt
\\&~~~ \nonumber
- \frac{1}{2} \int_{0}^{T} \int_{\{\dist(\cdot,\mathcal{I}(t)) \leq r_c\}} 
|\nabla s_{\mathcal{I}}\cdot\mathbf{H}|^2(1{-}|\xi|^2) \,d\omega_t dt
- \frac{1}{2} \int_{0}^{T} \int_{\R^d}
|(\xi\cdot\mathbf{H}) {-} (B\cdot\xi)|^2 \,d\omega_t dt
\\&~~~ \nonumber
- \int_{0}^{T} \int_{\R^d} (\xi\cdot\mathbf{H})(B\cdot\xi) \,d\omega_t dt.
\end{align}
As a consequence of plugging in test functions of the form $\eta B$ into the 
identity~\eqref{eq:weakCurvature}, where $\eta \in C_c^\infty([0,\infty);[0,1])$ 
runs through a sequence approximating the indicator function of $(0,T)$ monotonically from above,
and a splitting similar to the previous display, it also holds
\begin{align}
\nonumber
&- \int_{0}^{T} \int_{\R^d{\times}\S^{d-1}} (I_d {-} p\otimes p) : \nabla B \,d\mu_t dt
\\& \label{eq:stabilityEstimateAux9}
= \frac{1}{2} \int_{0}^{T} \int_{\{\dist(\cdot,\mathcal{I}(t)) > r_c\}} \mathbf{H}\cdot B \,d\omega_t dt
+ \int_{0}^{T} \int_{\R^d} (\xi\cdot\mathbf{H})(B\cdot\xi) \,d\omega_t dt
\\&~~~ \nonumber
+ \int_{0}^{T} \int_{\{\dist(\cdot,\mathcal{I}(t)) \leq r_c\}} 
(1{-}|\xi|^2)(\nabla s_{\mathcal{I}}\cdot\mathbf{H})(B\cdot\nabla s_{\mathcal{I}}) \,d\omega_t dt
\\&~~~ \nonumber
+ \int_{0}^{T} \int_{\{\dist(\cdot,\mathcal{I}(t)) \leq r_c\}} 
(I_d{-}\nabla s_{\mathcal{I}}\otimes\nabla s_{\mathcal{I}})\mathbf{H} \cdot 
(I_d{-}\nabla s_{\mathcal{I}}\otimes\nabla s_{\mathcal{I}}) B \,d\omega_t dt.
\end{align}
Summing~\eqref{eq:stabilityEstimateAux8} and \eqref{eq:stabilityEstimateAux9}, by an absorption argument, the bounds $1{-}|\xi|^2 \leq 2(1{-}|\xi|)
\leq 2(1 - n\cdot\xi)$, the fact $\supp\xi(\cdot,t)=\{\dist(\cdot,\mathcal{I}(t)) \leq r_c\}$, and
the representation~\eqref{eq:relEntropy3}, we obtain
\begin{align}
\nonumber
&- \frac{1}{2} \int_{0}^{T} \int_{\R^d} |\mathbf{H}|^2 \,d\omega_t dt
- \frac{1}{2}\int_{0}^{T} \int_{\R^d} (B\cdot\xi)^2 \,d\omega_t dt
- \int_{0}^{T} \int_{\R^d{\times}\S^{d-1}} (I_d {-} p\otimes p) : \nabla B \,d\mu_t dt
\\& \label{eq:stabilityEstimateAux10}
\leq 
C\int_{0}^{T} \int_{\{\dist(\cdot,\mathcal{I}(t)) \leq r_c\}} 
|(I_d{-}\nabla s_{\mathcal{I}}\otimes\nabla s_{\mathcal{I}}) B|^2 \,d\omega_t dt
+ C \int_{0}^{T} E[\mu,\chi|\mathscr{A}](t) \,dt.
\end{align}
Due to $|(n{-}\xi)\xi| \leq (1 {-} |\xi|^2) + (1 - n\cdot\xi)
\leq 3(1 - n\cdot\xi)$ we also get
\begin{align}
\nonumber
&\sigma\int_{0}^{T} \int_{\p^*A(t)} \big((n{-}\xi)\cdot B\big)(\nabla\cdot\xi) \,d\H^{d-1} dt
\\& \label{eq:stabilityEstimateAux11}
\leq \sigma\int_{0}^{T} \int_{\p^*A(t)} \big((n{-}\xi)\cdot (I_d{-}\xi\otimes\xi)B\big)(\nabla\cdot\xi) \,d\H^{d-1} dt
+ C \int_{0}^{T} E[\mu,\chi|\mathscr{A}](t) \,dt.
\end{align}
In summary, the estimates~\eqref{eq:stabilityEstimateAux6}, \eqref{eq:stabilityEstimateAux7}, \eqref{eq:stabilityEstimateAux10}, and~\eqref{eq:stabilityEstimateAux11}
together with $(I_d{-}\nabla s_{\mathcal{I}}\otimes\nabla s_{\mathcal{I}})B = 0$ (recall
the definition~\eqref{def:B} of~$B$) and $|\nabla s_{\mathcal{I}}\otimes\nabla s_{\mathcal{I}} - \xi\otimes\xi|
\leq 1 {-} |\xi|^2$ throughout $\{\dist(\cdot,\mathcal{I}(t)) \leq r_c\}$ imply
\begin{align}
\nonumber
E[\mu,\chi|\mathscr{A}](T)  &\leq
E[\mu_0,\chi_0|\mathscr{A}(0)] - \frac{1}{2} \int_{0}^{T} \int_{\R^d} |V {+} (\nabla\cdot\xi)|^2 \,d\omega_t dt
+ C \int_{0}^{T} E[\mu,\chi|\mathscr{A}](t) \,dt
\\&~~~ \nonumber
- \sigma\int_{0}^{T} \int_{\p^*A(t)} \big((n{-}\xi)\cdot (I_d{-}\xi\otimes\xi)B\big)(\nabla\cdot\xi) \,d\H^{d-1} dt
\\&~~~ \nonumber
+ C\int_{0}^{T} \int_{\{\dist(\cdot,\mathcal{I}(t)) \leq r_c\}} 
|(I_d{-}\nabla s_{\mathcal{I}}\otimes\nabla s_{\mathcal{I}}) B|^2 \,d\omega_t dt
\\& \label{eq:stabilityEstimateAux12}
\leq E[\mu_0,\chi_0|\mathscr{A}(0)] - \frac{1}{2} \int_{0}^{T} \int_{\R^d} |V {+} (\nabla\cdot\xi)|^2 \,d\omega_t dt
+ C \int_{0}^{T} E[\mu,\chi|\mathscr{A}](t) \,dt.
\end{align}
Of course, the previous display entails the asserted estimate.
(Note that only the very last step of this argument relied on the fact
that~$B$ always points in normal direction, which---at least in principle---is
the only point exploiting crucially the simple two-phase setting.)

\textit{Step 2: Proof of the stability estimate~\emph{\eqref{eq:stabilityBulkError}}.}
Plugging in the alternative characterization~\eqref{eq:bulkError} of the bulk error term at time $t=T$ and $t=0$,
exploiting the evolution equation~\eqref{eq:evolPhase}, 
and making use of $\vartheta(\cdot,t)=0$ on $\mathcal{I}(t)$
due to~\eqref{eq:coercivity0}, we first get
\begin{align}
\label{eq:stabilityEstimateAux13}
E_{\mathrm{bulk}}[\chi|\mathscr{A}](T)
&= E_{\mathrm{bulk}}[\chi_0|\mathscr{A}(0)] - \int_{0}^{T} V\vartheta \,d\omega_t dt
+ \sigma\int_{0}^{T} \int_{\R^d} (\chi_{A(t)} {-} \chi_{\mathscr{A}(t)}) \partial_t\vartheta \,dx dt.
\end{align}
Inserting the evolution equation~\eqref{eq:evolWeight},
appealing to the product rule in form of $(B\cdot\nabla)\vartheta=\nabla\cdot(\vartheta B)
- \vartheta(\nabla\cdot B)$,  recalling the definition~\eqref{eq:bulkError},
and finally integrating by parts and using $\rho=0$ on $\partial \mathscr{A}(t)$ shows
\begin{align}
\nonumber
\sigma\int_{0}^{T} \int_{\R^d} (\chi_{A(t)} {-} \chi_{\mathscr{A}(t)}) \partial_t\vartheta \,dx dt
&= - \sigma\int_{0}^{T} \int_{\R^d} (\chi_{A(t)} {-} \chi_{\mathscr{A}(t)}) \nabla\cdot(\vartheta B) \,dx dt
\\&~~~ \nonumber
+ \sigma\int_{0}^{T} \int_{\R^d} (\chi_{A(t)} {-} \chi_{\mathscr{A}(t)}) \vartheta(\nabla\cdot B) \,dx dt
\\& \label{eq:stabilityEstimateAux14}
\leq \sigma\int_{0}^{T} \int_{\p^*A(t)} (n\cdot B)\vartheta \,d\H^{d-1} dt
+ C \int_{0}^{T} E_{\mathrm{bulk}}[\chi|\mathscr{A}](t) \,dt.
\end{align}
Next, by adding zero three times as well as making use of the identity~\eqref{eq:RadonNikodymProperty},
Young's inequality and the coercivity properties~\eqref{eq:coercivity0}, \eqref{eq:bulkError}, \eqref{eq:coercivity2}
and \eqref{eq:coercivity5}--\eqref{eq:coercivity7} we may further estimate
\begin{align}
\nonumber
&- \int_{0}^{T} V\vartheta \,d\omega_t dt
+ \sigma\int_{0}^{T} \int_{\p^*A(t)} (n\cdot B)\vartheta \,d\H^{d-1} dt
\\& \nonumber
= - \int_{0}^{T} V\vartheta \,d\omega_t dt
+ \sigma\int_{0}^{T} \int_{\p^*A(t)} \big((n{-}\xi)\cdot B\big)\vartheta \,d\H^{d-1} dt
\\&~~~ \nonumber
+ \sigma\int_{0}^{T} \int_{\p^*A(t)} \big((\xi\cdot B){+}(\nabla\cdot\xi)\big)\vartheta \,d\H^{d-1} dt
\\&~~~ \nonumber
- \int_{0}^{T} \int_{\p^*A(t)} (\rho{-}1)(\nabla\cdot\xi)\vartheta \,d\omega_t dt
- \int_{0}^{T} \int_{\p^*A(t)} (\nabla\cdot\xi)\vartheta \,d\omega_t dt
\\& \label{eq:stabilityEstimateAux15}
\leq \frac{1}{4}\int_{0}^{T} \int_{\R^d} \big|(\nabla\cdot\xi) {+} V\big|^2 \,d\omega_t dt
+ C \int_{0}^{T} E[\mu,\chi|\mathscr{A}](t) \,dt.
\end{align}
Hence, the bounds~\eqref{eq:stabilityEstimateAux13}--\eqref{eq:stabilityEstimateAux15}
in total entail
\begin{align}
\label{eq:stabilityEstimateAux16}
E_{\mathrm{bulk}}[\chi|\mathscr{A}](T)
&\leq E_{\mathrm{bulk}}[\chi_0|\mathscr{A}(0)]
+ \frac{1}{4}\int_{0}^{T} \int_{\R^d}\big|(\nabla\cdot\xi) {+} V\big|^2 \,d\omega_t dt
\\&~~~ \nonumber
+ C \int_{0}^{T} E_{\mathrm{bulk}}[\chi|\mathscr{A}](t) + E[\mu,\chi|\mathscr{A}](t) \,dt.
\end{align}
We thus infer~\eqref{eq:stabilityBulkError} from
adding~\eqref{eq:stabilityEstimateAux16} to~\eqref{eq:stabilityEstimateAux12}.

\textit{Step 3: Proof of the weak-strong uniqueness principle~\emph{\eqref{eq:weakStrong1}--\eqref{eq:weakStrong2}}.}
This now follows immediately from the previous two steps.
\end{proof}

\section{Proof of Lemma \ref{lemma:discussion De Giorgi solution}}
\label{sec:proofLemma}

	\step{Argument for Item (\ref{item:classical is weak}):} Only the
	validity of the De~Giorgi type inequality~\eqref{eq:DeGiorgiInequality}
	deserves a comment, while the other requirements are immediate consequences
	of the smoothness of the underlying evolution and the definition of~$(\mu,\chi)$.
	With respect to~\eqref{eq:DeGiorgiInequality}, we note that it is sufficient
	to establish for a.e.\ $T \in (0,\infty)$
	\begin{align*}
	\sigma\mathcal{H}^{d-1}(\partial\mathscr{A}(T)) + 
	\sigma \int_0^T \int_{\R^d} Vn\cdot\mathbf{H} \,d\mathcal{H}^{d-1} dt
	&\leq \sigma\mathcal{H}^{d-1}(\partial\mathscr{A}(0)).
	\end{align*}
	Indeed, since we know by assumption that
	$V = n\cdot\mathbf{H}$ throughout $\R^d {\times} (0,\infty)$,
	the De~Giorgi type inequality~\eqref{eq:DeGiorgiInequality}
	is easily seen to be equivalent to the previous display.
	On the other side, in the setting considered here the inequality of the previous display holds true
	even as an equality which in turn is a well-known property for smooth evolutions.
	
	\step{Argument for Item (\ref{item:weak is classical}):} Since $(\partial\mathscr{A}(t))_{t \in [0,T)}$
	evolves smoothly, we have again for all $T' \in (0,T)$
	\begin{align*}
  \sigma\mathcal{H}^{d-1}(\partial\mathscr{A}(T')) + 
	\sigma \int_0^{T'} \int_{\R^d} Vn\cdot\mathbf{H} \,d\mathcal{H}^{d-1} dt
	&= \sigma\mathcal{H}^{d-1}(\partial\mathscr{A}(0)).
	\end{align*}
	On the other hand, by virtue of the assumed representation of the given De~Giorgi type varifold solution,
	\eqref{eq:DeGiorgiInequality} and completing squares, we also have for all $T' \in (0,T)$ 
	\begin{align*}
	  &\sigma\mathcal{H}^{d-1}(\partial\mathscr{A}(T')) + 
	\sigma \int_0^{T'} \int_{\R^d} \frac{1}{2}(V - n\cdot\mathbf{H})^2 \,d\mathcal{H}^{d-1} dt
	+ \sigma \int_0^{T'} \int_{\R^d} Vn\cdot\mathbf{H} \,d\mathcal{H}^{d-1} dt
	\\
	&\leq \sigma\mathcal{H}^{d-1}(\partial\mathscr{A}(0)).
	\end{align*}
	Subtracting the former display from the latter, 
	we obtain as desired that $V = n\cdot\mathbf{H}$ holds true throughout $\R^d {\times} (0,T)$.
	
	\step{Argument for Item (\ref{item:structure}):}
	The rectifiability of $\mu_t$ for a.e.~$t$ follows from the fact that \eqref{eq:DeGiorgiInequality} 
	implies $\int |\mathbf{H}(\cdot,t)|^2 \,d\omega_t <\infty$ for a.e.~$t$ and Allard's rectifiability criterion \cite{allard}, see also \cite{DePhilippis} for an alternative proof.
	
	Now we show~\eqref{eq:lemma cont vol}. Taking the difference of \eqref{eq:evolPhase} 
	with $T=t$ and $T=s$ for cylindrical test functions $(x,t)\mapsto\eta(t)\zeta(x)$,
	where $\zeta\in C^\infty_c(\R^d;[-1,1])$ and $\eta\in C^\infty_c([0,\infty);[0,1])$
	such that $\eta \equiv 1$ on $[s,t]$, and applying Cauchy--Schwarz yields
	\begin{align*}
	\sigma\int_{A(t) \Delta A(s)} \zeta(x) \,dx
	&= -\int_{s}^{t} \int_{\R^d} V(x,t)\zeta(x) \,d\omega_tdt
	\\& \leq \left( \int_{\R^d\times(0,\infty)} V^2 \, d\omega\right)^\frac12 
	\left( \int_{\R^d\times(s,t)} 1 \, d\omega\right)^\frac12.
	\end{align*}
	The first right-hand side integral is bounded by $2\omega_0(\R^d)$ while the second one is bounded by $(t-s)\omega_0(\R^d)$;
	both due to~\eqref{eq:DeGiorgiInequality}. Taking the supremum over~$\zeta$ thus entails~\eqref{eq:lemma cont vol}.
	
	\step{Argument for Item (\ref{item:compact sol space}):} 
	We again only focus on establishing~\eqref{eq:DeGiorgiInequality};
	the remaining conditions in fact essentially follow in the process.
	
	Define $\omega_k = \left(\mathcal{L}^1\llcorner (0,\infty)\right) \otimes ((\omega_k)_t)_{t \in (0,\infty)} $. Up to choosing
	a suitable subsequence, the sequence $\omega_k \in \mathcal{M}(\R^d{\times}(0,\infty))$
	converges weakly-$*$ to some $\omega \in \mathcal{M}(\R^d{\times}(0,\infty))$ in the sense of Radon measures.
	Recall that by assumption, for all $k$, the map $t\mapsto (\omega_k)_t(\R^d)$ is non-increasing, and 
	that the initial energy $(\omega_k)_0(\R^d)$ is uniformly bounded in $k\in \N$.
	Hence, a close inspection of the proof of Lemma~\ref{lemma:energy lsc} shows that
	$\omega=\left(\mathcal{L}^1\llcorner (0,\infty)\right) \otimes (\omega_t)_{t \in (0,\infty)} $, 
	where $(\omega_t)_{t \in (0,\infty)}$ is a weakly-$\ast$ $\mathcal{L}^1$-measurable family of
	finite Radon measures, and for a.e.\ $t \in (0,\infty)$ it moreover holds
		\begin{align}\label{eq:Auxcomp mass at t}
		\liminf_{k\downarrow0}(\omega_k)_t(\R^d) \geq \omega_t(\R^d).
	\end{align}
	As in the proof of Proposition \ref{prop:Vsquared}, there exists $V \in L^2((0,\infty);L^2(\R^d,d\omega_t))$
	arising as the Radon--Nikod\'ym derivative of $\partial_t\chi$ with respect to~$\omega$
	(in particular, the equation~\eqref{eq:evolPhase} follows)
	such that in addition for a.e.\ $T \in (0,\infty)$ 
	\begin{align}
			\liminf_{k\downarrow0} \frac12 
		\int_{\R^d {\times} (0,T)} 
		V_k^2 \,d\omega_k
		\geq \frac12  \int_{\R^d {\times} (0,T)}  V^2 \,d\omega.
	\end{align}
	By a straightforward argument solely based on the Riesz representation theorem and 
	the subsequential weak-$*$ convergence of $\mu_k$ to some $\mu$
	(with~$\omega$ being the associated mass measure) in the sense of Radon measures
	in $\R^d\times (0,\infty)\times \S^{d-1}$, there exists 
	$\mathbf{H} \in L^2((0,\infty);L^2(\R^d,d\omega_t;\R^d))$
	such that~\eqref{eq:weakCurvature} holds true and for a.e.\ $T \in (0,\infty)$, we have
	\begin{align}
			\liminf_{k\downarrow0} \frac12 
		\int_{\R^d {\times} (0,T)} 
		\mathbf{H}_k^2 \,d\omega_k
		\geq \frac12  \int_{\R^d {\times} (0,T)}  \mathbf{H}^2 \,d\omega.
	\end{align}
	All in all, in view of the previous three displays together with the validity of~\eqref{eq:DeGiorgiInequality}
	in terms of the sequence~$(\mu_k,\chi_k)$, we may deduce for a.e.\ $T \in (0,\infty)$ 
	that~\eqref{eq:DeGiorgiInequality} holds true in terms of~$(\mu,\chi)$ as desired.
	\qed


\section{A proposal in the setting of multi-phase mean curvature flow}
\label{sec:multi-phase}
In this section, we propose a varifold solution concept \`a la De~Giorgi
for multi-phase mean curvature flow. We only consider the case of equal
surface tensions (which we consequently normalize to the unit) 
meaning that at triple junctions the interfaces intersect
at equal angles of~$120^\circ$. When reduced to the two-phase setting,
our proposal turns out to be slightly stronger than what is required
in Definition~\ref{def:twoPhaseDeGiorgiVarifoldSolutions} in the sense
that, for almost every $t \in (0,T)$ and $\omega_t$ almost every $x \in \R^d$, 
the generalized mean curvature vector $\mathbf{H}(x,t)$ of~\eqref{eq:weakCurvature} points in the 
direction of the expected value $\langle\lambda_{x,t}\rangle := \smash{\int_{\S^{d-1}}p\,d\lambda_{x,t}}$ (if the latter does not vanish; see \eqref{eq:weakCurvatureDirection} below). 
The merit of this additional assumption is that it allows to close the Gronwall argument
from the previous section without requiring the velocity~$B$ to point 
in normal direction (cf.\ the proof of Theorem~\ref{theo:weakStrongUniquenessMultiPhase} below). 
This in turn is of course essential for the multi-phase case since~$B$ will necessarily contain tangential components due to the
presence of moving triple junctions.
Note that this additional assumption on $\mathbf{H}$ is a direct consequence of  Brakke's orthogonality theorem \cite[Chapter 5]{brakke} in the two-phase case since we have shown in Theorem \ref{theo:Existence} that the varifolds $\frac1\sigma \mu_t$ in our construction have integer multiplicity for a.e.~$t$.

\begin{definition}[De~Giorgi type varifold solutions for multi-phase mean curvature flow]
	\label{def:multiPhaseDeGiorgiVarifoldSolutions}
	Let $P\geq 2$ be the number of phases.
	For each pair of phases $i,j\in\{1,\ldots,P\}$, let $\mu^{i,j} = \mathcal{L}^1\otimes (\mu^{i,j}_t)_{t\in (0,\infty)}$
	be a family of oriented varifolds $\mu^{i,j}_t \in \mathcal{M}(\R^d{\times}\S^{d-1})$, $t\in (0,\infty)$,
	such that the map $(0,\infty) \ni t \mapsto \int_{\R^d{\times}\S^{d-1}} \eta(\cdot,\cdot,t) \,d\mu^{i,j}_t$
	is measurable for all test functions $\eta \in L^1((0,\infty);C_0(\R^d{\times}\S^{d-1}))$.
	Define evolving oriented varifolds $\mu^i=\mathcal{L}^1\otimes (\mu^i_t)_{t\in(0,\infty)}$, $i\in\{1,\ldots,P\}$,
	and $\mu=\mathcal{L}^1\otimes (\mu_t)_{t\in(0,\infty)}$ by means of
	\begin{subequations}
		\begin{align}
		\label{eq:varifoldGrain}
		\mu^i_t &:= 2\mu^{i,i}_t + \sum_{j=1,\,j\neq i}^P \mu^{j,i}_t,
		&& t \in (0,\infty),
		\\
		\label{eq:varifoldNetwork}
		\mu_t &:= \frac{1}{2}\sum_{i=1}^P\mu^i_t,
		&& t \in (0,\infty).
		\end{align}
		The disintegration of $\mu^{i,j}_t$ is expressed in form of $\mu^{i,j}_t = 
		\omega^{i,j}_t \otimes \smash{(\lambda^{i,j}_{x,t})_{x\in\R^d}}$ with expected
		value $\langle\lambda^{i,j}_{x,t}\rangle:=\smash{\int_{\S^{d-1}}p\,d\lambda^{i,j}_{x,t}}(p)$
		for all $i,j\in\{1,\ldots,P\}$ and all $t\in(0,\infty)$. Analogous expressions
		are introduced for the disintegrations of $\mu^{i}_t$ and $\mu_t$, respectively.
		
		Consider also a family $A=(A^{1},\ldots,A^{P})$ such that
		for each phase $i\in\{1,\ldots,P\}$ we have a family $A^{i}=(A^{i}(t))_{t \in (0,\infty)}$ of open subsets of~$\R^d$
		with finite perimeter in~$\R^d$. We also require $(A^1(t),\ldots,A^P(t))$ to be a partition of~$\R^d$ for all $t\in(0,\infty)$
		and that all except for the $P$-th phase have finite mass for all $t\in(0,\infty)$.
		Moreover, for each $i\in\{1,\ldots,P\}$ the associated indicator function~$\chi^{i}(\cdot,t):=\chi_{A^{i}(t)}$,
		$t \in (0,\infty)$, satisfies $\chi^{i} \in L^\infty((0,\infty);\mathrm{BV}_{\mathrm{loc}}(\R^d;\{0,1\}))$.
		We shortly write $\chi=(\chi^1,\ldots,\chi^P)$.
		
		Given initial data~$((\mu^{i,j}_0)_{i,j\in\{1,\ldots,P\}},(\chi^i_0)_{i\in\{1,\ldots,P\}})$
		of the above form, we call the pair~$(\mu,\chi)$ a \emph{De~Giorgi type varifold solution for multi-phase
			mean curvature flow with initial data~$(\mu_0,\chi_0)$} if the following 
		requirements hold true.
		\begin{itemize}
			\item (Existence of normal speeds) For each phase $i\in\{1,\ldots,P\}$,
			there exists a normal speed $V^{i} \in L^2((0,\infty);L^2(\R^d,d\omega^i_t))$ in the sense that
			\begin{align}
			\label{eq:evolMultiPhase}
			\int_{\R^d} \chi^i(\cdot,T)\zeta(\cdot,T) \,dx - 
			\int_{\R^d} \chi^i_0\zeta(\cdot,0) \,dx 
			= \int_{0}^{T} \int_{\R^d} \chi^i \partial_t\zeta \,dx dt
			- \int_{0}^{T} \int_{\R^d} V^i \zeta \,d\omega^i_t dt
			\end{align}
			for almost every $T \in (0,\infty)$ and all $\zeta \in C^\infty_{c}(\R^d {\times} [0,\infty))$.
			\item (Existence of a generalized mean curvature vector) There exists
			a generalized mean curvature vector
			$\mathbf{H} \in L^2((0,\infty);L^2(\R^d,d\omega_t;\R^d))$ 
			in the sense that 
			\begin{align}
			\label{eq:weakCurvatureMultiPhase}
			\int_{0}^{T} \int_{\R^d} \mathbf{H} \cdot B \,d\omega_t dt
			= - \int_{0}^{T} \int_{\R^d {\times} \S^{d-1}}
			(I_d {-} p \otimes p) : \nabla B \,d\mu_t dt
			\end{align}
			for almost every $T \in (0,\infty)$ and all $B \in C^\infty_{c}(\R^d {\times} [0,\infty);\R^d)$.
			\item (Sharp energy dissipation principle \`a la De~Giorgi) It holds
			\begin{align}
			\label{eq:DeGiorgiInequalityMultiPhase}
			\int_{\R^d} 1 \,d\omega_{T}
			+ \frac{1}{2} \sum_{i=1}^P \int_{0}^{T} \int_{\R^d} \frac{1}{2}|V^i|^2 \,d\omega^i_t dt
			+ \frac{1}{2} \int_{0}^{T} \int_{\R^d} |\mathbf{H}|^2 \,d\omega_t dt
			\leq \int_{\R^d} 1 \,d\omega_{0}
			\end{align}
			for almost every $T \in (0,\infty)$.
			\item (Compatibility conditions) For all $i,j\in\{1,\ldots,P\}$ with $i\neq j$ we require
			\begin{align}
				\label{eq:compNormalSpeeds1}
				\omega_t^{i,j} = \omega_t^{j,i}
			\end{align}
			for almost every $t\in (0,\infty)$ and
			\begin{align}
			\label{eq:compNormals}
			\langle\lambda^{i,j}_{x,t}\rangle &= - \langle\lambda^{j,i}_{x,t}\rangle,
			\quad \langle\lambda^{i,i}_{x,t}\rangle = 0,
			\\
			\label{eq:compNormalSpeeds2}
			V^i(x,t)  &= -V^j(x,t),
			\\
			\label{eq:weakCurvatureDirection}
			|\langle\lambda^{i,j}_{x,t}\rangle|^2 \mathbf{H}(x,t)
			&= \big(\mathbf{H}(x,t) \cdot \langle\lambda^{i,j}_{x,t}\rangle\big) \langle\lambda^{i,j}_{x,t}\rangle
			\end{align}
			for almost every $t\in(0,\infty)$ and $\omega^{i,j}_t$ almost every $x\in\R^d$.
			Finally, for all $i\in\{1,\ldots,P\}$
			\begin{align}
			\label{eq:compatibilityMultiPhase}
			\int_{\R^d} \xi \cdot \,d\nabla\chi^i(\cdot,t) 
			= \int_{\R^d {\times} \S^{d-1}} \xi \cdot p \,d\mu^i_t
			\end{align}
			for almost every $t \in (0,\infty)$ and all $\xi \in C^\infty_{c}(\R^d)$.
		\end{itemize}
	\end{subequations}
\end{definition}

Analogous to the two-phase case, the condition~\eqref{eq:compatibilityMultiPhase}
ensures that $|\nabla\chi^i(\cdot,t)|\leq \omega^{i}_t$ in the sense of
measures for almost every $t\in(0,\infty)$ and all $i\in\{1,\ldots,P\}$, and
that the associated Radon--Nikod\'ym derivative satisfies
\begin{align}
\label{eq:inverseMultiplicityMultiPhase}
\rho^i(x,t) &:= \frac{d|\nabla\chi^i(\cdot,t)|}{d\omega^{i}_t}(x) \in [0,1] 
&&\text{ for } \omega^i_t \text{ almost every } x \in \R^d,
\\ \label{eq:RadonNikodymPropertyMultiPhase}
\int_{\p^*A^{i}(t)} f \,d\H^{d-1} &= \int_{\R^d{\times}\S^{d-1}}
f\rho^{i}(\cdot,t) \,d\omega^i_t
&&\text{ for all } f \in C^\infty_{c}(\R^d).
\end{align}

At the time of this writing, we admittedly do not know of an existence proof for such
a solution concept which does not require an energy convergence assumption as in \cite{LucStu, Laux-Otto, Laux-Simon}.
Moreover, the extension to general classes of surface tension matrices
does not seem to be immediate. 
However, what we will argue in this section is that the above concept 
of varifold solutions for multi-phase mean curvature flow is at least somehow reasonable
in the sense that it is subject to a weak-strong uniqueness principle;
a computation which we wanted to share with the community.
To this end, we first recall the recently introduced concept of a gradient-flow
calibration (with respect to the $L^2$ gradient-flow of the interface energy).

The underlying data is given by a family $\mathscr{A}=(\mathscr{A}^{1},\ldots,\mathscr{A}^{P})$ 
such that for each phase $i\in\{1,\ldots,P\}$ we have a family 
$\mathscr{A}^{i}=(\mathscr{A}^{i}(t))_{t \in [0,\Tcal]}$ of non-empty open subsets of~$\R^d$
with finite perimeter in~$\R^d$. We also require $(\mathscr{A}^1(t),\ldots,\mathscr{A}^P(t))$ 
to be a partition of~$\R^d$ for all $t\in[0,\Tcal]$, that all except for the $P$th phase have 
finite mass for all $t\in[0,\Tcal]$, and that $\p\mathscr{A}^{i}(t)=\overline{\p^*\mathscr{A}^{i}(t)}$
as well as $\H^{d-1}(\p\mathscr{A}^{i}(t)\setminus\p^*\mathscr{A}^{i}(t))=0$
for all $i\in\{1,\ldots,P\}$ and all $t\in[0,\Tcal]$. Finally, we write $\mathcal{I}_{i,j}(t)$
for the associated interfaces $\p^*\mathscr{A}^{i}(t)\cap\p^*\mathscr{A}^{j}(t)$
between phases $i\neq j\in\{1,\ldots,P\}$ at time~$t\in[0,\Tcal]$, and $n_{\mathcal{I}_{i,j}(t)}$
for the corresponding unit normal vector field pointing from~$\mathscr{A}^{i}(t)$ to~$\mathscr{A}^{j}(t)$.
We then assume that~$\mathscr{A}$ is a calibrated evolution in the
precise sense of \cite[Definition~2]{Fischer-Hensel-Laux-Simon}. In particular,
there exists a family of sufficiently regular and sufficiently integrable vector fields
$((\xi_{i})_{i\in\{1,\ldots,P\}},(\xi_{i,j})_{i,j\in\{1,\ldots,P\},i\neq j},B)$ so that
\begin{align}
\label{eq:calibrationMultiPhase}
\xi_{i,j}(\cdot,t) &= \xi_{i}(\cdot,t) - \xi_{j}(\cdot,t),
\\
\label{eq:extensionMultiPhase}
\xi_{i,j}(\cdot,t) &= n_{\mathcal{I}_{i,j}(t)} \text{ along } \mathcal{I}_{i,j}(t),
\\ \label{eq:lengthControlMultiPhase}
c\min\{1,\mathrm{dist}^2(x,\mathcal{I}_{i,j}(t))\} &\leq 1 - |\xi_{i,j}(x,t)|,
\\ \label{eq:transportXiMultiPhase}
\partial_t\xi_{i,j} + (B\cdot\nabla)\xi_{i,j} + (\nabla B)^\mathsf{T}\xi_{i,j}
&= O\big(\min\{1,\dist(x,\mathcal{I}_{i,j}(t))\}\big),
\\ \label{eq:transportLengthXiMultiPhase}
\xi\cdot\big(\partial_t\xi_{i,j} + (B\cdot\nabla)\xi_{i,j}\big)
&= O\big(\min\{1,\mathrm{dist}^2(x,\mathcal{I}_{i,j}(t))\}\big),
\\ \label{eq:motionByMeanCurvatureMultiPhase}
B\cdot\xi_{i,j} + \nabla\cdot\xi_{i,j} &= O\big(\min\{1,\dist(x,\mathcal{I}(t))\}\big),
\end{align}
for all $(x,t) \in \R^d{\times}[0,\Tcal]$ and some $c\in (0,1]$.
Moreover, we assume that there exists a family of sufficiently regular
and sufficiently integrable weights~$(\vartheta^1,\ldots,\vartheta^P)$ such that
(cf.\ \cite[Definition~4]{Fischer-Hensel-Laux-Simon})
\begin{align}
\label{eq:signCondWeightMultiPhase}
\vartheta^{i}(\cdot,t) < 0 \text{ in } \mathscr{A}^{i}(t),
\quad \vartheta^{i}(\cdot,t) &> 0 \text{ in } \R^d\setminus\overline{\mathscr{A}^{i}(t)},
\\
\label{eq:evolWeightMultiPhase}
\big(\partial_t\vartheta^{i} + (B\cdot\nabla)\vartheta^{i}\big)(x,t) &= 
O\big(\min\{1,\dist(x,\p\mathscr{A}^{i}(t))\}\big),
\\
\label{eq:coercivity0MultiPhase}
c\min\{1,\dist(x,\p\mathscr{A}^{i}(t))\} &\leq |\vartheta^{i}(x,t)|
\leq C\min\{1,\dist(x,\p\mathscr{A}^{i}(t))\}
\end{align}
for some $c\in (0,1]$, $C\geq 1$ and all $(x,t) \in \R^d{\times}[0,\Tcal]$.

It turns out that sufficiently smooth solutions to multi-phase mean curvature
flow are indeed calibrated in this sense, see~\cite[Proposition~6 and Lemma~7]{Fischer-Hensel-Laux-Simon}
for the regime $d=2$ and $P\geq 2$, or~\cite[Theorem~1]{Hensel2021} for the
case $d=3$ and $P=3$. For the purposes of the present work, it turns out that
one needs one additional requirement in form of
\begin{align}
\label{eq:symmetryGradientB}
\big|\big(\xi_{i,j}\otimes(I_d{-}\xi_{i,j}\otimes\xi_{i,j})v + 
(I_d{-}\xi_{i,j}\otimes\xi_{i,j})v\otimes\xi_{i,j}\big):\nabla B\big|
&= O\big(|v|\min\{1,\mathrm{dist}(x,\mathcal{I}_{i,j}(t))\}\big)
\end{align}
for all $(x,t) \in \R^d{\times}[0,\Tcal]$ and all $v\in\R^d$.
The above existence proofs only require a minor extension
to include this additional condition.

We now define the multi-phase analogues of the relative entropy
and the bulk error from~\eqref{eq:relativeEntropy} and~\eqref{eq:bulkError},
respectively, measuring the difference between a De~Giorgi type varifold solution~$(\mu,\chi)$
of multi-phase mean curvature flow and a calibrated evolution~$\mathscr{A}$.
These are defined by
\begin{align}
\label{eq:relativeEntropyMultiPhase}
E[\mu,\chi|\mathscr{A}](t) &:= \int_{\R^d} 1 \,d\omega_t 
- \frac{1}{2}\sum_{i,j=1,\,i\neq j}^{P} \int_{\R^d {\times} \S^{d-1}} 
p \cdot \xi_{i,j}(\cdot,t) \,d\mu^{i,j}_t \geq 0, 
&& t \in [0,\Tcal],
\\ \label{eq:bulkErrorMultiPhase}
E_{\mathrm{bulk}}[\chi|\mathscr{A}](t) &:= \sum_{i=1}^P\int_{\R^d} 
|\chi_{A^{i}(t)} {-} \chi_{\mathscr{A}^{i}(t)}| \, |\vartheta^{i}(\cdot,t)| \,dx, 
&& t \in [0,\Tcal].
\end{align}
Note that by~\eqref{eq:varifoldGrain}--\eqref{eq:varifoldNetwork}
we may equivalently write
\begin{align}
\label{eq:relativeEntropy2MultiPhase}
E[\mu,\chi|\mathscr{A}](t) = \sum_{i=1}^{P} \int_{\R^d} 1 \,d\omega_{t}^{i,i}
+ \frac{1}{2}\sum_{i,j=1,\,i\neq j}^{P} \int_{\R^d {\times} \S^{d-1}} 
1 - p \cdot \xi_{i,j}(\cdot,t) \,d\mu^{i,j}_t.
\end{align}
By the calibration property~\eqref{eq:calibrationMultiPhase},
the anti-symmetry condition~\eqref{eq:compNormals}, the
decomposition~\eqref{eq:varifoldGrain}, and the compatibility condition~\eqref{eq:compatibilityMultiPhase}
we have 
\begin{align*}
&- \frac{1}{2}\sum_{i,j=1,\,i\neq j}^{P} \int_{\R^d {\times} \S^{d-1}}
p \cdot \xi_{i,j}(\cdot,t) \,d\mu^{i,j}_t
\\&
=  \sum_{i,j=1}^{P} \int_{\R^d {\times} \S^{d-1}}
p \cdot \xi_{i}(\cdot,t) \,d\mu^{j,i}_t
=  \sum_{i=1}^{P} \int_{\R^d {\times} \S^{d-1}}
p \cdot \xi_{i}(\cdot,t) \,d\mu^{i}_t
= \sum_{i=1}^{P} \int_{\R^d} (\xi_{i} \cdot \,d\nabla\chi^{i})(\cdot,t)
\end{align*}
so that by an integration by parts 
\begin{align}
\label{eq:repRelEntropyTimeEvolMultiPhase}
- \frac{1}{2}\sum_{i,j=1,\,i\neq j}^{P} \int_{\R^d {\times} \S^{d-1}}
p \cdot \xi_{i,j}(\cdot,t) \,d\mu^{i,j}_t
&= - \sum_{i=1}^{P} \int_{\R^d} \chi^i(\cdot,t)(\nabla\cdot\xi_i)(\cdot,t) \,dx.
\end{align}
Alternatively, writing $n_{i,j}(\cdot,t) := \frac{\nabla\chi_j(\cdot,t)}{|\nabla\chi_j(\cdot,t)|}
= - \frac{\nabla\chi_i(\cdot,t)}{|\nabla\chi_i(\cdot,t)|}$ along $\p^*A^{i}(t)\cap\p^*A^{j}(t)$
and capitalizing on the anti-symmetry $n_{i,j}=-n_{j,i}$ as well as the
calibration property~\eqref{eq:calibrationMultiPhase}, we also have
\begin{align}
\label{eq:repRelEntropyAux}
- \frac{1}{2}\sum_{i,j=1,\,i\neq j}^{P} \int_{\R^d {\times} \S^{d-1}}
p \cdot \xi_{i,j}(\cdot,t) \,d\mu^{i,j}_t
&= - \frac{1}{2}\sum_{i,j=1,\,i\neq j}^{P} 
\int_{\p^*A^{i}(t)\cap\p^*A^{j}(t)} 
(n_{i,j}\cdot\xi_{i,j})(\cdot,t) \,d\H^{d-1}.
\end{align}
Since also by~\eqref{eq:varifoldNetwork} and~\eqref{eq:RadonNikodymPropertyMultiPhase}
\begin{align}
\label{eq:altRepVarifoldEnergy}
\int_{\R^d} 1 \,d\omega_t &= 
\frac{1}{2}\sum_{i=1}^{P} \int_{\R^d} 1 - \rho^{i}(\cdot,t) \,d\omega^{i}_t
+ \frac{1}{2}\sum_{i,j=1,\,i\neq j}^{P} \int_{\p^*A^{i}(t)\cap\p^*A^{j}(t)} 1 \,d\H^{d-1},
\end{align}
a third representation of the relative entropy is therefore given by
\begin{align}
\label{eq:relativeEntropy3MultiPhase}
E[\mu,\chi|\mathscr{A}](t) = \frac{1}{2}\sum_{i=1}^{P} \int_{\R^d} 1 - \rho^{i}(\cdot,t) \,d\omega^{i}_t
+ \frac{1}{2}\sum_{i,j=1,\,i\neq j}^{P} \int_{\p^*A^{i}(t)\cap\p^*A^{j}(t)} 
(1 - n_{i,j}\cdot\xi_{i,j})(\cdot,t) \,d\H^{d-1}.
\end{align}

The three equivalent expressions~\eqref{eq:relativeEntropyMultiPhase},
\eqref{eq:relativeEntropy2MultiPhase} and \eqref{eq:relativeEntropy3MultiPhase}
combined with the length constraint~\eqref{eq:lengthControlMultiPhase}
again imply a number of useful coercivity properties for the relative entropy: 
\begin{align}
\label{eq:coercivity1MultiPhase}
\int_{\R^d} 1 \,d\omega_{t}^{i,i} 
+ \int_{\R^d} 1 - \rho^{i}(\cdot,t) \,d\omega^{i}_t
&\leq C E[\mu,\chi|\mathscr{A}](t),
\\ \label{eq:coercivity2MultiPhase}
\int_{\R^d {\times} \S^{d-1}} 1 - |\langle\lambda^{i,j}_{x,t}\rangle| \,d\omega^{i,j}_t
&\leq C E[\mu,\chi|\mathscr{A}](t),
\\ \label{eq:coercivity3MultiPhase}
\int_{\R^d {\times} \S^{d-1}} |p {-} \xi_{i,j}(\cdot,t)|^2 \,d\mu^{i,j}_t
&\leq C E[\mu,\chi|\mathscr{A}](t),
\\ \label{eq:coercivity4MultiPhase}
\int_{\R^d} \min\{1,\mathrm{dist}^2(\cdot,\mathcal{I}_{i,j}(t))\} \,d\omega^{i,j}_t
&\leq C E[\mu,\chi|\mathscr{A}](t),
\\ \label{eq:coercivity5MultiPhase}
\int_{\p^*A^{i}(t)\cap\p^*A^{j}(t)} |(n_{i,j} {-} \xi_{i,j})(\cdot,t)|^2 \,d\H^{d-1}
&\leq C E[\mu,\chi|\mathscr{A}](t),
\\ \label{eq:coercivity6MultiPhase}
\int_{\p^*A^{i}(t)\cap\p^*A^{j}(t)} \min\{1,\mathrm{dist}^2(\cdot,\mathcal{I}_{i,j}(t))\} \,d\H^{d-1}
&\leq C E[\mu,\chi|\mathscr{A}](t)
\end{align}
for all $i\neq j\in\{1,\ldots,P\}$ and all $t\in [0,\Tcal]$. Here $C=C(\mathscr{A})$ depends on the strong solution, but not on the weak solution. 
Finally, by~\eqref{eq:signCondWeightMultiPhase} and~\eqref{eq:coercivity0MultiPhase}
we also have in terms of the multi-phase bulk error
\begin{align}
\label{eq:coercivity7MultiPhase}
\sum_{i=1}^{P} \int_{\R^d} (\chi_{A^i(t)} {-} \chi_{\mathscr{A}^i(t)}) \vartheta^i(\cdot,t) \,dx
&= E_{\mathrm{bulk}}[\chi|\mathscr{A}](t),
\\ \label{eq:coercivity8MultiPhase}
\int_{\R^d} |\chi_{A^i(t)} {-} \chi_{\mathscr{A}^i(t)}|
\min\{1,\mathrm{dist}(\cdot,\p\mathscr{A}^{i}(t))\} \,dx
&\leq C E_{\mathrm{bulk}}[\chi|\mathscr{A}](t)
\end{align}
for all $i\in\{1,\ldots,P\}$ and all $t\in [0,\Tcal]$.

We have everything in place to formulate and prove the main result of this section.

\begin{theorem}[Weak-strong uniqueness and quantitative stability
	for multi-phase De~Giorgi type varifold solutions]
	\label{theo:weakStrongUniquenessMultiPhase}
	Let $\Tcal\in (0,\infty)$ be a finite time horizon,
	and let~$\mathscr{A}$ be a calibrated evolution on $[0,\Tcal]$ in the above sense.
	Furthermore, let $(\mu,\chi)$ be a De~Giorgi type varifold solution
	for multi-phase mean curvature flow with initial data~$(\mu_0,\chi_0)$ in the sense of
	Definition~\ref{def:multiPhaseDeGiorgiVarifoldSolutions}.
	
	Defining the relative entropy~$E[\mu,\chi|\mathscr{A}]$ 
	and the bulk error~$E_{\mathrm{bulk}}[\chi|\mathscr{A}]$ 
	by~\eqref{eq:relativeEntropyMultiPhase} and~\eqref{eq:bulkErrorMultiPhase}, respectively, 
	we then have stability estimates for these quantities in form of
	\begin{align}
	\label{eq:stabilityRelEntropyMultiPhase}
	E[\mu,\chi|\mathscr{A}](T) 
	&\leq E[\mu_0,\chi_0|\mathscr{A}(0)]
	+ C \int_{0}^{T} E[\mu,\chi|\mathscr{A}](t) \,dt, 
	\\ \label{eq:stabilityBulkErrorMultiPhase}
	E_{\mathrm{bulk}}[\chi|\mathscr{A}](T)
	&\leq E_{\mathrm{bulk}}[\chi_0|\mathscr{A}(0)]
	+ E[\mu_0,\chi_0|\mathscr{A}(0)]
	+ \int_{0}^{T} E_{\mathrm{bulk}}[\chi|\mathscr{A}](t)
	{+} E[\mu,\chi|\mathscr{A}](t) \,dt
	\end{align}
	for some constant $C=C(\mathscr{A})>0$
	and almost every $T \in [0,\Tcal]$.
	
	In particular, if $\chi^i_0=\chi_{\mathscr{A}^i(0)}$
	and $\mu^i_0 = |\nabla\chi^i_0|\otimes(\delta_{n^{i}_0(x)})_{x\in\R^d}$
	hold true with $n^{i}_0 = \frac{\nabla\chi^{i}_0}{|\nabla\chi^{i}_0|}$ for all $i\in\{1,\ldots,P\}$, 
	then the De~Giorgi type varifold solution~$(\mu,\chi)$ 
	reduces to the calibrated evolution~$\mathscr{A}$ 
	in the sense that for almost every $t \in [0,\Tcal]$ and all $i\in\{1,\ldots,P\}$
	\begin{align}
	\label{eq:weakStrong1MultiPhase}
	\chi_{A^{i}(t)} &= \chi_{\mathscr{A}^i(t)}
	\text{ almost everywhere in } \R^d,
	\\ \label{eq:weakStrong2MultiPhase}
	\mu^i_t &= |\nabla\chi_{A^{i}(t)}|\otimes(\delta_{n^i(x,t)})_{x\in\R^d},
	\end{align}
	where $n^i(\cdot,t)=\frac{\nabla\chi_{A^{i}(t)}}{|\nabla\chi_{A^{i}(t)}|}$.
\end{theorem}

\begin{proof}
	We again split the proof into three steps. Naturally, the argument is of similar structure
	as the one for the proof of Theorem~\ref{theo:weakStrongUniqueness}. We thus focus
	in the following on the key differences.
	
	\textit{Step 1: Proof of~\emph{\eqref{eq:stabilityRelEntropyMultiPhase}}.}
	The first major goal is to arrive at the multi-phase analogue of the estimate~\eqref{eq:stabilityEstimateAux6}.
	Starting point for this is to rewrite the relative entropy
	thanks to~\eqref{eq:repRelEntropyTimeEvolMultiPhase} as
	\begin{align*}
	E[\mu,\chi|\mathscr{A}](T) 
	= \int_{\R^d} 1 \,d\omega_{T}  - \sum_{i=1}^{P} \int_{\R^d} \chi^i(\cdot,T)(\nabla\cdot\xi_i)(\cdot,T) \,dx,
	\end{align*}
	so that the De~Giorgi type inequality~\eqref{eq:DeGiorgiInequalityMultiPhase},
	the evolution equations~\eqref{eq:evolMultiPhase}, and an integration by parts imply
	\begin{align}
	\label{eq:auxMultiPhase1}
	E[\mu,\chi|\mathscr{A}](T) &\leq E[\mu_0,\chi_0|\mathscr{A}(0)]
	- \frac{1}{2} \sum_{i=1}^P \int_{0}^{T} \int_{\R^d} \frac{1}{2}|V^i|^2 \,d\omega^i_t dt
	- \frac{1}{2} \int_{0}^{T} \int_{\R^d} |\mathbf{H}|^2 \,d\omega_t dt
	\\&~~~ \nonumber
	+ \sum_{i=1}^P \int_{0}^{T} \int_{\R^d} V^i (\nabla\cdot\xi_i) \,d\omega^i_t dt
	+ \sum_{i=1}^P \int_{0}^{T} \int_{\R^d} \frac{\nabla\chi_i}{|\nabla\chi_i|} 
	\cdot \partial_t\xi_{i} \,d|\nabla\chi_i| dt.
	\end{align}
	Splitting first $|\nabla\chi_i(\cdot,t)| = \sum_{j=1,j\neq i}^{P}\H^{d-1}
	\llcorner S_{i,j}(t)$, where for the purposes of the whole proof
	we abbreviate $S_{i,j}(t):=\p^*A^{i}(t)\cap\p^*A^{j}(t)$, and recalling second that
	$n_{i,j}(\cdot,t) = \frac{\nabla\chi_j(\cdot,t)}{|\nabla\chi_j(\cdot,t)|}
	= - \frac{\nabla\chi_i(\cdot,t)}{|\nabla\chi_i(\cdot,t)|} = -n_{j,i}(\cdot,t)$ along $\p^*A^{i}(t)\cap\p^*A^{j}(t)$,
	the calibration property~\eqref{eq:calibrationMultiPhase} allows to compute
	\begin{align}
	\label{eq:auxMultiPhase2}
	\sum_{i=1}^P \int_{0}^{T} \int_{\R^d} \frac{\nabla\chi_i}{|\nabla\chi_i|} 
	\cdot \partial_t\xi_{i} \,d|\nabla\chi_i| dt
	= - \frac{1}{2} \sum_{i,j=1,\,i\neq j}^{P} \int_{0}^{T}
	\int_{S_{i,j}(t)} n_{i,j} \cdot \partial_t\xi_{i,j} \,d\H^{d-1} dt.
	\end{align}
	In view of the estimates~\eqref{eq:transportXiMultiPhase}--\eqref{eq:transportLengthXiMultiPhase}
	and the coercivity property~\eqref{eq:coercivity6MultiPhase}, the same argument
	as for~\eqref{eq:stabilityEstimateAux2} entails for all $i,j\in\{1,\ldots,P\}$ with $i\neq j$
	\begin{align}
	\nonumber
	&- \int_{0}^{T} \int_{S_{i,j}(t)} n_{i,j} \cdot \partial_t\xi_{i,j} \,d\H^{d-1} dt
	\\& \label{eq:auxMultiPhase3}
	\leq \int_{0}^{T} \int_{S_{i,j}(t)} \xi_{i,j} \otimes (n_{i,j} {-} \xi_{i,j}) : \nabla B \,d\H^{d-1} dt
	+ \int_{0}^{T} \int_{S_{i,j}(t)} n_{i,j} \cdot (B \cdot \nabla)\xi_{i,j} \,d\H^{d-1} dt
	\\&~~~ \nonumber
	+ C \int_{0}^{T} E[\mu,\chi|\mathscr{A}](t) \,dt.
	\end{align}
	Observe next that the structurally analogous argument enabling~\eqref{eq:auxMultiPhase2} 
	further guarantees
	\begin{align*}
	\frac{1}{2} \sum_{i,j=1,\,i\neq j}^{P} &\int_{0}^{T}
	\int_{S_{i,j}(t)} n_{i,j} \cdot \big(\nabla \cdot (\xi_{i,j} \otimes B)\big) \,d\H^{d-1} dt
	\\ &= - \sum_{i=1}^P \int_{0}^{T} \int_{\R^d} \frac{\nabla\chi_i}{|\nabla\chi_i|} 
	\cdot \big(\nabla \cdot (\xi_i \otimes B)\big) \,d|\nabla\chi_i| dt.
	\end{align*}
	Of course, this identity is also valid with the roles of~$B$ and~$\xi_{i,j}$ 
	as well as $B$ and $\xi_i$ reversed, respectively, and the same arguments moreover show
	\begin{align*}
	\frac{1}{2} \sum_{i,j=1,\,i\neq j}^{P} \int_{0}^{T}
	\int_{S_{i,j}(t)} (n_{i,j} \cdot B) (\nabla\cdot\xi_{i,j}) \,d\H^{d-1} dt
	= - \sum_{i=1}^P \int_{0}^{T} \int_{\R^d} \Big(\frac{\nabla\chi_i}{|\nabla\chi_i|} 
	\cdot B\Big) (\nabla\cdot\xi_i) \,d|\nabla\chi_i| dt.
	\end{align*} 
	Due to the compatibility condition~\eqref{eq:compatibilityMultiPhase}
	and the calibration property~\eqref{eq:calibrationMultiPhase},
	as well as by replacing the anti-symmetry condition $n_{i,j}=-n_{j,i}$ and
	the splitting $|\nabla\chi_i(\cdot,t)| = \sum_{j=1,j\neq i}^{P}\H^{d-1}
	\llcorner S_{i,j}(t)$ by the anti-symmetry condition of~\eqref{eq:compNormals}
	and the decomposition~\eqref{eq:varifoldGrain}, respectively,
	the identity of the previous display may in turn be updated
	by the previous arguments to
	\begin{align}
	\label{eq:auxMultiPhase6}
	&\frac{1}{2} \sum_{i,j=1,\,i\neq j}^{P} \int_{0}^{T}
	\int_{S_{i,j}(t)} (n_{i,j} \cdot B) (\nabla\cdot\xi_{i,j}) \,d\H^{d-1} dt
	\\& \nonumber
	= - \sum_{i=1}^P \int_{0}^{T} \int_{\R^d{\times}\S^{d-1}} (p\cdot B) (\nabla\cdot\xi_i) \,d\mu^{i}_t dt
	= \frac{1}{2} \sum_{i,j=1,\,i\neq j}^{P} \int_{0}^{T}
	\int_{\R^d{\times}\S^{d-1}} (p \cdot B) (\nabla\cdot\xi_{i,j}) \,d\mu^{i,j}_t dt.
	\end{align}
	In particular, the argument in favor of~\eqref{eq:stabilityEstimateAux4} and~\eqref{eq:stabilityEstimateAux5}
	now implies because of the previous three displays, \eqref{eq:relativeEntropy3MultiPhase} 
	and~\eqref{eq:RadonNikodymPropertyMultiPhase}
	\begin{align}
	\nonumber
	&\frac{1}{2} \sum_{i,j=1,\,i\neq j}^{P} \int_{0}^{T}
	\int_{S_{i,j}(t)} n_{i,j} \cdot (B \cdot \nabla)\xi_{i,j} \,d\H^{d-1} dt
	\\& \label{eq:auxMultiPhase4}
	\leq \frac{1}{2} \sum_{i,j=1,\,i\neq j}^{P} \int_{0}^{T}
	\int_{\R^d{\times}\S^{d-1}} (p \cdot B) (\nabla\cdot\xi_{i,j}) \,d\mu^{i,j}_t dt
	- \int_{0}^{T} \int_{\R^d {\times} \S^{d-1}}
	(I_d {-} p \otimes p) : \nabla B \,d\mu_t dt
	\\&~~~ \nonumber
	+  \frac{1}{2} \sum_{i,j=1,\,i\neq j}^{P} \int_{0}^{T} 
	\int_{S_{i,j}(t)} n_{i,j} \otimes \xi_{i,j} : \nabla B \,d\H^{d-1} dt
	- \int_{0}^{T} \int_{\R^d {\times} \S^{d-1}} p \otimes p : \nabla B \,d\mu_t dt.
	\end{align}
	Hence, combining the estimates~\eqref{eq:auxMultiPhase3} and~\eqref{eq:auxMultiPhase4} ensures
	\begin{align}
	\label{eq:auxMultiPhase5}
	&- \frac{1}{2} \sum_{i,j=1,\,i\neq j}^{P} \int_{0}^{T} 
	\int_{S_{i,j}(t)} n_{i,j} \cdot \partial_t\xi_{i,j} \,d\H^{d-1} dt
	\\& \nonumber
	\leq \frac{1}{2} \sum_{i,j=1,\,i\neq j}^{P} \int_{0}^{T}
	\int_{\R^d{\times}\S^{d-1}} (p \cdot B) (\nabla\cdot\xi_{i,j}) \,d\mu^{i,j}_t dt
	- \int_{0}^{T} \int_{\R^d {\times} \S^{d-1}}
	(I_d {-} p \otimes p) : \nabla B \,d\mu_t dt
	\\&~~~ \nonumber
	- \frac{1}{2} \sum_{i,j=1,\,i\neq j}^{P} \int_{0}^{T} 
	\int_{S_{i,j}(t)} \xi_{i,j} \otimes \xi_{i,j} : \nabla B \,d\H^{d-1} dt
	- \int_{0}^{T} \int_{\R^d {\times} \S^{d-1}} p \otimes p : \nabla B \,d\mu_t dt
	\\&~~~ \nonumber
	+  \frac{1}{2} \sum_{i,j=1,\,i\neq j}^{P} \int_{0}^{T} 
	\int_{S_{i,j}(t)} (n_{i,j} \otimes \xi_{i,j} {+} \xi_{i,j} \otimes n_{i,j}) : \nabla B \,d\H^{d-1} dt
	+ C \int_{0}^{T} E[\mu,\chi|\mathscr{A}](t) \,dt.
	\end{align}
	We further post-process this estimate by noting first that the 
	argument of~\eqref{eq:auxMultiPhase6} also yields
	\begin{align}
	\label{eq:auxMultiPhase7}
	&\frac{1}{2} \sum_{i,j=1,\,i\neq j}^{P} \int_{0}^{T} 
	\int_{S_{i,j}(t)} n_{i,j} \otimes \xi_{i,j} : \nabla B \,d\H^{d-1} dt
	= \frac{1}{2} \sum_{i,j=1,\,i\neq j}^{P} \int_{0}^{T} 
	\int_{\R^d{\times}\S^{d-1}} p \otimes \xi_{i,j} : \nabla B \,d\mu^{i,j}_t dt,
	\end{align}
	where we remark that the same identity also holds true with the roles of $n_{i,j}$ and $\xi_{i,j}$
	as well as $p$ and $\xi_{i,j}$ reversed, respectively. For the next step, we claim that
	\begin{align}
	\nonumber
	&- \frac{1}{2} \sum_{i,j=1,\,i\neq j}^{P} \int_{0}^{T} 
	\int_{S_{i,j}(t)} \xi_{i,j} \otimes \xi_{i,j} : \nabla B \,d\H^{d-1} dt
	\\& \label{eq:auxMultiPhase8}
	\leq - \frac{1}{2} \sum_{i,j=1,\,i\neq j}^{P} \int_{0}^{T} 
	\int_{\R^d{\times}\S^{d-1}} \xi_{i,j} \otimes \xi_{i,j} : \nabla B \,d\mu^{i,j}_t dt
	+ C \int_{0}^{T} E[\mu,\chi|\mathscr{A}](t) \,dt.
	\end{align}
	Before we prove~\eqref{eq:auxMultiPhase8}, let us first observe that
	inserting~\eqref{eq:auxMultiPhase7}--\eqref{eq:auxMultiPhase8}
	back into~\eqref{eq:auxMultiPhase5} and then estimating the resulting
	quadratic term in the difference between the calibration normals~$\xi_{i,j}$
	and the varifold normal~$p$ from above by~\eqref{eq:coercivity3MultiPhase},
	we consequently obtain from~\eqref{eq:auxMultiPhase1} and~\eqref{eq:auxMultiPhase2}
	\begin{align}
	\label{eq:auxMultiPhase9}
	E[\mu,\chi|\mathscr{A}](T) &\leq E[\mu_0,\chi_0|\mathscr{A}(0)]
	- \frac{1}{2} \sum_{i=1}^P \int_{0}^{T} \int_{\R^d} \frac{1}{2}|V^i|^2 \,d\omega^i_t dt
	- \frac{1}{2} \int_{0}^{T} \int_{\R^d} |\mathbf{H}|^2 \,d\omega_t dt
	\\&~~~ \nonumber
	+ \sum_{i=1}^P \int_{0}^{T} \int_{\R^d} V^i (\nabla\cdot\xi_i) \,d\omega^i_t dt
	- \int_{0}^{T} \int_{\R^d {\times} \S^{d-1}}
	(I_d {-} p \otimes p) : \nabla B \,d\mu_t dt
	\\&~~~ \nonumber
	+ \frac{1}{2} \sum_{i,j=1,\,i\neq j}^{P} \int_{0}^{T}
	\int_{\R^d{\times}\S^{d-1}} (p \cdot B) (\nabla\cdot\xi_{i,j}) \,d\mu^{i,j}_t dt
	+ C \int_{0}^{T} E[\mu,\chi|\mathscr{A}](t) \,dt.
	\end{align}
	For a proof of~\eqref{eq:auxMultiPhase8}, a purely algebraic manipulation
	and $|\xi_{i,j} \cdot (\xi_{i,j} {-} n_{i,j})| \leq 3( 1 - n_{i,j}\cdot\xi_{i,j})$ gives
	\begin{align}
	\nonumber
	2\,\xi_{i,j} \otimes \xi_{i,j}
	&= \big( (\xi_{i,j} {-} n_{i,j}) \otimes \xi_{i,j}
	+ \xi_{i,j} \otimes (\xi_{i,j} {-} n_{i,j}) \big)
	+ n_{i,j} \otimes \xi_{i,j} + \xi_{i,j} \otimes n_{i,j}
	\\& \label{eq:auxMultiPhase10}
	= \big( (I_d{-}\xi_{i,j}\otimes\xi_{i,j})(\xi_{i,j} {-} n_{i,j}) \otimes \xi_{i,j}
	+ \xi_{i,j} \otimes (I_d{-}\xi_{i,j}\otimes\xi_{i,j})(\xi_{i,j} {-} n_{i,j}) \big)
	\\&~~~ \nonumber
	+ n_{i,j} \otimes \xi_{i,j} + \xi_{i,j} \otimes n_{i,j}
	+ O(1 - n_{i,j}\cdot\xi_{i,j}).
	\end{align}
	Hence, plugging in~\eqref{eq:auxMultiPhase7} and
	estimating based on~\eqref{eq:symmetryGradientB}, \eqref{eq:relativeEntropy3MultiPhase},
	\eqref{eq:coercivity5MultiPhase} and~\eqref{eq:coercivity6MultiPhase} we get
	\begin{align*}
	&- \frac{1}{2} \sum_{i,j=1,\,i\neq j}^{P} \int_{0}^{T} 
	\int_{S_{i,j}(t)} \xi_{i,j} \otimes \xi_{i,j} : \nabla B \,d\H^{d-1} dt
	\\& 
	\leq - \frac{1}{2} \sum_{i,j=1,\,i\neq j}^{P} \int_{0}^{T} 
	\int_{\R^d{\times}\S^{d-1}} (p \otimes \xi_{i,j} + \xi_{i,j} \otimes p) : \nabla B \,d\mu^{i,j}_t dt
	+ C \int_{0}^{T} E[\mu,\chi|\mathscr{A}](t) \,dt.
	\end{align*}
	Reading~\eqref{eq:auxMultiPhase10} with~$n_{i,j}$ replaced by~$p$
	and estimating the resulting error terms this time 
	by means of~\eqref{eq:symmetryGradientB}, \eqref{eq:relativeEntropy2MultiPhase},
	\eqref{eq:coercivity3MultiPhase} and~\eqref{eq:coercivity4MultiPhase},
	we deduce that the previous display implies the asserted estimate~\eqref{eq:auxMultiPhase8}.
	In particular, we may continue with~\eqref{eq:auxMultiPhase9}.
	
	To this end, relying on the conditions of~\eqref{eq:compNormalSpeeds1} and~\eqref{eq:compNormalSpeeds2}
	and yet again the calibration property~\eqref{eq:calibrationMultiPhase}, we have
	\begin{align*}
	\sum_{i=1}^P \int_{0}^{T} \int_{\R^d} V^i (\nabla\cdot\xi_i) \,d\omega^i_t dt
	= \frac{1}{2} \sum_{i,j=1,\,i\neq j}^{P} \int_{0}^{T}
	\int_{\R^d} V^i (\nabla\cdot\xi_{i,j}) \,d\omega^{i,j}_t dt.
	\end{align*}
	We then insert the decomposition~\eqref{eq:varifoldGrain},
	complete squares twice for off-diagonal terms, and estimate the diagonal terms with the
	help of Young's inequality and the coercivity property~\eqref{eq:coercivity1MultiPhase}
	to obtain together with the previous display and finally~\eqref{eq:motionByMeanCurvatureMultiPhase}
	as well as~\eqref{eq:coercivity4MultiPhase} that
	\begin{align}
	\nonumber
	&- \frac{1}{2} \sum_{i=1}^P \int_{0}^{T} \int_{\R^d} \frac{1}{2}|V^i|^2 \,d\omega^i_t dt
	+ \sum_{i=1}^P \int_{0}^{T} \int_{\R^d} V^i (\nabla\cdot\xi_i) \,d\omega^i_t dt
	\\& \nonumber
	+ \frac{1}{2} \sum_{i,j=1,\,i\neq j}^{P} \int_{0}^{T}
	\int_{\R^d{\times}\S^{d-1}} (p \cdot B) (\nabla\cdot\xi_{i,j}) \,d\mu^{i,j}_t dt
	\\& \nonumber
	\leq - \frac{1}{4} \sum_{i=1}^P \int_{0}^{T} \int_{\R^d} \frac{1}{2}|V^i|^2 \,d\omega^{i,i}_t dt
	- \frac{1}{2} \sum_{i,j=1,\,i\neq j}^{P} \int_{0}^{T}
	\int_{\R^d} \frac{1}{2} |V^{i} {-} (\nabla\cdot\xi_{i,j})|^2 \,d\omega^{i,j}_t dt
	\\&~~~ \nonumber
	+ \frac{1}{2} \sum_{i,j=1,\,i\neq j}^{P} \int_{0}^{T}
	\int_{\R^d{\times}\S^{d-1}} (p \cdot B) (\nabla\cdot\xi_{i,j}) \,d\mu^{i,j}_t dt
	\\&~~~ \nonumber
	+ \frac{1}{2} \sum_{i,j=1,\,i\neq j}^{P} \int_{0}^{T}
	\int_{\R^d} \frac{1}{2} (\nabla\cdot\xi_{i,j})^2 \,d\omega^{i,j}_t dt
	+ C \int_{0}^{T} E[\mu,\chi|\mathscr{A}](t) \,dt
	\\& \nonumber
	\leq - \frac{1}{4} \sum_{i=1}^P \int_{0}^{T} \int_{\R^d} \frac{1}{2}|V^i|^2 \,d\omega^{i,i}_t dt
	- \frac{1}{2} \sum_{i,j=1,\,i\neq j}^{P} \int_{0}^{T}
	\int_{\R^d} \frac{1}{2} |V^{i} {-} (\nabla\cdot\xi_{i,j})|^2 \,d\omega^{i,j}_t dt
	\\&~~~ \nonumber
	+ \frac{1}{2} \sum_{i,j=1,\,i\neq j}^{P} \int_{0}^{T}
	\int_{\R^d} \frac{1}{2} |(B\cdot\xi_{i,j}) {+} (\nabla\cdot\xi_{i,j})|^2 \,d\omega^{i,j}_t dt
	- \frac{1}{2} \sum_{i,j=1,\,i\neq j}^{P} \int_{0}^{T}
	\int_{\R^d} \frac{1}{2} (B\cdot\xi_{i,j})^2 \,d\omega^{i,j}_t dt
	\\&~~~ \nonumber
	+ \frac{1}{2} \sum_{i,j=1,\,i\neq j}^{P} \int_{0}^{T}
	\int_{\R^d} B\cdot(p {-} \xi_{i,j})(\nabla\cdot\xi_{i,j}) \,d\omega^{i,j}_t dt
	+ C \int_{0}^{T} E[\mu,\chi|\mathscr{A}](t) \,dt
	\\& \label{eq:auxMultiPhase11}
	\leq - \frac{1}{4} \sum_{i=1}^P \int_{0}^{T} \int_{\R^d} \frac{1}{2}|V^i|^2 \,d\omega^{i,i}_t dt
	- \frac{1}{2} \sum_{i,j=1,\,i\neq j}^{P} \int_{0}^{T}
	\int_{\R^d} \frac{1}{2} |V^{i} {-} (\nabla\cdot\xi_{i,j})|^2 \,d\omega^{i,j}_t dt
	\\&~~~ \nonumber
	+ \frac{1}{2} \sum_{i,j=1,\,i\neq j}^{P} \int_{0}^{T}
	\int_{\R^d} B\cdot(p {-} \xi_{i,j})(\nabla\cdot\xi_{i,j}) \,d\omega^{i,j}_t dt
	- \frac{1}{2} \sum_{i,j=1,\,i\neq j}^{P} \int_{0}^{T}
	\int_{\R^d} \frac{1}{2} (B\cdot\xi_{i,j})^2 \,d\omega^{i,j}_t dt
	\\&~~~ \nonumber
	+ C \int_{0}^{T} E[\mu,\chi|\mathscr{A}](t) \,dt.
	\end{align}
	Moreover, by~\eqref{eq:varifoldGrain} and \eqref{eq:varifoldNetwork}
	as well as~\eqref{eq:weakCurvatureDirection}
	in form of the identity $|\mathbf{H}(x,t)|^2|\langle\lambda^{i,j}_{x,t}\rangle|^2 
	= |\mathbf{H}(x,t)\cdot\langle\lambda^{i,j}_{x,t}\rangle|^2$,
	we rewrite one of the helpful dissipation-terms as
	\begin{align}
	\nonumber
	&- \frac{1}{2} \int_{0}^{T} \int_{\R^d} |\mathbf{H}|^2 \,d\omega_t dt
	\\& \label{eq:auxMultiPhase12}
	= - \frac{1}{2} \sum_{i=1}^P \int_{0}^{T} \int_{\R^d} |\mathbf{H}|^2 \,d\omega^{i,i}_t dt
	- \frac{1}{2} \sum_{i,j=1,\,i\neq j}^{P} \int_{0}^{T}
	\int_{\R^d} \frac{1}{2} |\mathbf{H}|^2(1{-}|\langle\lambda^{i,j}_{x,t}\rangle|^2) \,d\omega^{i,j}_t dt
	\\&~~~ \nonumber
	- \frac{1}{2} \sum_{i,j=1,\,i\neq j}^{P} \int_{0}^{T}
	\int_{\R^d} \frac{1}{2} |\mathbf{H}(x,t)\cdot\langle\lambda^{i,j}_{x,t}\rangle|^2 \,d\omega^{i,j}_t dt.
	\end{align}
	Since we may exploit~\eqref{eq:weakCurvatureMultiPhase}
	and then decompose in a similar fashion using~\eqref{eq:weakCurvatureDirection} 
	this time directly, we also get
	\begin{align}
	\nonumber
	&- \int_{0}^{T} \int_{\R^d {\times} \S^{d-1}}
	(I_d {-} p \otimes p) : \nabla B \,d\mu_t dt
	\\& \label{eq:auxMultiPhase13}
	= \sum_{i=1}^P \int_{0}^{T} \int_{\R^d} \mathbf{H} \cdot B \,d\omega^{i,i}_t dt
	+ \frac{1}{2} \sum_{i,j=1,\,i\neq j}^{P} \int_{0}^{T}
	\int_{\R^d} (1{-}|\langle\lambda^{i,j}_{x,t}\rangle|^2)\mathbf{H} \cdot B \,d\omega^{i,j}_t dt
	\\&~~~ \nonumber
	+ \frac{1}{2} \sum_{i,j=1,\,i\neq j}^{P} \int_{0}^{T}
	\int_{\R^d} (\mathbf{H} \cdot \langle\lambda^{i,j}_{x,t}\rangle) B \cdot p \,d\mu^{i,j}_t dt
	\end{align}
	Thanks to the control provided by~\eqref{eq:coercivity1MultiPhase} and~\eqref{eq:coercivity2MultiPhase},
	the bounds $1{-}|\langle\lambda^{i,j}_{x,t}\rangle| \leq 1{-}|\langle\lambda^{i,j}_{x,t}\rangle|^2
	\leq 2(1{-}|\langle\lambda^{i,j}_{x,t}\rangle|)$, and an absorption argument, adding \eqref{eq:auxMultiPhase12}
	to~\eqref{eq:auxMultiPhase13} thus results in the estimate
	\begin{align*}
	&- \frac{1}{2} \int_{0}^{T} \int_{\R^d} |\mathbf{H}|^2 \,d\omega_t dt
	- \int_{0}^{T} \int_{\R^d {\times} \S^{d-1}}
	(I_d {-} p \otimes p) : \nabla B \,d\mu_t dt
	\\&
	\leq - \frac{1}{4} \sum_{i=1}^P \int_{0}^{T} \int_{\R^d} |\mathbf{H}|^2 \,d\omega^{i,i}_t dt
	- \frac{1}{4} \sum_{i,j=1,\,i\neq j}^{P} \int_{0}^{T}
	\int_{\R^d} \frac{1}{2} |\mathbf{H}|^2(1{-}|\langle\lambda^{i,j}_{x,t}\rangle|^2) \,d\omega^{i,j}_t dt
	\\&~~~
	- \frac{1}{2} \sum_{i,j=1,\,i\neq j}^{P} \int_{0}^{T}
	\int_{\R^d} \frac{1}{2} |\mathbf{H}(x,t)\cdot\langle\lambda^{i,j}_{x,t}\rangle|^2 \,d\omega^{i,j}_t dt
	+ \frac{1}{2} \sum_{i,j=1,\,i\neq j}^{P} \int_{0}^{T}
	\int_{\R^d} (\mathbf{H} \cdot \langle\lambda^{i,j}_{x,t}\rangle) B \cdot p \,d\mu^{i,j}_t dt
	\\&~~~ \nonumber
	+ C \int_{0}^{T} E[\mu,\chi|\mathscr{A}](t) \,dt.
	\end{align*}
	Hence, as a consequence of the previous display, completing the square,
	adding zero, recalling the control provided by~\eqref{eq:coercivity3MultiPhase},
	and finally performing yet another absorption argument, it holds
	\begin{align}
	\nonumber
	&- \frac{1}{2} \int_{0}^{T} \int_{\R^d} |\mathbf{H}|^2 \,d\omega_t dt
	- \int_{0}^{T} \int_{\R^d {\times} \S^{d-1}}
	(I_d {-} p \otimes p) : \nabla B \,d\mu_t dt
	\\& \nonumber
	\leq - \frac{1}{2} \sum_{i,j=1,\,i\neq j}^{P} \int_{0}^{T}
	\int_{\R^d} \frac{1}{2} \big|\big(\mathbf{H}(x,t)\cdot\langle\lambda^{i,j}_{x,t}\rangle\big)
	{-} (B\cdot\xi_{i,j})\big|^2 \,d\omega^{i,j}_t dt
	\\&~~~ \nonumber
	+ \frac{1}{2} \sum_{i,j=1,\,i\neq j}^{P} \int_{0}^{T}
	\int_{\R^d} \frac{1}{2} (B\cdot\xi_{i,j})^2 \,d\omega^{i,j}_t dt
	\\&~~~ \nonumber
	- \frac{1}{2} \sum_{i,j=1,\,i\neq j}^{P} \int_{0}^{T}
	\int_{\R^d} \big(\big(\mathbf{H} \cdot \langle\lambda^{i,j}_{x,t}\rangle\big) 
	{-} (B\cdot\xi_{i,j})\big) B \cdot (\xi_{i,j} {-} p) \,d\mu^{i,j}_t dt
	\\&~~~ \nonumber
	- \frac{1}{2} \sum_{i,j=1,\,i\neq j}^{P} \int_{0}^{T}
	\int_{\R^d} (B\cdot\xi_{i,j}) B \cdot (\xi_{i,j} {-} p) \,d\mu^{i,j}_t dt
	+ C \int_{0}^{T} E[\mu,\chi|\mathscr{A}](t) \,dt
	\\& \label{eq:auxMultiPhase14}
	\leq  \frac{1}{2} \sum_{i,j=1,\,i\neq j}^{P} \int_{0}^{T}
	\int_{\R^d} \frac{1}{2} (B\cdot\xi_{i,j})^2 \,d\omega^{i,j}_t dt
	\\&~~~ \nonumber
	- \frac{1}{2} \sum_{i,j=1,\,i\neq j}^{P} \int_{0}^{T}
	\int_{\R^d} (B\cdot\xi_{i,j}) B \cdot (\xi_{i,j} {-} p) \,d\mu^{i,j}_t dt
	+ C \int_{0}^{T} E[\mu,\chi|\mathscr{A}](t) \,dt.
	\end{align}
	Adding~\eqref{eq:auxMultiPhase11} to~\eqref{eq:auxMultiPhase14},
	observing that as a result of~\eqref{eq:motionByMeanCurvatureMultiPhase},
	\eqref{eq:coercivity3MultiPhase} and~\eqref{eq:coercivity4MultiPhase}
	\begin{align*}
	- \frac{1}{2} \sum_{i,j=1,\,i\neq j}^{P} \int_{0}^{T}
	\int_{\R^d} \big((B\cdot\xi_{i,j}) {+} (\nabla\cdot\xi_{i,j})\big) B \cdot (\xi_{i,j} {-} p) \,d\mu^{i,j}_t dt
	\leq C \int_{0}^{T} E[\mu,\chi|\mathscr{A}](t)  dt,
	\end{align*}
	we may thus upgrade the preliminary estimate~\eqref{eq:auxMultiPhase9} to
	\begin{align}
	\label{eq:auxMultiPhase15}
	E[\mu,\chi|\mathscr{A}](T) &\leq E[\mu_0,\chi_0|\mathscr{A}(0)]
	- \frac{1}{2} \sum_{i,j=1,\,i\neq j}^{P} \int_{0}^{T}
	\int_{\R^d} \frac{1}{2} |V^{i} {-} (\nabla\cdot\xi_{i,j})|^2 \,d\omega^{i,j}_t dt
	\\&~~~ \nonumber
	- \frac{1}{4} \sum_{i=1}^P \int_{0}^{T} \int_{\R^d} \frac{1}{2}|V^i|^2 \,d\omega^{i,i}_t dt
	+ C \int_{0}^{T} E[\mu,\chi|\mathscr{A}](t) \,dt.
	\end{align}
	This, however, is already a stronger form of what is claimed by~\eqref{eq:stabilityRelEntropyMultiPhase}.
	
	\textit{Step 2: Proof of~\emph{\eqref{eq:stabilityBulkErrorMultiPhase}}.}
	Analogous to~\eqref{eq:stabilityEstimateAux13} and~\eqref{eq:stabilityEstimateAux14},
	relying this time on~\eqref{eq:evolWeightMultiPhase}, the lower bound from~\eqref{eq:coercivity0MultiPhase},
	the coercivity property~\eqref{eq:coercivity8MultiPhase}, and the
	compatibility condition~\eqref{eq:compatibilityMultiPhase}, we obtain
	\begin{align*}
	E_{\mathrm{bulk}}[\chi|\mathscr{A}](T)
	&\leq E_{\mathrm{bulk}}[\chi_0|\mathscr{A}(0)]	
	- \sum_{i=1}^{P} \int_0^T \int_{\R^d} V^{i}\vartheta^{i} \,d\omega^{i}_t dt
	\\&~~~
	+  \sum_{i=1}^{P} \int_{0}^{T} \int_{\R^d {\times} \S^{d-1}} 
	(p \cdot B) \vartheta^{i} \,d\mu^{i}_t dt
	+ C \int_0^T E_{\mathrm{bulk}}[\chi|\mathscr{A}](t) \,dt.
	\end{align*}
	Due to adding zero, the upper bound from~\eqref{eq:coercivity0MultiPhase},
	the decomposition~\eqref{eq:varifoldGrain},
	and the coercivity properties~\eqref{eq:coercivity1MultiPhase}, 
	\eqref{eq:coercivity3MultiPhase} and~\eqref{eq:coercivity4MultiPhase},
	it further holds
	\begin{align*}
	&\sum_{i=1}^{P} \int_{0}^{T} \int_{\R^d {\times} \S^{d-1}} 
	(p \cdot B) \vartheta^{i} \,d\mu^{i}_t dt
	\\&
	\leq - \sum_{i,j=1,\,i\neq j}^{P} \int_{0}^{T} \int_{\R^d {\times} \S^{d-1}} 
	(p \cdot B) \vartheta^{i} \,d\mu^{i,j}_t dt
	+ C \int_{0}^{T} E[\mu,\chi|\mathscr{A}](t) \,dt
	\\&
	\leq - \sum_{i,j=1,\,i\neq j}^{P} \int_{0}^{T} \int_{\R^d {\times} \S^{d-1}} 
	(\xi_{i,j} \cdot B) \vartheta^{i} \,d\omega^{i,j}_t dt
	+ C \int_{0}^{T} E[\mu,\chi|\mathscr{A}](t) \,dt
	\\&
	\leq \sum_{i,j=1,\,i\neq j}^{P} \int_{0}^{T} \int_{\R^d {\times} \S^{d-1}} 
	(\nabla\cdot\xi_{i,j}) \vartheta^{i} \,d\omega^{i,j}_t dt
	+ C \int_{0}^{T} E[\mu,\chi|\mathscr{A}](t) \,dt
	\end{align*}
	where in the last step, we also exploited~\eqref{eq:motionByMeanCurvatureMultiPhase}.
	Moreover, inserting first the decomposition~\eqref{eq:varifoldGrain}, rewriting
	off-diagonal terms based on the first condition of~\eqref{eq:compNormalSpeeds1} and~\eqref{eq:compNormalSpeeds2},
	and finally estimating diagonal terms by Young's inequality and \eqref{eq:coercivity1MultiPhase},
	we have
	\begin{align*}
	- \sum_{i=1}^{P} \int_0^T \int_{\R^d} V^{i}\vartheta^{i} \,d\omega^{i}_t dt
	&\leq \frac{1}{4} \sum_{i=1}^P \int_{0}^{T} \int_{\R^d} \frac{1}{2}|V^i|^2 \,d\omega^{i,i}_t dt
	-\sum_{i,j=1,\,i\neq j}^{P} \int_0^T \int_{\R^d} V^{i}\vartheta^{i} \,d\omega^{i,j}_t dt
	\\&~~~
	+ C \int_{0}^{T} E[\mu,\chi|\mathscr{A}](t) \,dt.
	\end{align*}
	Adding the previous two displays and then estimating by Young's inequality, 
	the upper bound from~\eqref{eq:coercivity0MultiPhase} as well as the 
	control provided by~\eqref{eq:coercivity4MultiPhase}, we infer
	\begin{align*}
	E_{\mathrm{bulk}}[\chi|\mathscr{A}](T)
	&\leq E_{\mathrm{bulk}}[\chi_0|\mathscr{A}(0)]	
	+ \frac{1}{4} \sum_{i=1}^P \int_{0}^{T} \int_{\R^d} \frac{1}{2}|V^i|^2 \,d\omega^{i,i}_t dt
	\\&~~~
	+ \frac{1}{2} \sum_{i,j=1,\,i\neq j}^{P} \int_{0}^{T}
	\int_{\R^d} \frac{1}{2} |V^{i} {-} (\nabla\cdot\xi_{i,j})|^2 \,d\omega^{i,j}_t dt
	\\&~~~
	+ C \int_0^T E_{\mathrm{bulk}}[\chi|\mathscr{A}](t) + E[\mu,\chi|\mathscr{A}](t) \,dt,
	\end{align*}
	so that adding the previous display to~\eqref{eq:auxMultiPhase15} finally
	yields~\eqref{eq:stabilityBulkErrorMultiPhase} as desired.
	
	\textit{Step 3: Proof of~\emph{\eqref{eq:weakStrong1MultiPhase}--\eqref{eq:weakStrong2MultiPhase}}.}
	This follows from straightforward arguments based on the coercivity properties
	of the relative entropy and the bulk error.
\end{proof}

\section*{Acknowledgments}
This project has received funding from the European Research Council 
(ERC) under the European Union's Horizon 2020 research and innovation 
programme (grant agreement No 948819)
\begin{tabular}{@{}c@{}}\includegraphics[width=8ex]{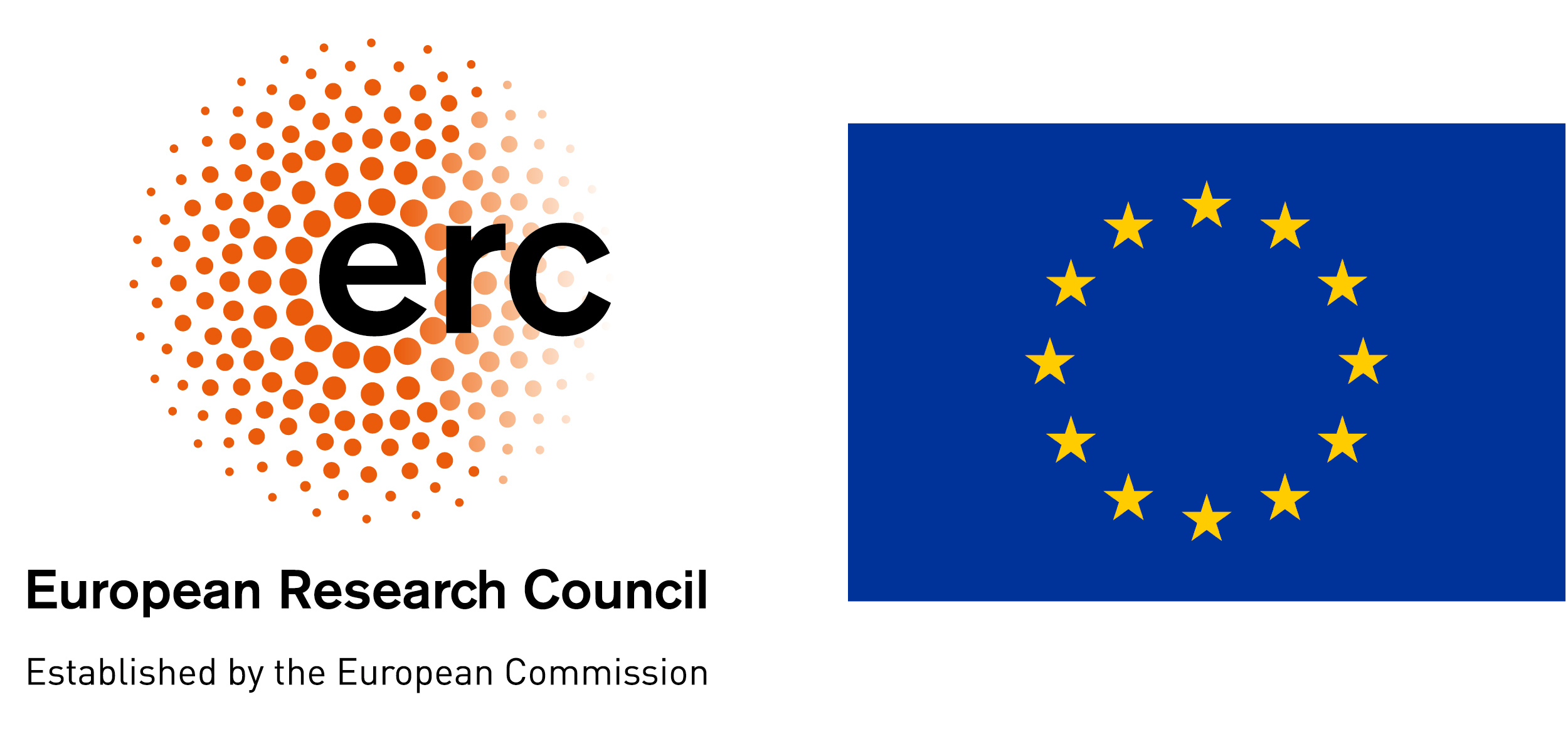}\end{tabular}, and 
from the Deutsche Forschungsgemeinschaft (DFG, German Research Foundation) under Germany's Excellence Strategy -- EXC-2047/1 -- 390685813.
The content of this paper was developed and parts of it were written 
during a visit of the first author to the Hausdorff Center of Mathematics (HCM),
University of Bonn. The hospitality and the support of HCM are gratefully acknowledged. 

\frenchspacing
\bibliographystyle{abbrv}
\bibliography{deGiorgiVarifoldSolutions}
	
  \end{document}